\documentclass[11pt]{article}
\usepackage{color}
\usepackage{relsize}
\usepackage{amsmath,amsthm,amssymb}
\usepackage{amsfonts}
\usepackage{enumerate}
\usepackage{appendix}
\usepackage{ulem}

\usepackage{graphicx}
\usepackage{float}

\newtheorem{thm}{Theorem}[section]
\newtheorem{cor}[thm]{Corollary}
\newtheorem{lem}[thm]{Lemma}
\newtheorem{prop}[thm]{Proposition}

\theoremstyle{definition}
\newtheorem{assu}[thm]{Assumption}
\newtheorem{defn}[thm]{Definition}

\numberwithin{equation}{section}
\def\bR{\mathbb{R}}

\theoremstyle{remark}
\newtheorem{rem}[thm]{Remark}

\begin{document}
	
	\title{Energy Transfer and Radiation in Hamiltonian Nonlinear Klein-Gordon Equations: General Case}
	
	\author{Zhen Lei \footnotemark[1]\ \footnotemark[2]
		\and Jie Liu    \footnotemark[1]\ \footnotemark[3]
		\and Zhaojie Yang  \footnotemark[1]\ \footnotemark[4]
	}
	\renewcommand{\thefootnote}{\fnsymbol{footnote}}
	\footnotetext[1]{School of Mathematical Sciences; LMNS and Shanghai Key Laboratory for Contemporary Applied Mathematics, Fudan University, Shanghai 200433, P. R.China.} \footnotetext[2]{Email: zlei@fudan.edu.cn}
	\footnotetext[3]{Email: jl15817@nyu.edu}
	\footnotetext[4]{Email: yangzj20@fudan.edu.cn}

	\date{\today}
	
	\maketitle
	
	\begin{abstract}
		In this paper, we consider Klein-Gordon equations with cubic nonlinearity in three spatial dimensions, which are Hamiltonian perturbations of the linear one with potential. It is assumed that the corresponding Klein-Gordon operator $B = \sqrt{-\Delta + V(x) + m^2} $ admits an arbitrary number of possibly degenerate eigenvalues in $(0, m)$, and hence the unperturbed linear equation has multiple time-periodic solutions known as bound states. In \cite{SW1999}, Soffer and Weinstein discovered a mechanism called Fermi's Golden Rule for this nonlinear system in the case of one simple but relatively large eigenvalue $\Omega\in (\frac{m}{3}, m)$, by which energy is transferred from discrete to continuum modes and the solution still decays in time. In particular, the exact energy transfer rate is given.  In \cite{LLY22}, we solved the general one simple eigenvalue case. In this paper, we solve this problem in full generality: multiple and simple or degenerate eigenvalues in $(0, m)$. The proof is based on a kind of \textit{pseudo-one-dimensional cancellation structure} in each eigenspace, a \textit{renormalized damping mechanism}, and an \textit{enhanced damping effect}. It also relies on a refined Birkhoff normal form transformation and an accurate generalized Fermi's Golden Rule over those of Bambusi--Cuccagna \cite{BC}. 
		
	\end{abstract}
		\tableofcontents

	\section{Introduction}
	We consider the Klein-Gordon equation with an external potential $V$ and a cubic nonlinearity in $3+1$ dimensions:
	\begin{equation}\label{NLKG}
		\begin{cases}
			\partial^2_t u-\Delta u + m^2u + V(x) u =  \lambda u^3, & t>0, x\in \mathbb{R}^3, \lambda \in \mathbb{R},\\
			u(x,0) = u_0(x), \quad \partial_t u(x,0) = u_1(x).
		\end{cases} 
	\end{equation}
	The potential function $V$ is assumed to be real-valued, smooth and sufficiently fast decaying. Thus, the corresponding Schr\"odinger operator $H = - \Delta +V $ has purely absolutely continuous spectrum $[0,+\infty)$ and a finite number of negative eigenvalues \cite{Sim81}. We denote these eigenvalues to be $0 > \lambda_1 > \lambda_2 > \dots > \lambda_n$, with each eigenvalue $\lambda_{j}$ the corresponding $l_j$ dimensional eigenspace is spanned by an orthonormal basis $\{ \varphi_{j1}, \dots, \varphi_{jl_j}\}$. These eigenfunctions are smooth and fast decaying, see \cite{Sim81}. We take a mass term  $m^2$ such that $-\Delta+V+m^2>0$. Set $B=\sqrt{-\Delta + V + m^2}$ and $\omega_j=\sqrt{m^2+\lambda_j}$, then $B$ has purely absolutely continuous spectrum $[m,+\infty)$ and $n$ distinct eigenvalues $m > \omega_1 > \omega_2 > \dots > \omega_n > 0$. 
	
	In this setting, the linear equation, i.e. \eqref{NLKG} with $\lambda=0$, possesses a family of time-periodic solutions
	\begin{equation*}
		u(t,x)=A\cos(\omega_j t+\theta)\varphi_{jk}(x),
	\end{equation*}
	for $1\le j\le n$, $1\le k\le l_j$ and $A, \theta\in \mathbb{R}$. In quantum mechanics, these periodic solutions are known as bound states. Under a small nonlinear perturbation, an excited state could be unstable with energy shifting to the ground state, free waves and nearby excited states. However, it has been observed that in the meanwhile an anomalously long-lived state, known as metastable state, exists \cite{AS, SW90, SW98, SW1999}. Thus, an interesting question is to investigate the long time behavior of these bound states especially under small Hamiltonian nonlinear perturbations. In particular, it is crucial to give a precise description on the mechanism and the rate that energy transfers from bound states to free waves. Besides, it is worth noting that this type of equations we consider in this paper appear naturally when studying the asymptotic stability of special solutions of nonlinear dispersive and hyperbolic equations, such as solitons, traveling waves, kinks. For instance see \cite{GP, LLS, LLSS, LP2021}.

	The rigorous mathematical analysis of such phenomenons began in the 1990s. In 1993, Sigal \cite{Sigal} first established the instability mechanism of quasi-periodic solutions to nonlinear Schr\"{o}dinger and wave equations in a qualitative manner, in which the Fermi's Golden Rule was first introduced and explored in the field of analysis and partial differential equations. In 1999, Soffer and Weinstein \cite{SW1999} made a significant progress and discovered the Fermi's Golden Rule for the Klein-Gordon equation \eqref{NLKG}. They proved that if the operator $B$ has one simple eigenvalue $\omega$ satisfying $3\omega > m$, then the Fermi's Golden Rule plays an instability role and small global solutions to \eqref{NLKG} decay to zero at an anomalously slow rate as time tends to infinity.  In particular, an accurate energy transfer rate from discrete to continuum modes is given.
	More precisely, the solution $u(t,x)$ has the following expansion as $t\to \pm\infty$:
	\begin{align}\label{eq:expansion-intro}
		u(t,x) = R(t) \cos(\omega t+\theta(t))\varphi(x) +\eta(t,x),
	\end{align}
	where
	\begin{align*}
		R(t) = \mathcal{O}(|t|^{-\frac 14}), \theta(t) = \mathcal{O}(|t|^{\frac 12}), \quad  \|\eta(t,\cdot)\|_{L^8} = \mathcal{O}(|t|^{-\frac 34}).
	\end{align*}
	The lower bound of the decay rate has later been proved using an alternative approach by An--Soffer \cite{AS}. In the recent interesting work \cite{LP2021}, Leger and Pusateri extended the results of \cite{SW1999} to quadratic nonlinearity and obtained the sharp decay rate. We point out that the general case with multiple and simple or degenerate eigenvalue case is left open, see the discussions in \cite{BC,SW1999}. 
	
	In \cite{LLY22}, the authors of this paper solved the problem in the one simple eigenvalue case, i.e. in the weak resonance regime $(2N-1)\omega<m<(2N+1)\omega$ with any given integer $N\ge 1$. The proof relies on the discovery of a generalized Fermi's Golden Rule and certain weighted dispersive estimates. More precisely, it is shown that the expansion \eqref{eq:expansion-intro} of global solution $u(t,x)$ still holds with following quantitative estimates:
	\begin{align*}
		\frac{\frac{1}{C}R(0)}{(1+4N\lambda^{2N}|R(0)|^{4N}\gamma t)^{\frac{1}{4N}}}\le R(t)\le \frac{CR(0)}{(1+4N \lambda^{2N}|R(0)|^{4N}\gamma t)^{\frac{1}{4N}}} ,\\
		\theta(t) = \mathcal{O}(|t|^{1-\frac{1}{2N}}), \quad  \|\eta(t,\cdot)\|_{L^8} = \mathcal{O}(|t|^{-\frac{3}{4N}}).
	\end{align*}
	for some positive constant $C>0$.
	
	In this paper, we solve this problem in full generality: multiple and simple or degenerate eigenvalues in $(0, m)$. The proof is based on a kind of \textit{pseudo-one-dimensional cancellation structure} in each eigenspace, a \textit{renormalized damping mechanism} and an \textit{enhanced damping effect} for the norms of discrete modes. It also relies on a refined Birkhoff normal form transformation and an accurate generalized Fermi's Golden Rule over those of Bambusi--Cuccagna \cite{BC}. See Theorem \ref{thm:main} and next subsection for more details.
	
	These results give a theoretic verification that an excited state could be unstable with energy shifting to the ground state, free waves and nearby excited states under small Hamiltonian perturbations. The underlying mechanism is a kind of generalized Fermi's Golden Rule, see Assumption \ref{assu:FGR}. They also provide a quantitative description on the energy transfer from discrete to continuum modes and on the radiation of continuum modes. As a corollary, there are no small global periodic or quasi-periodic solutions to \eqref{NLKG} under the generalized Fermi's Golden Rule. We mention that the Fermi's Golden Rule has also been used to study the asymptotic stability of solitons of nonlinear Schr\"odinger equations by Tsai--Yau \cite{TY}, Soffer--Weinstein \cite{SW04}, Gang \cite{Gang}, Gang--Sigal \cite{GS}, Gang--Weinstein \cite{GW}; see also the recent advances by Cuccagna--Maeda \cite{CM1}, their survey \cite{CM2} and references therein.

	
	Let us mention that when the operator $B$ has  multiple eigenvalues in general case, the first progress is made by Bambusi and Cuccagana \cite{BC}, where they proved that solutions of \eqref{NLKG} with small initial data in $H^1 \times L^2$ are asymptotically free under a non-degeneracy hypothesis. We note that the energy transfer rate can not be proved for $H^1 \times L^2$ initial data due to the conservation of energy. Indeed, the authors in \cite{BC} conjectured that appropriate decay rates are reachable if restricting initial data to certain class like that of Soffer-Weinstein \cite{SW1999}.
	
	
	We also mention that the phenomenon here is reminiscent of the famous Kolmogorov-Arnold-Moser (KAM) theory, which is concerned with the persistence of periodic and quasi-periodic motion under the Hamiltonian perturbation of a dynamical system. For a finite dimensional integrable Hamiltonian system, this was initiated by Kolmogorov \cite{K} and then extended by Moser \cite{M} and Arnold \cite{A}. Subsequently, many efforts have been focused on generalizing the KAM theory to infinite dimensional Hamiltonian systems (Hamiltonian PDEs), wherein solutions are defined on compact spatial domains, such as \cite{Bo,CW,Ku}. In all these results, appropriate non-resonance conditions imply the persistence of periodic and quasi-periodic solutions. See \cite{L,W} and the references therein for a comprehensive survey. However, the results here (and also in \cite{SW1999, LLY22}, etc.) show that a different scenario happens for Hamiltonian PDEs in the whole space, i.e. resonance conditions lead to the instability of periodic or quasi-periodic solutions.
	
	
	\subsection{Main Result}
	Before presenting the main result of this paper, we first state our assumptions:
	\begin{assu}Assume that the Schr\"odinger operator $H = - \Delta +V $ satisfies the following conditions:\\ 
		(V1) $V$ is real-valued, smooth and decays sufficiently fast;\\
		(V2) $0$ is not a resonance nor an eigenvalue of the operator $-\Delta + V $;\\
		(V3) For each $\omega_j$, there exists an integer $N_j$ such that $\frac{m}{2N_j+1}<\omega_j<\frac{m}{2N_j-1}$, with $1\le N_1\le N_2\le \dots \le N_n$;\\
		(V4) For any $\mu \in \mathbb{Z}^n$  with $|\mu|\le 100N_n$ and $|\mu|$ being odd, $\sum_{j=1}^{n}\mu_{j}\omega_{j}\ne m$;\\
		(V5) For any $\mu \in \mathbb{Z}^n$ with $|\mu|\le 100N_n$ and $|\mu|$ being even, $\sum_{j=1}^{n}\mu_{j}\omega_{j}= 0$ implies $\mu =0$;\\
		(V6) The generalized Fermi's Gordon Rule condition holds, i.e.  Assumption \ref{assu:FGR} holds.
	\end{assu}
	Denote $\mathbf{P}_c$ to be the projection onto the continuous spectral part of $B$, then any solution $u$ of the equation \eqref{NLKG} has the following decomposition:
	\begin{equation}
		u=\sum_{j=1}^{n}\sum_{k=1}^{l_j}q_{jk}\varphi_{jk}+ \mathbf{P}_c u,
	\end{equation}
	where $q_{jk}(t):=\langle u, \varphi_{jk}\rangle$.
	We also define 
	$$\|u\|_X :=\|u\|_{W^{100N_n, 1}}+\|u\|_{W^{100N_n, 2}}.$$
	The main result of this paper is as follows.
	\begin{thm}\label{thm:main}
		Under assumptions (V1)-(V6), there exists a small constant $\epsilon_0>0$ such that for any $0<\epsilon\le \epsilon_0$, if the initial data satisfies
		\begin{align}
			&\|u_0\|_X + \|u_1\|_X=\epsilon, \label{u-initial-mainresult}\\
			&\sum_{k=1}^{l_j}\left(|q_{jk}(0)|+|q'_{jk}(0)|\right)\lesssim \epsilon^{\alpha_j}, \quad \forall~ 1\le j\le n,\label{xi-initial-mainresult}\\
			&\|\mathbf{P}_c u_0\|_X + \|\mathbf{P}_c u_1\|_X \lesssim \epsilon^3,\label{f-initial-mainresult}
		\end{align}
		where $\alpha_j=\min\left\{\frac{N_n}{N_j}, 3\right\}$, then 
		\begin{align}
			&\sum_{k=1}^{l_j}\left(|q_{jk}(t)|+|q'_{jk}(t)|\right)\lesssim \frac{\epsilon^{\alpha_j}}{\left(1+ \epsilon^{4N_n} t\right)^{\frac{\alpha_j}{4N_n}}}, \quad \forall~ 1\le j\le n \label{upperbound-decay-mainresult}\\
			&\|\mathbf{P}_c u\|_{\infty} + \|\mathbf{P}_c \partial_t u\|_{\infty} \lesssim \frac{\epsilon^3}{\left(1+ \epsilon^{4N_n} t\right)^{\frac{3}{4N_n}}},\label{f-decay-mainresult}\\
			&\|\mathbf{P}_d u\|_{X} \approx \frac{\epsilon}{\left(1+ \epsilon^{4N_n} t\right)^{\frac{1}{4N_n}}}\label{xi-decay-mainresult}, \quad \mathbf{P}_d \triangleq 1- \mathbf{P}_c.
		\end{align}
	\end{thm}
	\begin{rem} By \eqref{f-decay-mainresult} and \eqref{xi-decay-mainresult}, we obtain that the sharp decay rate of $u$ is $$\|u\|_{\infty} \approx \frac{\epsilon}{\left(1+ \epsilon^{4N_n} t\right)^{\frac{1}{4N_n}}}.$$
		For $n=1$, this is reduced to the one simple eigenvalue case considered in \cite{LLY22}, which is further reduced to \cite{SW1999} when $N_n=1$.
	\end{rem}
	\begin{rem}
		We indicate that the assumption on initial data
		$$\sum_{k=1}^{l_j}\left(|q_{jk}(0)|+|q'_{jk}(0)|\right)\lesssim \epsilon^{\alpha_j}, \quad \forall~ 1\le j\le n,$$ is to ensure that the discrete mode with slowest decay dominates at the initial time, which is a technical issue for our perturbation argument to derive the lower bound of $u$. Assumptions like \eqref{f-initial-mainresult} is necessary, which leads to resonance-dominated solutions with the decay rates $\langle t\rangle^{-\frac{1}{4N_n}}$ . Otherwise, there may exist dispersion-dominated solutions with faster decay rates as pointed out by Tsai and Yau in \cite{TY}. However, It is worth noting that if we only want to get the upper bound of $u$, then \eqref{u-initial-mainresult} (without \eqref{xi-initial-mainresult} and \eqref{f-initial-mainresult}) is enough. In this case, by slightly modifying the proofs in Section \ref{sec-ODE} and Section \ref{sec-f}, we can still obtain
		\begin{align*}
			&\|\mathbf{P}_d u\|_{X} \lesssim \frac{\epsilon}{\left(1+ \epsilon^{4N_n} t\right)^{\frac{1}{4N_n}}},\\
			&\|\mathbf{P}_c u\|_{\infty} + \|\mathbf{P}_c \partial_t u\|_{\infty} \lesssim \frac{\epsilon^3}{\left(1+ \epsilon^{4N_n} t\right)^{\frac{3}{4N_n}}}+\epsilon \langle t \rangle^{-\frac{3}{2}}.
		\end{align*}
	\end{rem}
	\begin{rem}
		The choice of $\alpha_j$ is due to the normal form transformation. Since we only have $|\xi-\xi'|\lesssim |\xi|^3$(see \eqref{cubic difference}), the best result we can get is $|\xi'_{jk}|\lesssim |\xi'|^3$. Essentially, this is the consequence of cubic nonlinear interactions. See Section \ref{sec-nft} for details and relevant notations.
	\end{rem}
	\begin{rem}
		The choice of norm $X$ can be weakened. Here we take $100N_n$ for the convenience of presentation of our proof. 
	\end{rem}
	\subsection{Difficluties, New Ingredients and the Sketch of the Proof}
	Now we explain the main difficulties of this problem and our ideas and strategies. Without loss of generality, we set $\lambda=1$.
	\subsubsection{Resonance and Normal Form Transformation}\label{subsubsec:nfm}
	As illustrated in \cite{BC}, the energy transfer from discrete to continuum modes in \cite{SW1999}, for the case when there exists only one simple eigenvalue lying close to the continuous spectrum, is due to nonlinear coupling. Technically speaking, this occurs because the equation of the discrete mode has a key coefficient with a positive sign, being called Fermi's Golden Rule, which yields radiation. For the case when the eigenvalues of $B$ are not close to the continuous spectrum, however, the crucial coefficients in the equations of the discrete modes consist of terms of several different forms with indefinite sign, if one follows the non-Hamiltonian scheme of \cite{SW1999}. To overcome this difficulty, Bambusi--Cuccagna \cite{BC} introduced a novel Birkhoff normal form transformation, which preserves the Hamiltonian structure of \eqref{NLKG}. As we remarked in \cite{LLY22}, for the cubic nonlinearity $u^3$, this new normal form transformation can be done more delicately to make the results consistent with the non-Hamiltonian method in \cite{SW1999}. Actually, we found that the order of normal form is increased by two in each step, which has already been observed in the one simple eigenvalue case in \cite{LLY22}. 
	
	In this paper, we further refine the Birkhoff normal form transformation in \cite{BC} and obtain a generalization of the transformation in \cite{LLY22} to the multiple eigenvalues case. To illustrate, we write the nonlinear Klein-Gordon equations \eqref{NLKG} as the following Hamilton equations (see Section \ref{sec-nft} for details)
	\begin{align*}
		\dot{\xi}_{jk}&=-\mathrm{i}  \partial_{\bar{\xi}_{jk}} H, \quad 1\le j\le n, 1\le k\le l_j\\ \dot{f}&=-\mathrm{i} \partial_{\bar{f}} H.
	\end{align*}
	with the corresponding Hamiltonian
	$$
	\begin{aligned}
		H &=H_{L}+H_{P}, \\
		H_{L} &=\sum_{1\le j\le n}\sum_{1\le k \le l_j} \omega_j\left|\xi_{jk}\right|^{2}+\langle\bar{f}, B f\rangle, \\
		H_{P} &=-\frac{1}{4}\int_{\mathbb{R}^{3}} \left(\sum_{1\le j\le n}\sum_{1\le k \le l_j} \frac{\xi_{jk}+\bar{\xi}_{jk}}{\sqrt{2 \omega_j}} \varphi_{jk}(x)+U(x)\right)^4 d x,
	\end{aligned} $$
	where $\partial_{\bar{f}} $ is the gradient with respect to the $L^{2}$ metric, and $U=B^{-\frac{1}{2}}(f+\bar{f}) / \sqrt{2} \equiv \mathbf{P}_c u$. 
	We prove that for any $r\ge 0$ there exists an analytic canonical transformation $\mathcal{T}_{r}$ putting the system in normal form up to order $2r+4$, i.e.
	$$
	H^{(r)}:=H \circ \mathcal{T}_{r}=H_{L}+Z^{(r)}+\mathcal{R}^{(r)},
	$$
	where $Z^{(r)}$ is a polynomial of order $2r+2$ in normal form, i.e. $ Z^{(r)}=Z^{(r)}_{0}+Z^{(r)}_{1} $, $Z^{(r)}_{0}$ is a linear combination of monomials $\xi^{\mu}\bar{\xi}^{\nu}$ with $\omega \cdot (\nu-\mu)=0$, and $Z^{(r)}_{1}$ is a linear combination of monomials of the form
	$$
	\xi^{\mu} \overline{\xi^{\nu}} \int \Phi(x) f(x) dx, \quad   \overline{\xi^{\mu}}\xi^{\nu} \int \Phi(x) \bar{f}(x) dx
	$$
	with indexes satisfying
	$|\mu+\nu|\le 2r+1,
	\omega \cdot (\nu-\mu)>m,  
	$
	and $\Phi \in \mathcal{S}\left(\mathbb{R}^{3}, \mathbb{C}\right)$. $\mathcal{R}^{(r)}$ is considered as an error term, we will explore its structure carefully in Section \ref{sec-nft}. Compared to the normal form transformation in \cite{BC}, the main differences are as follows: (i) we find that the order of normal form actually increases by two in each step, which enables us to derive the accurate decay rates of discrete modes; (ii) we give explicit forms of these coefficients appeared in error terms, whose structure will be crucial in the subsequent error estimates. 
	\subsubsection{\textit{Pseudo-one-dimensional} Structure of Each Eigenspace}
	After applying the normal form transformation for some large $r$ (here we choose $r=100N_n$ for simplicity), we work on the new variables which we still denote them by $(\xi, f)$. Denote
	$$ Z_1(\xi,\mathbf{f}) : = \langle G, f \rangle + \langle \bar{G}, \bar{f} \rangle,$$
	$$ 
	G:=\sum_{(\mu, \nu) \in M} \xi^{\mu} \bar{\xi}^{\nu} \Phi_{\mu \nu}(x), \Phi_{\mu \nu} \in \mathcal{S}\left(\mathbb{R}^{3}, \mathbb{C}\right),
	$$
	where 
	$$
	M=\{(\mu, \nu)\mid |\mu +\nu|=2k+1, 1 \leq k \leq 100 N_n, \omega \cdot (\nu-\mu)>m\} .
	$$
	Then, the corresponding Hamilton equations are
	\begin{align}
		\dot{f} & = -\mathrm{i}(B f+\bar{G})-\mathrm{i}\partial_{\bar{f}}\mathcal{R}, \label{eq:f-intro}\\
		\dot{\xi}_{jk}& = - \mathrm{i}\omega_j\xi_{jk} - \mathrm{i}\partial_{\bar{\xi}_{jk}}Z_0-\mathrm{i}\left\langle \partial_{\bar{\xi}_{jk}}G, f\right\rangle - \mathrm{i}\left\langle \partial_{\bar{\xi}_{jk} }\bar{G}, \bar{f}\right\rangle -\mathrm{i}\partial_{\bar{\xi}_{jk}}\mathcal{R}. \label{eq:xi-intro}
	\end{align}
	Unlike the one eigenvalue case considered in \cite{LLY22, LP2021, SW1999}, we need to deal not only with the interaction between discrete and continuum modes, but also with the coupling between different discrete modes. Rather than considering an ODE of one discrete mode there, we are facing an ODE system of multiple discrete modes. This is much more complicated in its nature. Substituting \eqref{eq:f-intro} into \eqref{eq:xi-intro} and using normal form transformation to eliminate the oscillatory terms, we get the following ODE (here we omit higher order terms and error terms):
	\begin{align}\label{eq:eta-intro}
		\frac{1}{2}\frac{d}{dt}|\eta_{jk}|^2 =& Im\left(\bar{\eta}_{jk}\partial_{\bar{\eta}_{jk}}Z_0\right) \nonumber\\
		&-Im\bigg(\sum_{\substack{(\mu, \nu) \in M \\(\mu', \nu') \in M\\ \omega\cdot(\nu-\mu+\mu'-\nu')=0}}\eta^{\mu+\nu'} \bar{\eta}^{\nu+\mu'}(\nu_{jk}c_{\mu\nu\mu'\nu'}+\mu_{jk}'\bar{c}_{\mu'\nu'\mu\nu})\bigg),
	\end{align}
	where $\eta$ is the new variable after the transformation from $\xi$ and $c_{\mu\nu\mu'\nu'}$ are constants. 
	
	Since the eigenvalues are allowed to be degenerate, the first term  $Im\left(\bar{\eta}_{jk}\partial_{\bar{\eta}_{jk}}Z_0\right)$ in \eqref{eq:eta-intro} does not vanish in general. Due to the Hamiltonian structure, it is easy to derive that 
	\begin{equation*}
		\sum_{1\le j\le n, 1\le k\le l_j} Im\left(\bar{\eta}_{jk}\partial_{\bar{\eta}_{jk}}Z_0\right)=0.
	\end{equation*} 
	However, this is not enough to handle the interactions between the ODE system for $\eta_{jk}$. Our further observation is that
	\begin{equation*}
		\sum_{1\le k\le l_j} Im\left(\bar{\eta}_{jk}\partial_{\bar{\eta}_{jk}}Z_0\right)=0,  ~\forall 1\le j\le n,
	\end{equation*} 
	which is due to the fact that $Z_0$ is real and of norm form, i.e. monomials $	\xi^{\mu} \overline{\xi^{\nu}}$ satisfying $\omega \cdot (\mu-\nu)=0$. This observation implies that the first term $Im\left(\bar{\eta}_{jk}\partial_{\bar{\eta}_{jk}}Z_0\right)$ could only contribute to the internal energy transfer between discrete modes related to the same eigenvalue $\omega_{j}$. Hence, if we collect all $\eta_{jk} (1\le k\le l_j)$ and define
	$$X_j :=\frac{1}{2}\sum_{1\le k\le l_j}|\eta_{jk}|^2,$$ then the interactions of $\eta_{jk}$ in the same mode are eliminated. This way we can treat the problem as if the eigenspace related to every eigenvalue $\omega_j$ is one dimensional, giving a \textit{pseudo-one-dimensional} structure of each eigenspace.
	\subsubsection{Isolation of the Key Resonances and Generalized Fermi's Golden Rule}
	To figure out the damping mechanism of the equation, we need to study the finer structure of nonlinearities in $\eta$. We denote the resonance set
	\begin{equation*}
		\Lambda:=\left\{(\lambda,\rho)~|~ \lambda_j=\sum_k \nu_{jk}, \rho_j=\sum_k \mu_{jk}, (\mu, \nu) \in M\right\},
	\end{equation*}
	and subsets of $M$
	\begin{equation*}	
		M_{\lambda,\rho} :=\left\{(\mu, \nu) \in M~|~ \sum_k\nu_{jk}=\lambda_j , \sum_k \mu_{jk}=\rho_j, \forall 1\le j\le n\right\}.
	\end{equation*}
	Then the equation \eqref{eq:eta-intro} is reduced to
	\begin{equation}
		\frac{d}{dt}X_j=-\sum_{\substack{(\lambda,\rho)\in\Lambda \\ (\lambda',\rho')\in\Lambda \\ \lambda-\rho=\lambda'-\rho'}}\sum_{\substack{(\mu, \nu) \in M_{\lambda,\rho} \\ (\mu', \nu') \in M_{\lambda',\rho'} }}Im\left(\eta^{\mu+\nu'} \bar{\eta}^{\nu+\mu'}(\lambda_j c_{\mu\nu\mu'\nu'}+\rho_{j}'\bar{c}_{\mu'\nu'\mu\nu})\right).\label{intro-Xj}
	\end{equation}	
	To isolate the key resonant terms, it is natural to analyze the structure of the minimal set of the resonance set $\Lambda$:
	\begin{equation*}
		\Lambda^* :=\left\{(\lambda,\rho)\in \Lambda ~|~ \forall (\lambda', \rho') \in \Lambda,  (\lambda', \rho') \le (\lambda,\rho) \Rightarrow (\lambda', \rho') = (\lambda,\rho)\right\}.
	\end{equation*}
	We find that $\Lambda^*$ satisfies the following nice properties:
	\begin{enumerate}[(i)]
		\item If $(\lambda,\rho), (\lambda',\rho')\in\Lambda^*$ satisfy $\lambda-\rho=\lambda'-\rho'$, then we have $(\lambda,\rho)= (\lambda',\rho')$. This enables us to treat $Im\left(\eta^{\mu+\nu'} \bar{\eta}^{\nu+\mu'} c_{\mu\nu\mu'\nu'}\right)$ as an Hermite quadratic form, which leads to the definition of the generalized Fermi's Golden Rule, see Assumption \ref{assu:FGR}.
		\item It turns out that terms with indexes in $\Lambda^*$ dominates the behavior of $X_j$, in the sense that other terms can be treated perturbatively, see Lemma \ref{lem:pert}. We remark here that since for a given $j$ the coefficients $\lambda_{j}$ and $\rho'_{j}$ in \eqref{intro-Xj} may vanish for some $(\lambda,\rho)\in \Lambda^*$, this property is nontrivial.
	\end{enumerate}
	As a consequence, we can derive the key resonant ODE system:
	\begin{align*}
		\frac{d}{dt}X_j=-\sum_{(\lambda,\rho)\in\Lambda^*}(\lambda_j-\rho_{j})c_{\lambda\rho}X^{\lambda+\rho},
	\end{align*}
	where  $c_{\lambda\rho}\approx 1$ is due to our Fermi's Golden Rule assumption.
	\subsubsection{Bad Resonance and \textit{Renormalized Damping Mechanism}}	
	To analyze the long-time dynamical behavior of $X_j$, the main difficulty is the emergence of ``Bad Resonances", i.e. $(\lambda,\rho)\in\Lambda^*$ such that $\lambda_j-\rho_{j}<0.$ We remark here that the emergence of bad resonances only occurs in the multiple eigenvalues case, which may lead to a growth of discrete modes. Due to the Hamiltonian nature, there is a good point that the total effect is damping-like: if we sum over all $j$ with weight $\omega_{j}$, then by the definition of resonance set $\Lambda$, we have
	$$\sum_{1\le j\le n}\omega_{j}(\lambda_j-\rho_{j})>m,$$ 
	which is strictly positive. Unfortunately, this total effect of positive sign is far from enough to characterize the dynamic of $X$, because it only characterize the dynamics of the slowest decay mode $X_n$, which is not sufficient for our analysis. Our strategy is to introduce renormalized variables $\tilde{X}$, due to a new observation that there exists an inherent mechanism to eliminate these bad resonance. Indeed, for any $(\lambda,\rho)\in\Lambda^*$, we have
	\begin{enumerate}[(i)]
		\item $|\rho|=0$ or  $1,$
		\item if $|\rho|=1$, then there exists $j\ge 2$ such that $\rho_{j}=1$ and $\lambda_{k}=0$ for any $k\ge j.$
	\end{enumerate}
	This special structure of $\Lambda^*$ (see Lemma \ref{lem:Lambda}) implies that if bad resonance occurs, then it must belong to $(\lambda,\rho)\in \Lambda$ with $\rho_{j}=1$ and $\lambda_{k}=0$ for any $k\ge j.$ Thus, we can prove that 
	\begin{equation*}
		\sum_{k\le j}\omega_k(\lambda_k-\rho_{k})\approx \sum_{k\le j}	\lambda_k+\rho_{k},
	\end{equation*}
	see Lemma \ref{lem:cancellation}. Using this property, it is natural to introduce a new set of ``Renormalized Variables" $\tilde{X}$:
	\begin{equation*}
		\tilde{X}_j=\sum_{k\le j}\omega_k X_k, ~\forall 1\le j\le n,
	\end{equation*}
	then the equations of $\tilde{X}$ become (after omitting some higher order terms) :
	\begin{equation}\label{eq:tildeX-intro}
		\frac{d}{dt}\tilde{X}_j\approx -\sum_{(\lambda,\rho)\in\Lambda}\sum_{k\le j}(\lambda_k+\rho_{k})\tilde{X}^{\lambda+\rho}.
	\end{equation}
	Then the \textit{renormalized damping mechanism} is present. Moreover, the decay information of $X$ is preserved, see Figure 1 below and Lemma \ref{lem:X-tildeX} for details. Hence, we are able to characterize the dynamics of discrete modes. This is presented in Section \ref{sec:bad-res}.
	
	\begin{figure}[H]
		\centering
		\includegraphics[width=0.7\textwidth]{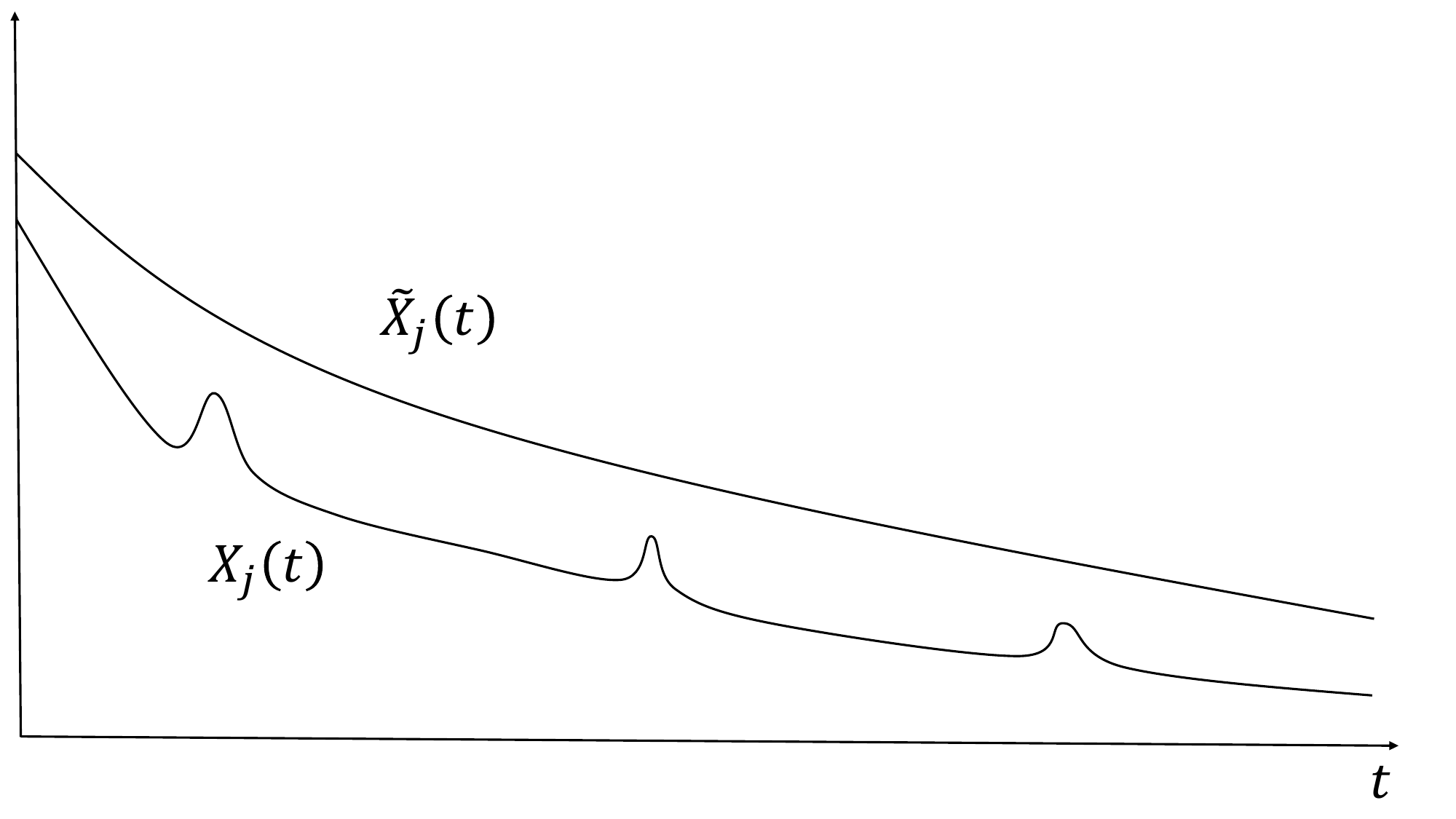}
		\caption{The behavior of $X_j(t)$ and $\tilde{X}_j(t)$. In the diagram it shows that $X_j(t)$ may grow at some time $t$, but it is bounded by its renormalized variable $\tilde{X}_j(t)$,  which is always decaying.}
	\end{figure}	
	
	\subsubsection{Coupling Between Discrete Modes and \textit{Enhanced Damping Effect}}
	The coupling between discrete modes also brings trouble in determining the exact decay rates. In the one simple eigenvalue case \cite{LLY22}, it was proved that if $\frac{m}{2N+1}<\omega<\frac{m}{2N-1}$, then the discrete mode has a decay rate of $\langle t\rangle^{-\frac{1}{4N}}$, or equivalently $X\approx\langle t\rangle^{-\frac{1}{2N}}$. For the multiple eigenvalue case, by \eqref{eq:tildeX-intro} we also have
	\begin{equation*}
		\frac{d}{dt}\tilde{X}_j\lesssim -\tilde{X}_j^{2N_j+1},
	\end{equation*} 
	which implies that $\tilde{X}_j\lesssim \langle t\rangle^{-\frac{1}{2N_j}}$. However, these decay rates of $\{\tilde{X}_j\}$ are not enough to close our estimates on $X$ and $f$. A new observation is  that $\langle t\rangle^{-\frac{1}{2N_j}}$ may not be the optimal decay rate of $\tilde{X}_j.$ To illustrate, let us consider a two-states ODE model as an example:
	\begin{align}
		&\dot{X}_1=-3X_1^3-2X_1^2X_2-X_1X_2^4,\label{eq:X_1}\\
		&\dot{X}_2=-5X_2^5-X_1^2X_2-4X_1X_2^4.\label{eq:X_2}
	\end{align}
	This is a toy model for a two-eigenvalue problem, with $\omega_1$ and $\omega_2$ satisfying the following conditions:
	\begin{align*}
		&\frac{m}{3}< \omega_1<m,\\
		&\frac{m}{5}<\omega_2<\frac{m}{3},\\
		&2\omega_1+\omega_2>m,\\
		&\omega_1+2\omega_2<m.
	\end{align*}
	If there is no coupling terms like $X_1^2X_2$ and $X_1X_2^4$ on the right hand side of \eqref{eq:X_1} and \eqref{eq:X_2}, then we have
	$X_1\approx \langle t\rangle^{-\frac{1}{2}}$ and $X_2\approx \langle t\rangle^{-\frac{1}{4}}$. However, when the coupling between $X_1$ and $X_2$ exists, the situation can be better. Indeed, in \eqref{eq:X_2} we see that the dominant term is still $-5X_2^5$, hence $X_2\approx \langle t\rangle^{-\frac{1}{4}}$ still holds. In \eqref{eq:X_1}, we have $X_1^2X_2\gg X_1^3$, thus $-2X_1^2X_2$ dominates $-3X_1^3$, which implies that $X_1$ decays at least at $\langle t\rangle^{-\frac{3}{4}}$! 
	
	This example shows that the interaction between discrete modes may accelerate the decay of some modes. In general cases, such \textit{enhanced damping effect} could be much more involved, 
	and the decay rate here is very sensitive to the size of each $\omega_j$ and the coefficients of resonant terms. We are not going to pursue how this mechanism would affect every single mode $\tilde{X}_j$, but turn to study the equation of every $\tilde{X}^{\lambda+\rho}$ for $(\lambda,\rho)\in \Lambda$, which is enough for our purpose. Indeed, by \eqref{eq:tildeX-intro}, we have (here we omit higher order terms for simplicity):
	\begin{equation*}
		\frac{d}{dt}\tilde{X}^{\lambda+ \rho}\lesssim -\sum_{1\le j\le n}\sum_{1\le k\le j}\sum_{(\tilde{\lambda},\tilde{\rho})\in\Lambda}\frac{\tilde{X}^{\lambda+ \rho}\tilde{X}^{\tilde{\lambda}+\tilde{\rho}}}{\tilde{X}_{j}}(\lambda_j+ \rho_j)(\tilde{\lambda}_k+\tilde{\rho}_{k}).
	\end{equation*}
	Take $(\tilde{\lambda},\tilde{\rho})=(\lambda,\rho)$ and we get
	\begin{align*}
		\frac{d}{dt}\tilde{X}^{\lambda+ \rho}\lesssim -\sum_{1\le j\le n}\frac{\tilde{X}^{2\lambda+2 \rho}}{\tilde{X}_{j}}(\lambda_j+ \rho_j).
	\end{align*}
	This implies that 
	\begin{equation}\label{eq:X-lambda-rho}
		X^{\lambda+\rho}\lesssim \langle t\rangle^{-\frac{2N_j+1}{2N_j}},
	\end{equation}
	for any $j$ with $\lambda_j+ \rho_j\ne 0$.  We remark that even the decay \eqref{eq:X-lambda-rho} may be not optimal, but it improves the following trivial estimate (for some $(\lambda, \rho)$)
	\begin{equation}\label{eq:X-trivial}
		X^{\lambda+\rho}\lesssim \langle t\rangle^{-\sum_j \frac{\lambda_{j}+\rho_{j}}{2N_j}},
	\end{equation}
	which is the contribution of
	the enhanced damping effect.
	Actually, as in the two-eigenvalue problem, we have by \eqref{eq:X-lambda-rho} (choose $j=1$)
	$$X_1^2 X_2 \lesssim \langle t\rangle^{-\frac{3}{2}},$$
	while using \eqref{eq:X-trivial} we only have
	$$ X_1^2 X_2 \lesssim \langle t\rangle^{-\frac{5}{4}}.$$
	This improvement by the enhanced damping effect is crucial for our perturbation arguments and error estimates, see Theorem \ref{thm:X} for more details.
	\subsubsection{Error Estimates}
	A remaining technical difficulty is to estimate the error terms. As mentioned before, the exact decay rates of discrete modes are unattainable, thus we can not treat all higher order terms perturbatively. To address this issue, we need   to explore the explicit structure of $\partial_{\bar{f}}\mathcal{R}$ and then use an iteration scheme to derive the following expansion of $f$:     
	\begin{equation*}
		f=\sum_{l=0}^{l_0-1} f^{(l)}_M + f^{(l_0)}
	\end{equation*}
	for some large $l_0$, where $$f^{(l)}_M \approx \sum_{(\mu, \nu) \in M^{(l)}} \bar{\xi}^{\mu}\xi^{\nu}\bar{Y}_{\mu\nu}^{(l)}.$$
	This is achievable due to our refined normal form transformation mentioned in Section \ref{subsubsec:nfm}, see also Theorem \ref{thm:nft}. The virtue of this expansion is of  two folds. First, the high order term $f^{(l)}_M (l\ge 1)$ enjoys the same form of resonance, thus its counterpart in the equation of $X_j$ can be controlled by
	$$|X|\sum_{(\lambda,\rho)\in\Lambda^*} X^{\lambda+\rho}(\lambda_{j}+\rho_{j}) + X_j\sum_{(\lambda,\rho)\in\Lambda^*}X^{\lambda+\rho}.$$
	See Lemma \ref{lem:pert}. The first term can be controlled by the leading order term, while the second term can be absorbed by a uniformly bounded variable transformation, see \eqref{eq:hat-X-def}.
	
	Second, we mention that the Strichartz norms of every component of $f$ is bounded, which is crucial to prove the fast decay of $f^{(l_0)}$ for sufficiently large $l_0$. To illustrate this, we write
	\begin{align*}
		B^{-1/2}f^{(l_0)}=&\int_{0}^{t}e^{-\mathrm{i}B (t-s)}\xi^2 B^{-1/2}\left(\Psi B^{-1/2}f^{(l_0)} \right) ds \\
		& -\mathrm{i} \int_{0}^{t}e^{-\mathrm{i}B (t-s)}B^{-1/2}\left(B^{-1/2}f^{(l_0)}+B^{-1/2}\bar{f}^{(l_0)}\right)^3 ds+ \cdots 
	\end{align*}
	where $\xi^2$ denotes some quadratic monomials of $\xi$ and $\bar{\xi}$ and we only list some typical terms for simplicity. The main difficulty of the estimate of $f^{(l_0)}$ comes from the loss of derivatives. For example, if  we choose to estimate the $L^{8}$ norm (or other $L^p$ norms for $p>6$), then by the classical dispersive estimates of the linear Klein-Gordon equation, we have 
	\begin{align*}
		\|B^{-1/2}f^{(l_0)}(t)\|_{W^{k,8}_x}
		\lesssim &\int_0^t \langle t-s \rangle^{-\frac{9}{8}}\bigg(|\xi|^2\|B^{-1/2}f^{(l_0)}(s)\|_{W^{k+1,8}_x}\\
		&+\|B^{-1/2}f^{(l_0)}(t)\|_{W^{k+1,2}_x}^{\frac{4}{3}}\|B^{-1/2}f^{(l_0)}(s)\|_{W^{k,8}_x}^{\frac{5}{3}}\bigg)ds + \cdots 
	\end{align*}
	Thus, we have to use $W^{k+1,8}_x$ norm of $B^{-1/2}f^{(l_0)}$ to control its $W^{k,8}_x$ norm (for other $p>6$ the situation is similar), otherwise there is a loss of decay, see \cite{LLY22, SW1999}. To overcome this difficulty, we use a backward induction argument. By high order Strichartz estimates, we can prove that $\|B^{-1/2}f^{(l_0)}(t)\|_{W^{k,8}_x}$ and $\|B^{-1/2}f^{(l_0)}(t)\|_{W^{k,2}_x}$ are uniformly bounded for large $k$. Fix a large $k$, we then have
	\begin{align*}
		\|B^{-1/2}f^{(l_0)}(t)\|_{W^{k-1,8}_x}
		\lesssim &\int_0^t \langle t-s \rangle^{-\frac{9}{8}}\bigg(|\xi|^2\|B^{-1/2}f^{(l_0)}(s)\|_{W^{k,8}_x}\\
		&+\|B^{-1/2}f^{(l_0)}(t)\|_{W^{k,2}_x}^{\frac{4}{3}}\|B^{-1/2}f^{(l_0)}(s)\|_{W^{k-1,8}_x}^{\frac{5}{3}}\bigg)ds + \cdots\\
		\lesssim &\int_0^t \langle t-s \rangle^{-\frac{9}{8}}\bigg( \langle s\rangle^{-\frac{1}{2N_n}}\|B^{-1/2}f^{(l_0)}(s)\|_{W^{k,8}_x}\\
		&+\|B^{-1/2}f^{(l_0)}(s)\|_{W^{k-1,8}_x}^{\frac{5}{3}}\bigg)ds + \cdots,
	\end{align*}
	which implies that $\|B^{-1/2}f^{(l_0)}(t)\|_{W^{k-1,8}_x}
	\lesssim \langle t\rangle^{-\frac{1}{2N_n}}$. Repeating this process, we can obtain
	\begin{align*}
		\|B^{-1/2}f^{(l_0)}(t)\|_{W^{k-2,8}_x}
		\lesssim &\int_0^t \langle t-s \rangle^{-\frac{9}{8}}\bigg(|\xi|^2\|B^{-1/2}f^{(l_0)}(s)\|_{W^{k-1,8}_x}\\
		&+\|B^{-1/2}f^{(l_0)}(t)\|_{W^{k-1,2}_x}^{\frac{4}{3}}\|B^{-1/2}f^{(l_0)}(s)\|_{W^{k-2,8}_x}^{\frac{5}{3}}\bigg)ds + \cdots\\
		\lesssim &\int_0^t \langle t-s \rangle^{-\frac{9}{8}}\bigg( \langle s\rangle^{-\frac{1}{2N_n}}\|B^{-1/2}f^{(l_0)}(s)\|_{W^{k-1,8}_x}\\
		&+\|B^{-1/2}f^{(l_0)}(s)\|_{W^{k-2,8}_x}^{\frac{5}{3}}\bigg)ds + \cdots,
	\end{align*}
	which implies that $\|B^{-1/2}f^{(l_0)}(t)\|_{W^{k-2,8}_x}
	\lesssim \langle t\rangle^{-\frac{2}{2N_n}}$. In general, we can prove that for $k'\le \frac{9N_n}{4}$,
	$$\|B^{-1/2}f^{(l_0)}(t)\|_{W^{k-k',8}_x}
	\lesssim \langle t\rangle^{-\frac{k'}{2N_n}}.$$ Choose $k$ and $k'$ large, we then get the desired decay estimate.
	\subsection{Structure of the Paper}
	The remaining part of this paper is organized as follows. In Section \ref{sec-Preliminary}, we introduce some useful dispersive estimates and weighted inequalities for linear equations with potential. Besides, we present the global existence theory and energy conservation of the nonlinear Klein-Gordon equation. In Section \ref{sec-nft}, we begin our proof by performing Birkhoff normal form transformation. In Section \ref{sec-iteration}, we use an iteration scheme to derive the expansion of $f$ up to higher orders. Then we isolate the key resonant terms in the dynamical equations of  discrete modes and derive the generalized Fermi's Golden Rule in Section \ref{sec-FGR}. In Section \ref{sec:bad-res}, we introduce a new variable $\tilde{X}$ to cancel the bad resonances. In Section \ref{sec-ODE}, we analyze the ODE and derive the asymptotic behavior of discrete modes. In Section \ref{sec-f}, we give estimates of the continuum variable $f$ and error terms. In Section \ref{sec-final}, we prove the main theorem in this paper. 
	\subsection{Notations}
	Throughout our paper, we adopt the following notations.
	\begin{itemize}
		\item We write $ A\lesssim B$ to mean that $A\le CB$ for some absolute constant $C>0$. We use $A\approx B$ to denote both $ A\lesssim B$ and $ B\lesssim A.$ 
		\item  We denote the vector $\xi=(\xi_{jk})_{1\le j\le n, 1\le k\le l_j}\in \mathbb{C}^{\sum_{j}l_j}$.  For multiple indexes $\mu=(\mu_{jk})_{1\le j\le n, 1\le k\le l_j},\nu= (\nu_{jk})_{1\le j\le n, 1\le k\le l_j}\in \mathbb{N}^{\sum_{j}l_j}$, we denote	$$\xi^{\mu}=\prod_{j,k}\xi_{jk}^{\mu_{jk}}, \quad  \xi^{\mu}\bar{\xi}^{\nu}= \prod_{j,k}\xi_{jk}^{\mu_{jk}}\bar{\xi}_{jk}^{\nu_{jk}}, \quad |\mu|=\sum_{jk}\mu_{jk}.$$
		Denote $\omega=(\omega_{jk})_{1\le j\le n, 1\le k\le l_j},$ where $\omega_{jk}=\omega_j$. Hence, $$\omega \cdot \mu = \sum_{1\le j\le n}\sum_{1\le k \le l_j}\omega_j \mu_{jk}.$$ 
		\item We define the unit vector $e_{jk}\in \mathbb{Z}^{\sum_{j}l_j}$ such that it equals  $1$ for the $jk$-th component and equals $0$ for other components.
		\item We define $\sum_{finite}$ to be a finite sum of terms with the same form, where we omit the summation index for simplicity.
	\end{itemize}

	\section{Preliminaries: Linear Theory and Global Well-posedness}\label{sec-Preliminary}
	
	In this section, we provide some useful lemmas on the linear analysis for the Klein-Gordon equation with potential and the global well-posedness theory of the nonlinear Klein-Gordon equation \eqref{NLKG}.
	
	\subsection{Linear Dispersive Estimates}
	Consider the Cauchy problem for three dimensional linear Klein-Gordon equation with a potential
	\begin{equation}\label{LKW}
		\begin{cases}
			\partial_t^2 u -\Delta u + m^2 u + V(x)u = 0, & \quad t>0, x\in \bR^3,\\
			u(0,x)=u_0,\quad \partial_{t} u(0,x)=u_1.
		\end{cases}
	\end{equation}
	Denote $B^2= -\Delta  + m^2  + V(x)$, then equation \eqref{LKW} can be solved as
	\begin{align*}
		u(t,x)= \cos{Bt}\ u_0 + \frac{\sin Bt}{B}\  u_1.
	\end{align*}

	For $V(x)=0$, i.e. free Klein-Gordon case, the standard $L^p$ dispersive estimates follow from an oscillatory integration method and the conservation of the $L^2$ norm. More precisely, the $L^p$ norm of the solution to $u(t,x)$ satisfies the dispersive decay estimate $\|u(t, \cdot)\|_{L^p} \leq C|t|^{-3(\frac12-\frac{1}{p})}$. 
	
	For $V(x)\neq 0$, if $V(x)$ satisfies some suitable decay and regularity conditions, then the same decay rate of $u$ can be obtained by the $W^{k,p}$-boundedness of the wave operator after being projected on the continuous spectrum of $B$. For instance, see \cite{JSS},\cite{RS},\cite{Yajima}. 
	\begin{lem}[$L^p$ dispersive estimates]\label{Lp-Lp'}
		Assume that $V(x)$ is a real-valued function and satisfies (V1),(V2). Let $1<p\le 2$, $\frac 1p+\frac{1}{p'}=1$, $0\le \theta \le 1$, $l=0,1$,  and $s= (4+\theta)(\frac 12-\frac{1}{p'})$. Then
		\begin{equation*}
			\|\mathrm{e}^{\mathrm{i}Bt}B^{-l}\mathbf{P}_{\mathrm{c}}\psi\|_{l,p'}\lesssim |t|^{-(2+ \theta)(\frac12-\frac{1}{p'})}\|\psi\|_{s,p}, \quad |t|\ge 1,
		\end{equation*}
		and
		\begin{equation*}
			\|\mathrm{e}^{\mathrm{i}Bt}B^{-l}\mathbf{P}_{\mathrm{c}}\psi\|_{l,p'}\lesssim |t|^{-(2- \theta)(\frac12-\frac{1}{p'})}\|\psi\|_{s,p}, \quad 0<|t|\le 1.
		\end{equation*}
	\end{lem}
	
	Moreover, we also use the following Strichartz type estimates, see \cite{BC}.	
	\begin{lem}\label{lem:Str-1}
		Assume (V1)-(V2). Then there exists a constant $C_0$ such that for any two admissible pairs $(p, q)$ and $(a, b)$ we have
		\begin{equation*}
			\left\|e^{-\mathrm{i} t B} P_c u_0\right\|_{L_t^p W_x^{\frac{1}{q}-\frac{1}{p}, q}} \leq C_0\left\|u_0\right\|_{W^{\frac{1}{2},2}}	
		\end{equation*}
		\begin{equation*}
			\left\|\int_{0}^{t} e^{-\mathrm{i}(t-s) B} P_c F(s) d s\right\|_{L_t^p W_x^{\frac{1}{q}-\frac{1}{p}, q}} \leq C_0\|F\|_{L_t^{a^{\prime}} W_x^{\frac{1}{a}-\frac{1}{b}+1, b^{\prime}}} .
		\end{equation*}
		Here an admissible pair $(p,q)$ means 
		\begin{equation*}
			\frac{2}{p}+\frac{3}{q}=\frac{3}{2}, 2\le p\le +\infty, 6\ge q\ge 2 .
		\end{equation*}
	\end{lem}
	\begin{lem}\label{lem:Str-2}
		Assume (V1)-(V2). Then for any $s>1$ there exists a constant $C_0=$ $C_0(s, a)$ such that for any admissible pair $(p, q)$ we have
		$$
		\left\|\int_0^t e^{-\mathrm{i}(t-s) B} P_c F(s) d s\right\|_{L_t^p W_x^{\frac{1}{q}-\frac{1}{p}, q}} \leq C_0\left\|B^{\frac{1}{2}} P_c F\right\|_{L_t^a L_x^{2, s}}
		$$
		where for $p>2$ we can pick any $a \in[1,2]$ while for $p=2$ we pick $a \in[1,2)$.
	\end{lem}

	\subsection{Singular Resolvents and Time Decay}
	The following local decay estimates for singular resolvents $\mathrm{e}^{\mathrm{i} B t}(B-\Lambda+\mathrm{i} 0)^{-l}$,  which was proved in \cite{SW1999}, are also significant. Here, $\Lambda$ is a point in the interior of the continuous spectrum of $B(\Lambda>m)$.
	\begin{lem}[Decay estimates for singular resolvents]\label{lem-singular-resolvents}
		Assume that $V(x)$ is a real-valued function and satisfies (V1)-(V3). Let $\sigma>16/5$. Then for any point $\Lambda>m$ in the continuous spectrum of $B$, we have for $l=1,2:$ 
		\begin{align*}
			&\left\|\langle x\rangle^{-\sigma} \mathrm{e}^{\mathrm{i} B t}(B-\Lambda+\mathrm{i} 0)^{-l} \mathbf{P}_{\mathrm{c}}\langle x\rangle^{-\sigma} \psi\right\|_{2} \lesssim \langle t\rangle^{-\frac{6}{5}}\|\psi\|_{1,2}, \quad t>0,\\	
			&\left\|\langle x\rangle^{-\sigma} \mathrm{e}^{\mathrm{i} B t}(B-\Lambda-\mathrm{i} 0)^{-l} \mathbf{P}_{\mathrm{c}}\langle x\rangle^{-\sigma} \psi\right\|_{2} \lesssim \langle t\rangle^{-\frac{6}{5}}\|\psi\|_{1,2}, \quad t<0.
		\end{align*}
	\end{lem}

	\subsection{Global Well-Posedness and Energy Conservation}
	The global well-posedness of \eqref{NLKG} with small initial data is well-known. 
	\begin{thm}
		Assume $V \in L^{p}$ with $p>3 / 2$. Then, there exists $\varepsilon_{0}>0$ and $C>0$, such that for any $\left\|\left(u_{0}, u_{1}\right)\right\|_{H^{1} \times L^{2}} \leq \epsilon<\varepsilon_{0}$, equation \eqref{NLKG} admits exactly one solution
		$u \in C^{0}\left(\mathbb{R}; H^{1}\right) \cap C^{1}\left(\mathbb{R}; L^{2}\right)$
		such that $(u(0), \partial_{t}u(0))=\left(u_{0}, u_{1}\right)$. Furthermore, the map $\left(u_{0}, u_{1}\right) \mapsto (u(t), \partial_{t}u(t))$ is continuous from the ball $\left\|\left(u_{0}, u_{1}\right)\right\|_{H^{1} \times L^{2}}<\varepsilon_{0}$ to $C^{0}\left(I; H^{1}\right) \times C^{0}\left(I; L^{2}\right)$ for any bounded interval $I$. Moreover, the energy 
		\begin{align*}
			\mathcal{E}[u,\partial_{t} u]\equiv \frac12\int (\partial_t u)^2 + |\nabla u|^2 + m^2 u^2 + V(x)u^2 dx - \frac{\lambda}{4}  \int u^4 dx.
		\end{align*}
		is conserved and
		$$
		\|(u(t), v(t))\|_{H^{1} \times L^{2}} \leq C\left\|\left(u_{0}, v_{0}\right)\right\|_{H^{1} \times L^{2}} .
		$$
	\end{thm}
	We refer to \cite{CH} for details.	
	
	\section{Normal Form Transformation}\label{sec-nft}
	In this section, we present a new Birkhoff normal form transformation which is a refined version of Theorem 4.9 in \cite{BC}. 
	\subsection{Hamiltonian Structure}	
	Recall the 3D nonlinear Klein Gordon equation (NLKG)
	\begin{align}
		u_{t t}-\Delta u+V u+m^{2} u =  u^3, \quad(t, x) \in \mathbb{R} \times \mathbb{R}^{3}, 
	\end{align}
	which is an Hamiltonian perturbation of the linear Klein-Gordon equation with potential. More precisely, in $H^{1}\left(\mathbb{R}^{3}, \mathbb{R}\right) \times L^{2}\left(\mathbb{R}^{3}, \mathbb{R}\right)$ endowed with the standard symplectic form, namely
	$$
	\Omega\left(\left(u_{1}, v_{1}\right) ;\left(u_{2}, v_{2}\right)\right):=\left\langle u_{1}, v_{2}\right\rangle_{L^{2}}-\left\langle u_{2}, v_{1}\right\rangle_{L^{2}},
	$$
	we consider the Hamiltonian
	$$
	\begin{aligned}
		H &=H_{L}+H_{P}, \\
		H_{L} &:=\int_{\mathbb{R}^{3}} \frac{1}{2}\left(v^{2}+|\nabla u|^{2}+V u^{2}+m^{2} u^{2}\right) d x, \\
		H_{P} &:= \int_{\mathbb{R}^{3}} -\frac{1}{4} u^4  d x.
	\end{aligned}
	$$ 
	The corresponding Hamilton equations are $\dot{v}=-\nabla_{u} H, \dot{u}=\nabla_{v} H$, where $\nabla_{u} H$ is the gradient with respect to the $L^{2}$ metric, explicitly defined by
	$$
	\left\langle\nabla_{u} H(u), h\right\rangle=d_{u} H(u) h, \quad \forall h \in H^{1},
	$$
	and $d_{u} H(u)$ is the Frech\'et derivative of $H$ with respect to $u$. It is easy to see that the Hamilton equations are explicitly given by
	\begin{align*}
		\left(\dot{v}=\Delta u-V u-m^{2} u +  u^3, \dot{u}=v\right) \Longleftrightarrow \ddot{u}=\Delta u-V u-m^{2} u +  u^3.
	\end{align*}
	Write $$u=\sum_{j=1}^n \sum_{k=1}^{l_j} q_{jk} \varphi_{jk} + P_c u, \quad v=\sum_{j=1}^n \sum_{k=1}^{l_j} p_{jk}  \varphi_{jk} + P_c v,$$
	with a slightly abuse of notations, from now on we denote
	$$
	B:=P_{c}\left(-\Delta+V+m^{2}\right)^{1 / 2} P_{c},
	$$
	and define the complex variables
	\begin{equation}
		\xi_{jk}:=\frac{q_{jk} \sqrt{\omega_j}+\mathrm{i} \frac{p_{jk}}{\sqrt{\omega_j}}}{\sqrt{2}}, \quad f:=\frac{B^{1 / 2} P_{c} u+\mathrm{i} B^{-1 / 2} P_{c} v}{\sqrt{2}}. \label{def xi}
	\end{equation}
	Then, in terms of these variables the symplectic form can be written as
	$$
	\begin{array}{r}
		\Omega\left(\left(\xi^{(1)}, f^{(1)}\right) ;\left(\xi^{(2)}, f^{(2)}\right)\right)=2 \operatorname{Re}\left[\mathrm{i}\left( \xi^{(1)}\cdot \bar{\xi}^{(2)}+\left\langle f^{(1)}, \bar{f}^{(2)}\right\rangle\right)\right] \\
		=-\mathrm{i} \sum_{j}\left(\bar{\xi}^{(1)}\cdot \xi^{(2)}-\xi^{(1)}\cdot \bar{\xi}^{(2)}\right)-\mathrm{i}\left(\left\langle f^{(2)}, \bar{f}^{(1)}\right\rangle-\left\langle f^{(1)}, \bar{f}^{(2)}\right\rangle\right)
	\end{array}
	$$
	and the Hamilton equations take the form
	$$
	\dot{\xi}_{jk}=-\mathrm{i} \frac{\partial H}{\partial \bar{\xi}_{jk}}, \quad \dot{f}=-\mathrm{i} \nabla_{\bar{f}} H.
	$$
	where
	$$
	\begin{gathered}
		H_{L}=\sum_{j,k} \omega_j\left|\xi_{jk}\right|^{2}+\langle\bar{f}, B f\rangle, \\
		H_{P}(\xi, f)=-\frac{1}{4}\int_{\mathbb{R}^{3}} \left(\sum_{j,k} \frac{\xi_{jk}+\bar{\xi}_{jk}}{\sqrt{2 \omega_j}} \varphi(x)+U(x)\right)^4 d x
	\end{gathered}
	$$
	with $U=B^{-\frac{1}{2}}(f+\bar{f}) / \sqrt{2} \equiv P_{c} u$.
	The Hamiltonian vector field $X_{H}$ of a function is given by
	$$
	X_{H}(\xi, \bar{\xi}, f, \bar{f})=\left(-\mathrm{i} \frac{\partial H}{\partial \bar{\xi}}, \mathrm{i} \frac{\partial H}{\partial \xi},-\mathrm{i} \nabla_{\bar{f}} H, \mathrm{i} \nabla_{f} H\right).
	$$
	The associate Poisson bracket is given by
	$$
	\begin{aligned}
		\{H, K\}:=& \mathrm{i} \left(\frac{\partial H}{\partial \xi}\cdot \frac{\partial K}{\partial \bar{\xi}}-\frac{\partial H}{\partial \bar{\xi}} \cdot \frac{\partial K}{\partial \xi}\right) \\
		&+\mathrm{i}\left\langle\nabla_{f} H, \nabla_{\bar{f}} K\right\rangle-\mathrm{i}\left\langle\nabla_{\bar{f}} H, \nabla_{f} K\right\rangle .
	\end{aligned}
	$$
	Denote $z=(\xi,f), \mathbf{f}=(f,\bar{f}),$ and $\mathcal{P}^{k, s}=\mathbb{C}^{\sum_j l_j} \times P_{c} H^{k, s}\left(\mathbb{R}^{3}, \mathbb{C}\right)$, where $$H^{k, s}\left(\mathbb{R}^{3}, \mathbb{C}\right)=\left\{f: \mathbb{R}^{3} \rightarrow \mathbb{C} \text { s.t. }\|f\|_{H^{s, k}}:=\left\|\langle x\rangle^{s}(-\Delta+1)^{k / 2} f\right\|_{L^{2}}<\infty\right\}.$$		
	
	\subsection{Normal Form Transformation}
	\begin{defn}[Normal Form]\label{def:normalform}
		A polynomial $Z$ is in normal form if
		$$
		Z=Z_{0}+Z_{1}
		$$
		where $Z_{0}$ is a linear combination of monomials $	\xi^{\mu} \overline{\xi^{\nu}}$ such that $	\omega \cdot (\mu-\nu)=0$ , and $Z_{1}$ is a linear combination of monomials of the form
		$$
		\xi^{\mu} \bar{\xi}^{\nu} \int \Phi(x) f(x) d x, \quad \bar{\xi}^{\mu} \xi^{\nu} \int \Phi(x) \bar{f}(x) d x
		$$
		with indexes satisfying
		$$
		\omega \cdot (\nu-\mu)>m, 
		$$
		and $\Phi \in \mathcal{S}\left(\mathbb{R}^{3}, \mathbb{C}\right).$
	\end{defn} 
	
	Now we present the following normal form transformation. 
	\begin{thm}\label{thm:nft}
		For any $\kappa>0,s>0$ and any integer $r\ge 0$, there exist open neighborhoods of the origin $\mathcal{U}_{r, \kappa, s} \subset \mathcal{P}^{1 / 2,0}$, $\mathcal{U}_{r}^{-\kappa,-s} \subset \mathcal{P}^{-\kappa,-s}$, and an analytic canonical transformation $\mathcal{T}_{r}: \mathcal{U}_{r, \kappa, s} \rightarrow \mathcal{P}^{1 / 2,0}$, such that $\mathcal{T}_{r}$ puts the system in normal form up to order $2r+4$. More precisely, we have
		$$
		H^{(r)}:=H \circ \mathcal{T}_{r}=H_{L}+Z^{(r)}+\mathcal{R}^{(r)}
		$$
		where:
		(i) $Z^{(r)}\in \mathbb{R}$ is a polynomial of degree $2r+2$ which is in normal form,\\
		(ii) $I-\mathcal{T}_{r}$ extends into an analytic map from $\mathcal{U}_{r}^{-\kappa,-s}$ to $\mathcal{P}^{\kappa, s}$ and
		\begin{equation}
			\left\|z-\mathcal{T}_{r}(z)\right\|_{\mathcal{P}^{\kappa, s}} \lesssim \|z\|_{\mathcal{P}^{-\kappa,-s}}^{3}.\label{cubic difference}
		\end{equation}
		(iii) we have $\mathcal{R}^{(r)}=\sum_{d=0}^{5} \mathcal{R}_{d}^{(r)}$ with the following properties: \\
		(iii.0) we have
		$$
		\mathcal{R}_{0}^{(r)}=\sum_{|\mu+\nu|=2r+4} a_{\mu \nu}^{(r)}\left(\xi\right) \xi^{\mu} \bar{\xi}^{\nu}  
		$$
		where $a_{\mu \nu}^{(r)}\in C^{\infty}, \overline{a_{\mu \nu}^{(r)}}= a_{\nu\mu}^{(r)}$ satisfying the following expansion with a sufficiently large integer $M^\star>0$:
		\begin{equation}\label{eq:R0-coef}
			a_{\mu \nu}^{(r)}(\xi)=\sum_{k=0}^{M^\star}\sum_{|\alpha+\beta|=2k}a^{(r)}_{\mu \nu \alpha \beta} \xi^{\alpha}\bar{\xi}^{\beta},
		\end{equation}
		(iii.1) we have
		$$
		\mathcal{R}_{1}^{(r)}=\sum_{|\mu+\nu|=2r+3} \xi^{\mu} \bar{\xi}^{\nu} \int_{\mathbb{R}^{3}} \mathbf{\Phi}_{\mu \nu}^{(r)}\left(x, \xi\right) \cdot  \mathbf{f}(x) d x
		$$
		where $
		\mathbf{\Phi}_{\mu \nu}^{(r)}= (\Phi_{\mu \nu}^{(r)},\overline{\Phi_{\nu \mu}^{(r)}})
		$ is smooth and satisfies the following expansion:
		\begin{equation}\label{eq:R1-coef}	
			\Phi_{\mu \nu}^{(r)}(\cdot, \xi)=\sum_{k=0}^{M^\star}\sum_{|\alpha+\beta|=2k} \Phi^{(r)}_{\mu \nu \alpha \beta}(x) \xi^{\alpha}\bar{\xi}^{\beta}
		\end{equation}
		with $\Phi^{(r)}_{\mu \nu \alpha \beta}(x) \in \mathcal{S}\left(\mathbb{R}^{3}, \mathbb{C}\right)$.\\
		(iii.2-4) for $d=2,3,4$, we have
		\begin{equation}\label{eq:Rd}
			\mathcal{R}_{d}^{(r)}= \int_{\mathbb{R}^{3}} F_{d}^{(r)}\left(x, z\right)[U(x)]^{d} d x + \sum_{finite} \prod_{l=1}^{d}\int_{\mathbb{R}^{3}} \mathbf{\Lambda}_{dl}^{(r)}(x,z)\cdot\mathbf{f} d x,
		\end{equation}
		where $F_{4}^{(r)}\equiv 1;$ for $d=2,3,$
		$F_{d}^{(r)}(x, z) \in\mathbb{R}$ is a linear combination of terms of the form
		\begin{align}\label{eq:Rd-coef-F}
			\sum_{k= 0}^{M^\star}\sum_{i=0}^k\sum_{|\mu+\nu|= 4-d + 2k - i}  \xi^\mu\bar{\xi}^\nu\prod_{j=1}^i \int \mathbf{\Phi}^{ij}_{\mu\nu}(x) \cdot \mathbf{f} dx \Psi^{i}_{\mu\nu}(x),
		\end{align}
		and $\mathbf{\Lambda}_{dl}^{(r)}(x, z)= (\Lambda_{dl}^{(r)},\overline{\Lambda_{dl}^{(r)}}) \ (d=2,3,4)$, $\Lambda_{dl}^{(r)}$ is a linear combination of terms of the form
		\begin{align}\label{eq:Rd-coef-lambda}
			\sum_{k= 0}^{M^\star}\sum_{i=0}^k\sum_{|\mu+\nu|= 1 + 2k - i}  \xi^\mu\bar{\xi}^\nu\prod_{j=1}^i \int \tilde{\mathbf{\Phi}}^{ij}_{\mu\nu}(x) \cdot \mathbf{f} dx\tilde{\Psi}^{i}_{\mu\nu}(x),
		\end{align}
		with 
		$\mathbf{\Phi}^{ij}_{\mu\nu}(x),  \tilde{\mathbf{\Phi}}^{ij}_{\mu\nu}(x) \in \left(\mathcal{S}\left(\mathbb{R}^{3}, \mathbb{C}\right)\right)^2,
		\Psi^{i}_{\mu\nu}(x), \tilde{\Psi}^{i}_{\mu\nu}(x) \in \mathcal{S}\left(\mathbb{R}^{3}, \mathbb{C}\right),$.\\
		(iii.5) for $d=5$, we have
		$$\left\|\nabla_{z,\bar{z}}\mathcal{R}_{5}^{(r)}\right\|_{\left(\mathcal{P}^{\kappa, s}\right)^2}\lesssim |\xi|^{M^\star}.$$
		
	\end{thm}
	\begin{rem}
		Here the constant $M^\star$ is chosen to be sufficiently large, for our paper $M^\star=100N_n$ is sufficient.
	\end{rem}
	\begin{proof}
		The proof is similar to Theorem 3.3 in \cite{LLY22}, we shall give a sketch here for self-completeness. We prove Theorem \ref{thm:nft} by induction. With some slightly abuses of notations, we denote $a$ with indexes as a constant, and $\Phi$ or $\Psi$ with indexes as a Schwartz function; they may change line from line, depending on the context. \\
		\textbf{(Step 0)} First, when $r=0$, Theorem \ref{thm:nft} holds with $\mathcal{T}_{0}= I, Z^{(0)}=0, \mathcal{R}^{(0)}=H_P.$\\ 
		\noindent \textbf{(Step $r\to r+1$)} Now we assume that the theorem holds for some $r\ge 0$, we shall prove this for $r+1$. More precisely, define
		\begin{align*}
			&\mathcal{R}_{02}^{(r)}=\mathcal{R}_{0}^{(r)}-\sum_{|\mu+\nu|=2r+4}  a_{\mu \nu}^{(r)}( 0) \xi^{\mu} \bar{\xi}^{\nu}  ,\\
			&\mathcal{R}_{12}^{(r)}=\mathcal{R}_{1}^{(r)}-\sum_{|\mu+\nu|=2r+3} \xi^{\mu} \bar{\xi}^{\nu} \int_{\mathbb{R}^{3}} \mathbf{\Phi}_{\mu \nu}^{(r)}(x,0) \cdot \mathbf{f}(x) d x.
		\end{align*} 
		By \eqref{eq:R0-coef} and \eqref{eq:R1-coef}, we have
		\begin{align*}
			\mathcal{R}_{02}^{(r)} + \mathcal{R}_{12}^{(r)} = &\sum_{|\mu+\nu|=2r+6} a_{\mu \nu}^{(r+1)}(\xi) \xi^{\mu} \bar{\xi}^{\nu}   \\
			&+\sum_{|\mu+\nu|=2r+5} \xi^{\mu} \bar{\xi}^{\nu} \int_{\mathbb{R}^{3}} \mathbf{\Phi}_{\mu \nu}^{(r+1)}(x,\xi) \cdot \mathbf{f}(x) d x
		\end{align*} 
		where the coefficients $a_{\mu \nu}^{(r+1)}(\xi),\mathbf{\Phi}_{\mu \nu}^{(r+1)}(x,\xi)$ satisfy  \eqref{eq:R0-coef}-\eqref{eq:R1-coef} respectively, with $r$ replaced by $r+1$. 
		
		Set 
		\begin{align*}
			K_{r+1}:=& \sum_{|\mu+\nu|=2r+4} a_{\mu \nu}^{(r)}(0) \xi^{\mu} \bar{\xi}^{\nu} +\sum_{|\mu+\nu|=2r+3} \xi^{\mu} \bar{\xi}^{\nu} \int_{\mathbb{R}^{3}} \mathbf{\Phi}_{\mu \nu}^{(r)}(x,0) \cdot \mathbf{f}(x) d ,
		\end{align*}
		which is real valued. Then, we solve the following homological equation
		\begin{align*}
			\left\{H_{L}, \chi_{r+1}\right\}+Z_{r+1}=K_{r+1},
		\end{align*}
		with $Z_{r+1}$ in normal form. Thus, 
		\begin{align*}
			Z_{r+1}=&\sum_{\substack{|\mu+\nu|=2r+4 \\ \omega\cdot(\mu-\nu)=0}  
			} a_{\mu \nu}^{(r)}(0) \xi^{\mu} \bar{\xi}^{\nu} +\sum_{\substack{|\mu+\nu|=2r+3\\ \omega \cdot(\mu-\nu)<-m}} \xi^{\mu} \bar{\xi}^{\nu} \int \Phi_{\mu \nu}^{(r)}(x,0)  f(x) d x\\
			&+\sum_{\substack{|\mu+\nu|=2r+3\\\omega\cdot (\mu-\nu)>m}} \xi^{\mu} \bar{\xi}^{\nu} \int \overline{\Phi_{\nu\mu}^{(r)}}(x,0) \bar{f}(x) d x,
		\end{align*}
		and
		\begin{align*}
			\chi_{r+1}=& i\sum_{\substack{|\mu+\nu|=2r+4 \\
					\omega\cdot (\mu-\nu)\neq 0}}  \frac{a_{\mu \nu}^{(r)}(0)}{\omega\cdot (\mu-\nu)} \xi^{\mu} \bar{\xi}^{\nu} +i\sum_{\substack{|\mu+\nu|=2r+3\\ \omega\cdot (\mu-\nu)>-m}} \xi^{\mu} \bar{\xi}^{\nu} \int R_{\nu\mu}\Phi_{\mu \nu}^{(r)}(x,0)f d x\\
			&-i\sum_{\substack{|\mu+\nu|=2r+3\\\omega\cdot (\mu-\nu)<m}} \xi^{\mu} \bar{\xi}^{\nu} \int R_{\mu\nu}\overline{\Phi_{\nu\mu}^{(r)}}(x,0) \bar{f} d x,
		\end{align*}
		where the operator $$R_{\mu\nu} := (B-\omega\cdot (\mu-\nu))^{-1}.$$
		Let $\phi_{r+1}$ be the Lie transform generated by $\chi_{r+1}$, i.e. $\phi_{r+1}= \phi_{r+1}^t|_{t=1}$, where
		$$\frac{d\phi_{r+1}^t}{dt}= X_{\chi_{r+1}}=(-i\partial_{\bar{\xi}} \chi_{r+1},-i\nabla_{\bar{f}}\chi_{r+1}).$$
		Then, for $z'=(\xi',f')=\phi_{r+1}(\xi,f),$ we have following expansion, see Lemma 3.1 in \cite{LLY22}:
		\begin{align}
			&\xi'_{jk} = \xi_{jk} + \sum_{l=1}^{\infty}\sum_{i=0}^{l}\sum_{|\mu+\nu|=(2r+2)l+1-i}a_{jk, i\mu \nu} \xi^{\mu}\bar{\xi}^{\nu}\sum_{finite}\prod_{\alpha=1}^{i}\int \mathbf{\Phi}_{\alpha\mu\nu}^{jk,i}\cdot \mathbf{f}dx, \label{eq:xi-expansion}\\
			&f'  = f+ \sum_{l=1}^{\infty}\sum_{i=0}^{l-1}\sum_{|\mu+\nu|=(2r+2)l+1-i} \xi^{\mu}\bar{\xi}^{\nu}\sum_{finite}\prod_{\alpha=1}^{i}\int \mathbf{\Lambda}_{\alpha\mu\nu}^{i}\cdot \mathbf{f}dx \Psi_{\mu \nu}^{i}. \label{eq:f-expansion}
		\end{align}
		Recall that $$R^{(r)}=K_{r+1}+ R^{(r)}_{02}+ R^{(r)}_{12}+\sum_{d=2}^5 R^{(r)}_d,$$ and $ K_{r+1}= Z_{r+1} + \{H_L,\chi_{r+1}\}$, we have
		\begin{align}
			\nonumber H^{(r+1)}\triangleq & H^{(r)}\circ\phi_{r+1}= H\circ ( \mathcal{T}_{r}\circ\phi_{r+1})\equiv H\circ  \mathcal{T}_{r+1} \\
			= & H_L \circ \phi_{r+1}+ Z^{(r)} \circ \phi_{r+1} + R^{(r)}\circ \phi_{r+1} \nonumber \\
			= & H_L + Z^{(r)} + Z_{r+1} \nonumber\\
			&+[H_L\circ\phi_{r+1}-(H_L+\{\chi_{r+1},H_L\})] \label{eq:H-error-r}\\
			&+ Z^{(r)} \circ \phi_{r+1} - Z^{(r)} \label{eq:Z-error-r}\\
			&+(K_{r+1}\circ\phi_{r+1}-K_{r+1})\label{eq:K-error-r}\\
			&+ (R^{(r)}_{02}+ R^{(r)}_{12})\circ\phi_{r+1} \label{eq:R02-error}\\
			&+\sum_{d=2}^5 R^{(r)}_d\circ\phi_{r+1}.\label{eq:R-error-r}
		\end{align}
		We define $Z^{(r+1)}= Z^{(r)} + Z_{r+1}$ in the normal form of order $2r+4$. For the term \eqref{eq:H-error-r}, we have
		\begin{align*}
			&H_L\circ\phi_{r+1}-(H_L+\{\chi_{r+1},H_L\})\\
			= &\sum_{k=2}^{\infty} \frac{1}{k!}\underbrace{\{\chi_{r+1},\dots\{\chi_{r+1}}_{k \text{ times}},H_L\}\} \\
			= &  \sum_{k=2}^{\infty}\sum_{i=0}^{k}\sum_{|\mu+\nu|=2(r+1)k+2-i}a_{i\mu \nu} \xi^{\mu}\bar{\xi}^{\nu}\sum_{finite}\prod_{j=1}^{i}\int \mathbf{\Phi}_{\mu\nu}^{ij}\cdot \mathbf{f}dx \\
			= & \sum_{k=2}^{M^*}\left(\sum_{|\mu+\nu|=2(r+1)k+2}a_{0\mu \nu} \xi^{\mu}\bar{\xi}^{\nu} + \sum_{|\mu+\nu|=2(r+1)k+1} a_{1\mu \nu}\xi^{\mu}\bar{\xi}^{\nu} \int  \mathbf{\Phi}^{11}_{\mu\nu}\cdot\mathbf{f}dx \right) \nonumber\\
			& + \sum_{k=2}^{M^*}\sum_{i=2}^{k}\sum_{|\mu+\nu|=2(r+1)k+2-i}a_{i\mu \nu} \xi^{\mu}\bar{\xi}^{\nu}\sum_{finite}\prod_{j=1}^{i}\int \mathbf{\Phi}_{\mu\nu}^{ij}\cdot \mathbf{f}dx + \mathcal{O}(|\xi|^{M^*}).
		\end{align*}
		Thus \eqref{eq:H-error-r} can be absorbed into $R^{(r+1)}_0, R^{(r+1)}_1$, $R^{(r+1)}_2$ and $R^{(r+1)}_5$.
		
		The terms \eqref{eq:Z-error-r},   \eqref{eq:K-error-r} and \eqref{eq:R02-error} can be handled similarly.
		
		For the term \eqref{eq:R-error-r}, denote $f'= f+ G_f, U'= U + G_U$, then for $d=2,3,4$, we have 
		\begin{align*}
			&R^{(r)}_d\circ\phi_{r+1}\\
			=& \int_{\mathbb{R}^{3}} F_{d}^{(r)}\left(x, z'\right)(U+G_U)^{d} d x + \sum_{finite}\prod_{l=1}^{d}\int_{\mathbb{R}^{3}} \mathbf{\Lambda}_{dl}^{(r)}(x,z')\cdot \left(\mathbf{f} + \mathbf{G}_f\right)d x \\
			=& \sum_{j=0}^d \left[\int F_{d}^{(r)}\left(x, z'\right) U^{j} G_U^{d-j} d x + \sum_{finite}\sum_{l_i}\prod_{i=1}^{j}\int_{\mathbb{R}^{3}} \mathbf{\Lambda}_{d,l_i}^{(r)}(x,z')\cdot \mathbf{f} d x \prod_{l\neq l_i}\int_{\mathbb{R}^{3}} \mathbf{\Lambda}_{dl}^{(r)}(x,z')\cdot \mathbf{G}_f d x\right]\\
			:=& \sum_{j=0}^d H_{dj}.
		\end{align*}
		By \eqref{eq:f-expansion}, we have
		\begin{align*}
			G_f = \sum_{k=1}^{\infty}\sum_{i=0}^{k-1}\sum_{|\mu+\nu|=(2r+2)k+1-i} \xi^{\mu}\bar{\xi}^{\nu}\sum_{finite}\prod_{j=1}^{i}\int \mathbf{\Lambda}_{\mu\nu}^{ij}\cdot \mathbf{f}dx \Psi_{\mu, \nu}^{i}, \quad G_U= (G_f+\overline{G_f})/\sqrt{2B}.
		\end{align*}
		Therefore, by \eqref{eq:Rd-coef-F}, \eqref{eq:Rd-coef-lambda} and \eqref{eq:xi-expansion}, we derive
		\begin{align*}
			H_{d0} &= \int F_{d}^{(r)}\left(x, z'\right) G_U^{d} d x + \sum_{finite}\prod_{l=1}^d\int_{\mathbb{R}^{3}} \mathbf{\Lambda}_{dl}^{(r)}(x,z')\cdot \mathbf{G}_f d x \\
			&= \sum_{k=0}^{M^*}\sum_{i=0}^k \sum_{|\mu+\nu|=4+(2r+2)d + 2k -i} a_{i\mu\nu}\xi^{\mu}  \bar{\xi}^{\nu}\sum_{finite}\prod_{j=1}^{i}\int \mathbf{\Phi}_{\mu\nu}^{ij}\cdot \mathbf{f}dx+\mathcal{O}(|\xi|^{M^*}),\\
			H_{d1} &= \int F_{d}^{(r)}\left(x, z'\right) U G_U^{d-1} d x +\sum_{finite} \sum_{i=1}^d \int_{\mathbb{R}^{3}} \mathbf{\Lambda}_{di}^{(r)}(x,z')\cdot \mathbf{f} d x \prod_{l\neq i}\int_{\mathbb{R}^{3}} \mathbf{\Lambda}_{dl}^{(r)}(x,z')\cdot \mathbf{G}_f d x\\
			&= \sum_{k=0}^{M^*}\sum_{i=0}^k \sum_{|\mu+\nu|=3+(2r+2)(d-1) + 2k -i} a_{i\mu\nu}\xi^{\mu}  \bar{\xi}^{\nu}\sum_{finite}\prod_{j=1}^{i+1}\int \mathbf{\Phi}_{\mu\nu}^{ij}\cdot \mathbf{f}dx+\mathcal{O}(|\xi|^{M^*}), 
		\end{align*}
		and for $2\le j \le d$,
		\begin{align*}
			H_{dj} &= \int F_{d}^{(r)}\left(x, z'\right) U^{j} G_U^{d-j} d x +\sum_{finite}\sum_{l_i} \prod_{i=1}^{j}\int_{\mathbb{R}^{3}} \mathbf{\Lambda}_{d,l_i}^{(r)}(x,z')\cdot \mathbf{f} d x \prod_{l\neq l_i}\int_{\mathbb{R}^{3}} \mathbf{\Lambda}_{dl}^{(r)}(x,z')\cdot \mathbf{G}_f d x\\
			&= \int_{\mathbb{R}^{3}} F_{j}^{(r+1)}\left(x, z\right)U^{j} d x + \sum_{finite} \prod_{l=1}^{j}\int_{\mathbb{R}^{3}} \mathbf{\Lambda}_{jl}^{(r+1)}(x,z)\cdot \mathbf{f}d x + \mathcal{O}(|\xi|^{M^*})
		\end{align*}
		where 
		\begin{align*}
			F_{j}^{(r+1)}&= F_{d}^{(r)}\left(x, z'\right) G_U^{d-j}-\mathcal{O}(|\xi|^{M^*})\\
			&= \sum_{k=0}^{M^*}\sum_{i=0}^k \sum_{|\mu+\nu|=4-j + (2r+2)(d-j) + 2k -i} a_{i\mu\nu}\xi^{\mu}  \bar{\xi}^{\nu}\sum_{finite}\prod_{l=1}^{i}\int \mathbf{\Phi}_{\mu\nu}^{il}\cdot \mathbf{f}dx \psi^l_{\mu\nu}(x),
		\end{align*}
		note that $F_{4}^{(r+1)}\equiv 1.$
		Thus $H_{dj}$ can be absorbed into $R^{(r+1)}$. 
		Finally,  it is direct to see $R^{(r)}_5\circ\phi_{r+1}$ can be absorbed into $R^{(r+1)}_5.$
	\end{proof}
	\section{Decoupling of Discrete and Continuum Modes: An Iteration Process}\label{sec-iteration}
	Applying Theorem \ref{thm:nft} for $r=100N_n$, we obtain a new Hamiltonian 
	\begin{align*}
		H = H_L(\xi,\mathbf{f}) + Z_0(\xi) + Z_1(\xi,\mathbf{f}) + \mathcal{R},
	\end{align*}
	where
	$$ Z_1(\xi,\mathbf{f}) : = \langle G, f \rangle + \langle \bar{G}, \bar{f} \rangle, $$
	$$
	G:=\sum_{(\mu, \nu) \in M} \xi^{\mu} \bar{\xi}^{\nu} \Phi_{\mu \nu}(x), \Phi_{\mu \nu} \in \mathcal{S}\left(\mathbb{R}^{3}, \mathbb{C}\right),
	$$
	with 
	$$
	M=\{(\mu, \nu)\mid |\mu +\nu|=2k+1, 0 \leq k \leq  100N_n, \omega \cdot (\nu-\mu)>m\} .
	$$
	Then, the corresponding Hamilton equations are
	\begin{align}
		\dot{f} & = -\mathrm{i}(B f+\bar{G})-\mathrm{i}\partial_{\bar{f}}\mathcal{R}, \label{eq:f}\\
		\dot{\xi}_{jk}& = - \mathrm{i}\omega_{j}\xi_{jk} - \mathrm{i}\partial_{\bar{\xi}_{jk}}Z_0-\mathrm{i}\left\langle \partial_{\bar{\xi}_{jk}}G, f\right\rangle - \mathrm{i}\left\langle \partial_{\bar{\xi}_{jk} }\bar{G}, \bar{f}\right\rangle -\mathrm{i}\partial_{\bar{\xi}_{jk}}\mathcal{R}.\label{eq:xi}
	\end{align}
	
	\subsection{Structure of The Error Term $\partial_{\bar{f}}\mathcal{R}$}	
	In \cite{LLY22}, it is sufficient to treat $\partial_{\bar{f}}\mathcal{R}$ as an error term. However, for the multiple eigenvalues case, to get finer estimates of every $\xi_{jk}$, we have to explore an explicit structure of $\partial_{\bar{f}}\mathcal{R}$ .  By Theorem \ref{thm:nft}, we have
	\begin{prop}\label{prop:R-f}
		$\partial_{\bar{f}}\mathcal{R}=\sum_{d=1}^{5}\partial_{\bar{f}}\mathcal{R}_{d}$ satisfies following properties:\\
		(i)  $\partial_{\bar{f}}\mathcal{R}_{1}$ is a linear combination of terms $\xi^{\mu}\bar{\xi}^{\nu}\Psi$, where $|\mu+\nu|\ge 100N, \Psi$ is smooth.\\
		(ii-iii) For $2\le d\le 3$,
		$\partial_{\bar{f}}\mathcal{R}_{d}$ are linear combinations of terms of following forms		
		$$\xi^{\mu}\bar{\xi}^{\nu}\prod_j^i\int \mathbf{\Phi}_{j} \cdot\mathbf{f}dx \int \Psi U^{d}dx \Psi',\quad  |\mu+\nu|= 5-d+2k-i, 0\le i\le k, $$
		$$\xi^{\mu}\bar{\xi}^{\nu}\prod_j^i\int \mathbf{\Phi}_{j} \cdot\mathbf{f}dx B^{-1/2}\left(\Psi U^{d-1}\right), \quad  |\mu+\nu|= 4-d+2k-i, 0\le i\le k,$$
		$$\xi^{\mu}\bar{\xi}^{\nu}\prod_j^{d-1+i}\int \mathbf{\Phi}_{j} \cdot\mathbf{f}dx \Psi, \quad  |\mu+\nu|= 4-d+2k-i, 0 \le i\le k.$$
		(iv)  $\partial_{\bar{f}}\mathcal{R}_{4}$ is a linear combination of terms of following forms
		$$B^{-1/2}\left(\Psi U^{3}\right)$$ $$\xi^{\mu}\bar{\xi}^{\nu}\prod_j^{3+i}\int \mathbf{\Phi}_{j} \cdot\mathbf{f}dx \Psi, \quad  |\mu+\nu|= 2k-i, 0 \le i\le k.$$
		(v) $\|\partial_{\bar{f}}\mathcal{R}_{5}\|_{H^{s, k}}\lesssim |\xi|^{M^*}$ for any $s, k$.
	\end{prop}
	
	Rearranging these components, we can write
	\begin{align*}
		\partial_{\bar{f}}\mathcal{R}=\sum_{d=0}^{4}Q_d(\xi,\bar{\xi},f,\bar{f}),
	\end{align*}
	where
	$
	\|Q_0\|_{H^{s, k}}\lesssim |\xi|^{100N},
	$
	$Q_1(\xi,\bar{\xi},f,\bar{f})$ is a linear combination of terms in the form of
	$$	\xi^{\mu}\bar{\xi}^{\nu}B^{-1/2}\left(\Psi U\right), \ \xi^{\mu}\bar{\xi}^{\nu}\int \mathbf{\Phi} \cdot\mathbf{f}dx \Psi,  \quad  |\mu+\nu| = 2 + 2k, k\ge 0.$$
	$Q_2(\xi,\bar{\xi},f,\bar{f})$  a linear combination of terms in the form of
	$$\xi^{\mu}\bar{\xi}^{\nu} \int \Psi U^{2}dx \Psi',\quad  |\mu+\nu|= 3+2k, k \ge 0,$$
	\begin{equation*}
		\xi^{\mu}\bar{\xi}^{\nu} B^{-1/2}\left(\Psi U^{2}\right), \quad  |\mu+\nu| = 1 +2k, k \ge 0,
	\end{equation*}
	\begin{equation*}
		\xi^{\mu}\bar{\xi}^{\nu}\prod_{j=1}^2\int \mathbf{\Phi}_{j} \cdot\mathbf{f}dx \Psi, \quad |\mu+\nu| = 1 +2k, k \ge 0.
	\end{equation*}
	$Q_3(\xi,\bar{\xi},f,\bar{f})$  a linear combination of terms in the form of
	$$B^{-1/2}\left(\Psi U^{3}\right),$$ 
	$$\xi^{\mu}\bar{\xi}^{\nu}\prod_{j=1}^{3}\int \mathbf{\Phi}_{j} \cdot\mathbf{f}dx \Psi, \quad  |\mu+\nu|= 2k, k \ge 0,$$
	$$\xi^{\mu}\bar{\xi}^{\nu}\int \mathbf{\Phi} \cdot\mathbf{f}dx B^{-1/2}\left(\Psi U^{2}\right), \quad  |\mu+\nu|= 2+2k, k \ge 0,$$
	$$\xi^{\mu}\bar{\xi}^{\nu}\prod_{j=1}^{2}\int \mathbf{\Phi}_{j} \cdot\mathbf{f}dx B^{-1/2}\left(\Psi U\right), \quad  |\mu+\nu|= 4+2k, k \ge 0,$$	
	$$\xi^{\mu}\bar{\xi}^{\nu} \int \Psi U^{3}dx \Psi',\quad  |\mu+\nu|= 2+2k, k \ge 0,$$
	$$\xi^{\mu}\bar{\xi}^{\nu} \int \mathbf{\Phi} \cdot\mathbf{f}dx\int \Psi U^{2}dx \Psi',\quad  |\mu+\nu|= 4+2k, k \ge 0,$$	
	$Q_4(\xi,\bar{\xi},f,\bar{f})$ is a linear combination of terms that are quartic or higher in $f$:
	$$\xi^{\mu}\bar{\xi}^{\nu}\prod_j^i\int \mathbf{\Phi}_{j} \cdot\mathbf{f}dx \int \Psi U^{d}dx \Psi',\quad  |\mu+\nu|= 5-d+2k-i, 4-d\le i\le k, d=2,3 $$
	$$\xi^{\mu}\bar{\xi}^{\nu}\prod_j^i\int \mathbf{\Phi}_{j} \cdot\mathbf{f}dx B^{-1/2}\left(\Psi U^{d-1}\right), \quad  |\mu+\nu|= 4-d+2k-i, 5-d\le i\le k, d=2,3$$
	$$\xi^{\mu}\bar{\xi}^{\nu}\prod_j^{d-1+i}\int \mathbf{\Phi}_{j} \cdot\mathbf{f}dx \Psi, \quad  |\mu+\nu|= 4-d+2k-i, 5-d \le i\le k, d=2,3$$
	$$\xi^{\mu}\bar{\xi}^{\nu}\prod_j^{3+i}\int \mathbf{\Phi}_{j} \cdot\mathbf{f}dx \Psi, \quad  |\mu+\nu|= 2k-i, 1 \le i\le k.$$
	
	
	\subsection{Iteration Process}\label{subsec-Qd}
	In this subsection, we use an iteration scheme to derive a further decomposition of $f$. The insight is that via each step we can extract the main part of $f^{(l)}$ which  we denote them by $f^{(l)}_M$ and get $f^{(l+1)}$ which is of higher order. Hence the interaction between discrete modes and the continuum mode is further decoupled this way. The virtue of this decomposition is that the Strichartz norms of its every component remain bounded. By \eqref{eq:f} and Duhamel's formula, we have
	\begin{align*}
		f&=e^{-\mathrm{i}B t}f(0)+\int_{0}^{t}e^{-\mathrm{i}B (t-s)}(-\mathrm{i}\bar{G}-\mathrm{i}\partial_{\bar{f}}\mathcal{R})ds \\
		&= -\mathrm{i}\int_{0}^{t}e^{-\mathrm{i}B (t-s)}\bar{G}ds+e^{-\mathrm{i}B t}f(0) -\mathrm{i}\int_{0}^{t}e^{-\mathrm{i}B (t-s)}\partial_{\bar{f}}\mathcal{R}ds\\
		&:=f_M+f^{(1)},
	\end{align*}
	where  $f_M=-\mathrm{i}\int_{0}^{t}e^{-\mathrm{i}B (t-s)}\bar{G}ds$,
	$f^{(1)}=e^{-\mathrm{i}B t}f(0)-\mathrm{i}\int_{0}^{t}e^{-\mathrm{i}B (t-s)}\partial_{\bar{f}}\mathcal{R}ds.$
	Using the structure of $\partial_{\bar{f}}\mathcal{R}$, we obtain
	\begin{align*}
		f^{(1)}&=e^{-\mathrm{i}B t}f(0)-\mathrm{i}\int_{0}^{t}e^{-\mathrm{i}B (t-s)}\partial_{\bar{f}}\mathcal{R}ds\\
		&=e^{-\mathrm{i}B t}f(0)-\mathrm{i}\int_{0}^{t}e^{-\mathrm{i}B (t-s)}\sum_{d=0}^{4}Q_d(\xi,\bar{\xi},f,\bar{f}) ds\\
		&=e^{-\mathrm{i}B t}f(0)-\mathrm{i}\int_{0}^{t}e^{-\mathrm{i}B (t-s)}\sum_{d=0}^{4}Q_d(\xi,f_M+f^{(1)}) ds,
	\end{align*}
	where we denote $Q_d(\xi,f)=Q_d(\xi,\bar{\xi},f,\bar{f}) $ to simplify our notation. 
	
	Expanding each $Q_d$, we can write
	\begin{equation*}
		\sum_{d=0}^{4}Q_d(\xi,f_M+f^{(1)})=\sum_{d=0}^{4}Q^{(1)}_d(f^{(1)}),	
	\end{equation*}
	where $Q^{(1)}_d$ contains all $d$-th order terms of $f^{(1)}$ for $0\le d\le 3$ and all quartic or higher order terms of $f^{(1)}$ for $d=4$. Thus,    
	\begin{align*}
		f^{(1)}
		&=e^{-\mathrm{i}B t}f(0)-\mathrm{i}\int_{0}^{t}e^{-\mathrm{i}B (t-s)}\left(\sum_{d=0}^{4}Q^{(1)}_d(f^{(1)})\right) ds\\
		&=-\mathrm{i}\int_{0}^{t}e^{-\mathrm{i}B (t-s)}Q^{(1)}_0 ds+e^{-\mathrm{i}B t}f(0)-\mathrm{i}\int_{0}^{t}e^{-\mathrm{i}B (t-s)}\sum_{d=1}^{4}Q^{(1)}_d ds\\
		&:=f^{(1)}_M+f^{(2)}
	\end{align*}
	Repeating this process, we have for $l\ge 1$
	\begin{align*}
		f^{(l)}&=-\mathrm{i}\int_{0}^{t}e^{-\mathrm{i}B (t-s)}Q^{(l)}_0 ds+e^{-\mathrm{i}B t}f(0)-\mathrm{i}\int_{0}^{t}e^{-\mathrm{i}B (t-s)}\sum_{d=1}^{4}Q^{(l)}_d(f^{(l)}) ds\\
		&:=f^{(l)}_M+f^{(l+1)},
	\end{align*}
	and we write
	\begin{equation*}
		\sum_{d=1}^{4}Q^{(l)}_d(f^{(l)})=\sum_{d=1 }^{4}Q^{(l)}_d(f^{(l)}_M+f^{(l+1)})=\sum_{d=0}^{4}Q^{(l+1)}_d(f^{(l+1)}),	
	\end{equation*}
	then
	\begin{align*}
		f^{(l+1)}&=-\mathrm{i}\int_{0}^{t}e^{-\mathrm{i}B (t-s)}Q^{(l+1)}_0 ds+e^{-\mathrm{i}B t}f(0)-\mathrm{i}\int_{0}^{t}e^{-\mathrm{i}B (t-s)}\sum_{d=1}^{4}Q^{(l+1)}_d(f^{(l+1)}) ds\\
		&:=f^{(l+1)}_M+f^{(l+2)}.
	\end{align*}
	For the structure of $Q^{(l)}_d$, terms of $Q^{(l)}_d$ are schematically of the form:
	\begin{align}
		&\begin{aligned}
			Q^{(l)}_0: &\xi^\mu\bar{\xi}^\nu B^{-1/2}\left(\Psi B^{-1/2}f^{(l-1)}_M \right), |\mu+\nu|=2,\\
			&\xi^\mu\bar{\xi}^\nu B^{-1/2}\left(\Psi  B^{-1/2}f^{(j)}_M B^{-1/2}f^{(l-1)}_M \right), 0\le j\le l-1, |\mu+\nu|=1,\\
			& B^{-1/2}\left(B^{-1/2}f^{(i)}_M B^{-1/2}f^{(j)}_M B^{-1/2}f^{(l-1)}_M \right), 0\le i, j\le l-1,
		\end{aligned}\\
		&\begin{aligned}
			Q^{(l)}_1: &\xi^\mu\bar{\xi}^\nu B^{-1/2}\left(\Psi B^{-1/2}f^{(l)} \right), |\mu+\nu|=2,\\
			&\xi^\mu\bar{\xi}^\nu B^{-1/2}\left(\Psi  B^{-1/2}f^{(j)}_M B^{-1/2}f^{(l)} \right), 0\le j\le l-1, |\mu+\nu|=1,\\
			& B^{-1/2}\left(B^{-1/2}f^{(i)}_M B^{-1/2}f^{(j)}_M B^{-1/2}f^{(l)} \right), 0\le i, j\le l-1,
		\end{aligned}\\
		&\begin{aligned}
			Q^{(l)}_2: &\xi^\mu\bar{\xi}^\nu B^{-1/2}\left(\Psi  \left(B^{-1/2}f^{(l)}\right)^2 \right), |\mu+\nu|=1,\\
			& B^{-1/2}\left( B^{-1/2}f^{(j)}_M \left(B^{-1/2}f^{(l)}\right)^2 \right), 0\le  j\le l-1,
		\end{aligned}\\
		& Q^{(l)}_3: B^{-1/2}\left(  \left(B^{-1/2}f^{(l)}\right)^3 \right).
	\end{align}
	Terms in $Q^{(l)}_4$ are higher order compared with $Q^{(l)}_d, 0\le d\le 3.$ The remaining terms is similar or  of higher order.
	
	
	\subsection{Decomposition of $f$}
	From above, we obtain  
	\begin{equation*}
		f=\sum_{l=0}^{l_0-1} f^{(l)}_M + f^{(l_0)}
	\end{equation*}
	where
	$$f^{(l)}_M=-\mathrm{i}\int_{0}^{t}e^{-\mathrm{i}B (t-s)}Q^{(l)}_0 ds, l\ge 1$$
	$$f^{(0)}_M=-\mathrm{i}\int_{0}^{t}e^{-\mathrm{i}B (t-s)}\bar{G}ds,$$
	and
	\begin{equation*}
		f^{(l_0)}=e^{-\mathrm{i}B t}f(0)-\mathrm{i}\int_{0}^{t}e^{-\mathrm{i}B (t-s)}\sum_{d=0}^{4}Q^{(l_0)}_d(f^{(l_0)}) ds.
	\end{equation*}
	In fact, $f^{(l)}_M$ can be further decomposed, we have
	\begin{prop}\label{prop:f-lM} The following decomposition holds
		\begin{align*}
			f^{(l)}_M=\sum_{(\mu, \nu) \in M^{(l)}} \bar{\xi}^{\mu}\xi^{\nu}\bar{Y}_{\mu\nu}^{(l)}+f^{(l)}_{M, R}, \quad  l\ge 0,
		\end{align*}
		where\\
		(i) The leading order terms of $\sum_{(\mu, \nu) \in M^{(0)}} \bar{\xi}^{\mu}\xi^{\nu}\bar{Y}_{\mu\nu}^{(0)}$are $$-\sum_{(\mu, \nu) \in M} \bar{\xi}^{\mu}\xi^{\nu}R_{\nu\mu}^{+}\bar{\Phi}_{\mu \nu}, \quad R_{\nu\mu }^{ \pm}:=\lim _{\epsilon \rightarrow 0^{+}}(B-(\nu-\mu) \cdot \omega \mp \mathrm{i} \epsilon)^{-1},$$
		in the sense that remaining terms are at least $O(|\xi|^2)$  order higher. \\
		(ii) For $l\ge 1$, $M^{(l)}$ are higher order index sets of $M$, i.e. for each $(\mu',\nu')\in M^{(l)}$, there is a $(\mu,\nu)\in M$, such that $(\mu',\nu')\ge (\mu,\nu)$ and $|\mu'+\nu'|\ge |\mu+\nu|+2$. \\
		(iii) $\bar{Y}_{\mu\nu}^{(l)}(x)$ belongs to $L^{2, -s}(\mathbb{R}^{3})$.
		\begin{rem}
			$f^{(l)}_{M, R}$ are higher order terms which will be estimated in Section \ref{sec-f}.
		\end{rem}
	\end{prop}
	\begin{proof}By definition, $f^{(0)}_M$ satisfies 
		$$\partial_t f^{(0)}_M +  \mathrm{i}B f^{(0)}_M = -i\bar{G},$$
		where		
		$$
		G:=\sum_{(\mu, \nu) \in M} \xi^{\mu} \bar{\xi}^{\nu} \Phi_{\mu \nu}(x), \Phi_{\mu \nu} \in \mathcal{S}\left(\mathbb{R}^{3}, \mathbb{C}\right).
		$$
		As in \cite{BC}, we	write 
		$$g=f^{(0)}_M+\bar{Y},$$
		where
		$$
		\bar{Y}(\xi,\bar{\xi})=\sum_{(\mu, \nu)\in M} \bar{Y}_{\mu\nu}(x)\bar{\xi}^{\mu} \xi^{\nu}.
		$$
		Set $\bar{Y}_{\mu\nu}(x)=R_{\nu\mu}^{+}\bar{\Phi}_{\mu \nu}$, then $g$ satisfies the following equation:
		\begin{align*}
			\partial_t g +  \mathrm{i}B g =  \sum_{(\mu, \nu)\in M}\bar{\xi}^{\mu} \xi^{\nu}\left[- \mathrm{i}\frac{\nu_{jk}}{\xi_{jk}}\partial_{\bar{\xi}_{jk}}Z_0+\mathrm{i}\frac{\mu_{jk}}{\bar{\xi}_{jk}}\partial_{\xi_{jk}}Z_0\right]R_{\nu\mu}^{+}\bar{\Phi}_{\mu \nu}+g_1+g_R,
		\end{align*}
		where 
		\begin{align*}
			g_1=\sum_{(\mu, \nu)\in M}\bar{\xi}^{\mu} \xi^{\nu}\left[\frac{\nu_{jk}}{\xi_{jk}}\left(-\mathrm{i}\left\langle \partial_{\bar{\xi}_{jk}}G, f\right\rangle - \mathrm{i}\left\langle \partial_{\bar{\xi}_{jk} }\bar{G}, \bar{f}\right\rangle \right)+\frac{\mu_{jk}}{\bar{\xi}_{jk}} C.C.\right]R_{\nu\mu}^{+}\bar{\Phi}_{\mu \nu}.
		\end{align*}
		\begin{align*}
			g_R=\sum_{(\mu, \nu)\in M}\bar{\xi}^{\mu} \xi^{\nu}\left[-\mathrm{i}\frac{\nu_{jk}}{\xi_{jk}}\partial_{\bar{\xi}_{jk}}\mathcal{R}+\mathrm{i}\frac{\mu_{jk}}{\bar{\xi}_{jk}}\partial_{\xi_{jk}}\mathcal{R}\right]R_{\nu\mu}^{+}\bar{\Phi}_{\mu \nu}.
		\end{align*}		
		Notice that the term
		$$\sum_{(\mu, \nu)\in M}\bar{\xi}^{\mu} \xi^{\nu}\left[- \mathrm{i}\frac{\nu_{jk}}{\xi_{jk}}\partial_{\bar{\xi}_{jk}}Z_0+\mathrm{i}\frac{\mu_{jk}}{\bar{\xi}_{jk}}\partial_{\xi_{jk}}Z_0\right]R_{\nu\mu}^{+}\bar{\Phi}_{\mu \nu}$$
		has the same form as $G$, but with a higher order. Thus, we could repeat the above process to extract the discrete parts and obtain a much higher order remainder. Besides, $g_1$ is a higher order term with respect to $f$, hence we can treat it based on the decomposition of $f$. Finally, $g_R$ can be handled in a similar way. 
		
		The decomposition of $f^{(l)}_M$ can be done inductively. Recall that
		\begin{equation*}
			f^{(l)}_M=-\mathrm{i}\int_{0}^{t}e^{-\mathrm{i}B (t-s)}Q^{(l)}_0 ds,
		\end{equation*}
		where $Q^{(l)}_0$ are schematically of the form 
		\begin{align}
			Q^{(l)}_0: &\xi^\mu\bar{\xi}^\nu B^{-1/2}\left(\Psi B^{-1/2}f^{(l-1)}_M \right), |\mu+\nu|=2, \label{Q0 1}\\
			&\xi^\mu\bar{\xi}^\nu B^{-1/2}\left(\Psi  B^{-1/2}f^{(j)}_M B^{-1/2}f^{(l-1)}_M \right), 0\le j\le l-1, |\mu+\nu|=1, \label{Q0 2}\\
			& B^{-1/2}\left(B^{-1/2}f^{(i)}_M B^{-1/2}f^{(j)}_M B^{-1/2}f^{(l-1)}_M \right), 0\le i, j\le l-1, \label{Q0 3}
		\end{align}
		For terms in \eqref{Q0 2} and \eqref{Q0 3}, we put them into $f^{(l)}_{M, R}$. For terms in \eqref{Q0 1},  by induction, we have
		\begin{equation*}
			f^{(l-1)}_M=\sum_{(\mu, \nu) \in M^{(l-1)}} \bar{\xi}^{\mu}\xi^{\nu}\bar{Y}_{\mu\nu}^{(l-1)}+f^{(l-1)}_{M, R}.
		\end{equation*}
		We can substitute it into the equation of $f^{(l)}_M$ and expand $f^{(l)}_M$ as the case $l=0.$
	\end{proof}
	
	Now we have the decomposition of $f$:
	\begin{cor}
		\begin{equation*}
			f=\sum_{(\mu, \nu) \in \tilde{M}} \bar{\xi}^{\mu}\xi^{\nu}\bar{Y}_{\mu\nu}+f_{R},
		\end{equation*}
		where\\
		(i)the leading order terms of $\sum_{(\mu, \nu) \in \tilde{M}} \bar{\xi}^{\mu}\xi^{\nu}\bar{Y}_{\mu\nu}$ are 
		\begin{equation*}
			-\sum_{(\mu, \nu) \in M} \bar{\xi}^{\mu}\xi^{\nu}R_{\nu\mu}^{+}\bar{\Phi}_{\mu \nu}	,
		\end{equation*}
		(ii)the error terms are
		\begin{equation*}
			f_{R}=\sum_{l=0}^{l_0-1}f^{(l)}_{M, R}+	f^{(l_0)}.
		\end{equation*}
	\end{cor}

	\section{Key Resonant Terms and Fermi's Golden Rule}\label{sec-FGR}
	To analyze the dynamics of $\xi_{jk}$, we also study the structure of $\partial_{\bar{\xi}_{jk}}\mathcal{R}$.
	\begin{lem}\label{lem-partial-xi-R}
		The leading order terms of $\partial_{\bar{\xi}_{jk}}\mathcal{R}$ are
		\begin{align*}
			\mathcal{O}(|\xi|^{100N}), \ 
			\xi^{\mu}\bar{\xi}^\nu\prod_{j=1}^2\int \mathbf{\Phi}_{j} \cdot \mathbf{f}dx, \  \prod_{j=1}^3\int \mathbf{\Phi}_{j} \cdot \mathbf{f}dx,  \quad |\mu+\nu|=1,
		\end{align*}
		in the sense that remaining terms are at least $O(|\xi|^2)$  order higher.
	\end{lem}
	Using decomposition of $f$, we have 
	\begin{prop}\label{prop:partial-xi-R}
		\begin{equation*}
			\partial_{\bar{\xi}}\mathcal{R}=\sum_{\substack{(\mu, \nu) \in \tilde{M} \\(\mu', \nu') \in \tilde{M}}}\sum_{|\alpha|+|\beta|\ge 1}c_{\alpha\beta\mu\nu\mu'\nu'}\xi^{\mu+\nu'+\alpha} \bar{\xi}^{\nu+\mu'+\beta}+R_{\xi},
		\end{equation*}
		where 
		\begin{equation*}
			R_{\xi}=\mathcal{O}\left(|\xi|^{100N}+\|f_R\|_{L^{2,-s}}\sum_{(\mu, \nu) \in \tilde{M}} |\bar{\xi}^{\mu}\xi^{\nu}|+\|f_R\|_{L^{2,-s}}^2\right)	.
		\end{equation*}
	\end{prop}
	Substituting the expansion of $f$ and $\partial_{\bar{\xi}_{jk}}\mathcal{R}$ into \eqref{eq:xi}, we have
	\begin{align}
		\dot{\xi}_{jk} = - \mathrm{i}\omega_{j}\xi_{jk} - \mathrm{i}\partial_{\bar{\xi}_{jk}}Z_0+\mathrm{i}\sum_{\substack{(\mu, \nu) \in M \\(\mu', \nu') \in M}}\frac{\xi^{\mu+\nu'} \bar{\xi}^{\nu+\mu'}}{\bar{\xi}_{jk}}(\nu_{jk}c_{\mu\nu\mu'\nu'}+\mu_{jk}'\bar{c}_{\mu'\nu'\mu\nu}) +\sum_{(\mu, \nu) \in \mathcal{A}_{jk}}c_{\mu\nu}\xi^{\mu}\bar{\xi}^{\nu}+\mathcal{R}_{1jk},
	\end{align} 
	where\\
	(i)$c_{\mu\nu\mu'\nu'}=\langle  \Phi_{\mu \nu}, R_{\nu'\mu'}^{+}\bar{\Phi}_{\mu' \nu'}\rangle$.\\
	(ii)$\mathcal{A}_{jk}$ contains indexes of higher order, in the sense that for any $(\mu, \nu) \in \mathcal{A}_{jk}$, there exists $(\mu_1, \nu_1),(\mu_2, \nu_2)\in M$ such that $\mu+\nu+e_{jk}>\mu_1+\nu_1+\mu_2+\nu_2 $.\\
	(iii)
	\begin{equation}
		\mathcal{R}_{1jk}=\mathcal{O}\left(|\xi|^{100N}+\|f_R\|_{L^{2,-s}}\sum_{(\mu, \nu) \in \tilde{M}} \frac{\nu_{jk}|\bar{\xi}^{\mu}\xi^{\nu}|}{|\xi_{jk}|}+\|f_R\|_{L^{2,-s}}^2\right)	.
	\end{equation}
	To proceed, we have to eliminate non-resonant terms using normal form transformation. The idea is similar to \cite{BC}, \cite{SW1999},  the difference is that in this paper we perform it more than one time. For convenience, we will momentarily write $$\mathrm{i}\sum_{\substack{(\mu, \nu) \in M \\(\mu', \nu') \in M}}\frac{\xi^{\mu+\nu'} \bar{\xi}^{\nu+\mu'}}{\bar{\xi}_{jk}}(\nu_{jk}c_{\mu\nu\mu'\nu'}+\mu_{jk}'\bar{c}_{\mu'\nu'\mu\nu})+\sum_{(\mu, \nu) \in \mathcal{A}_{jk}}c_{\mu\nu}\xi^{\mu}\bar{\xi}^{\nu}:=\sum_{(\mu, \nu) \in \mathcal{B}_{jk}}c_{\mu\nu}\xi^{\mu}\bar{\xi}^{\nu},$$
	then 
	\begin{equation*}
		\dot{\xi}_{jk} = - \mathrm{i}\omega_{j}\xi_{jk} - \mathrm{i}\partial_{\bar{\xi}_{jk}}Z_0+\sum_{(\mu, \nu) \in \mathcal{B}_{jk}}c_{\mu\nu}\xi^{\mu}\bar{\xi}^{\nu} +\mathcal{R}_{1jk}	.
	\end{equation*}	
	For any $(\mu, \nu)$, note that
	\begin{align*}
		\left(\frac{d}{dt}+\mathrm{i}\omega_{j}\right)\xi^{\mu}\bar{\xi}^{\nu}=&\mathrm{i}\omega\cdot(\nu-\mu+e_{jk})\xi^{\mu}\bar{\xi}^{\nu}\\
		&+\sum_{j'k'}\xi^{\mu}\bar{\xi}^{\nu}\left[\frac{\mu_{j'k'}}{\xi_{j'k'}}\left(- \mathrm{i}\partial_{\bar{\xi}_{j'k'}}Z_0+\sum_{(\mu', \nu') \in \hat{M}_{j'k'}}c_{\mu'\nu'}\xi^{\mu'}\bar{\xi}^{\nu'} +\mathcal{R}_{1j'k'}\right)+\frac{\nu_{j'k'}}{\bar{\xi}_{j'k'}}C.C.\right],
	\end{align*}
	let 
	\begin{equation}
		\xi^{(1)}_{jk}=\xi_{jk}+\Delta^{(1)}_{jk},
	\end{equation}
	where 
	\begin{equation}
		\Delta^{(1)}_{jk}=-\sum_{\substack{(\mu, \nu)\in \mathcal{B}_{jk} \\ \omega\cdot(\nu-\mu+e_{jk}) \ne 0}}\frac{c_{\mu\nu}}{\omega\cdot(\nu-\mu+e_{jk})}\xi^{\mu}\bar{\xi}^{\nu},
	\end{equation}
	then the equations of $\xi^{(1)}$ satisfies
	\begin{align}
		\dot{\xi}^{(1)}_{jk} = - \mathrm{i}\omega_{j}\xi^{(1)}_{jk} - \mathrm{i}\partial_{\bar{\xi}^{(1)}_{jk}}Z_0+\sum_{\substack{(\mu, \nu)\in \mathcal{B}_{jk} \\ \omega\cdot(\nu-\mu+e_{jk}) = 0}}c_{\mu\nu}{\xi^{(1)}}^{\mu} {\bar{\xi^{(1)}}}^{\nu}
		+\sum_{(\mu, \nu) \in \mathcal{B}_{jk}^{(1)}}c_{\mu\nu}{\xi^{(1)}}^{\mu} {\bar{\xi^{(1)}}}^{\nu}
		+\mathcal{R}^{(1)}_{1jk},
	\end{align}
	where $\mathcal{B}_{jk}^{(1)}$ are higher order terms of $\mathcal{B}_{jk}$, and
	\begin{align*}
		\mathcal{R}^{(1)}_{1jk}=\mathcal{O}\left(\mathcal{R}_{1jk}+|\xi|^{100N}\right).
	\end{align*}
	Repeating this step for $l$ times and using the iteration relation
	\begin{equation*}
		\xi^{(i)}_{jk}=\xi^{(i-1)}_{jk}+\Delta^{{(i)}}_{jk},
	\end{equation*}
	where
	\begin{equation*}
		\Delta^{(i)}_{jk}=-\sum_{\substack{(\mu, \nu)\in \mathcal{B}_{jk}^{(i-1)} \\ \omega\cdot(\nu-\mu+e_{jk}) \ne 0}}\frac{c_{\mu\nu}}{\omega\cdot(\nu-\mu+e_{jk})}{\xi^{(i-1)}}^{\mu}{\overline{\xi^{(i-1)}}}^{\nu},
	\end{equation*}
	we have
	\begin{align*}
		\dot{\xi}^{(l)}_{jk} = - \mathrm{i}\omega_{j}\xi^{(l)}_{jk} - \mathrm{i}\partial_{\bar{\xi}^{(l)}_{jk}}Z_0+\sum_{i=0}^{l-1}\sum_{\substack{(\mu, \nu)\in \mathcal{B}_{jk}^{(i)} \\ \omega\cdot(\nu-\mu+e_{jk}) = 0}}c_{\mu\nu}{\xi^{(l)}}^{\mu} {\bar{\xi^{(l)}}}^{\nu}
		+\sum_{(\mu, \nu) \in \mathcal{B}_{jk}^{(l)}}c_{\mu\nu}{\xi^{(l)}}^{\mu}\bar{\xi^{(l)}}^{\nu} 
		+\mathcal{R}^{(l)}_{1jk},
	\end{align*}
	where $\mathcal{B}_{jk}^{(i)}$ are higher order terms of $\mathcal{B}_{jk}$, with $|\mu+\nu|\ge 3+2i$ for  $(\mu,\nu)\in \mathcal{B}_{jk}^{(i)}$ and
	\begin{align*}
		\mathcal{R}^{(l)}_{1jk}=\mathcal{O}\left(\mathcal{R}_{1jk}+|\xi|^{100N}\right).
	\end{align*}
	Choosing $l$ sufficiently large and denoting
	\begin{equation*}
		\eta := {\xi}^{(l)}, ~~~ \mathcal{R}_{2jk} := \sum_{(\mu, \nu) \in \mathcal{B}_{jk}^{(l)}}c_{\mu\nu}\xi^{\mu}\bar{\xi}^{\nu} + \mathcal{R}^{(l)}_{1jk},
	\end{equation*}
	we get
	\begin{equation}
		\dot\eta_{jk} = - \mathrm{i}\omega_{j}\eta_{jk} - \mathrm{i}\partial_{\bar{\eta}_{jk}}Z_0+\sum_{i=0}^{l-1}\sum_{\substack{(\mu, \nu)\in \mathcal{B}_{jk}^{(i)} \\ \omega\cdot(\nu-\mu+e_{jk}) = 0}}c_{\mu\nu}{\eta}^{\mu} {\bar{\eta}}^{\nu}
		+\mathcal{R}_{2jk},
	\end{equation}
	with 
	\begin{equation*}
		\mathcal{R}_{2jk}=\mathcal{O}\left(\mathcal{R}_{1jk}+|\xi|^{100N}\right).
	\end{equation*}
	Using the fact that $\mathcal{B}_{jk}^{(i)}$ has higher order than $\mathcal{B}_{jk}^{(i-1)}$, and Proposition \ref{prop:partial-xi-R}, we restore the equation as
	\begin{align}\label{eq:eta}
		\begin{aligned}
			\dot\eta_{jk} =& - \mathrm{i}\omega_{j}\eta_{jk} - \mathrm{i}\partial_{\bar{\eta}_{jk}}Z_0 + \mathrm{i}\sum_{\substack{(\mu, \nu) \in M \\(\mu', \nu') \in M\\ \omega\cdot(\nu-\mu+\mu'-\nu')=0}}\frac{\eta^{\mu+\nu'} \bar{\eta}^{\nu+\mu'}}{\bar{\eta}_{jk}}(\nu_{jk}c_{\mu\nu\mu'\nu'}+\mu_{jk}'\bar{c}_{\mu'\nu'\mu\nu})\\ &+\sum_{\substack{(\mu, \nu)\in \mathcal{C}_{jk} \\ \omega\cdot(\nu-\mu+e_{jk}) = 0}}c_{\mu\nu}\eta^{\mu}\bar{\eta}^{\nu}+\mathcal{R}_{2jk},
		\end{aligned}   		
	\end{align}
	where
	\begin{equation*}
		\sum_{\substack{(\mu, \nu)\in \mathcal{C}_{jk} \\ \omega\cdot(\nu-\mu+e_{jk}) = 0}}c_{\mu\nu}\eta^{\mu}\bar{\eta}^{\nu}=
		\sum_{\substack{(\mu, \nu) \in \mathcal{A}_{jk} \\ \omega\cdot(\nu-\mu+e_{jk}) = 0} }c_{\mu\nu}{\eta}^{\mu} {\bar{\eta}}^{\nu}+\sum_{i=1}^{l-1}\sum_{\substack{(\mu, \nu)\in \mathcal{B}_{jk}^{(i)} \\ \omega\cdot(\nu-\mu+e_{jk}) = 0}}c_{\mu\nu}{\eta}^{\mu} {\bar{\eta}}^{\nu}
	\end{equation*}
	are high order terms. Our next observation is that
	\begin{lem}
		For any $1\le j\le n$, 
		\begin{equation*}
			\left\{\sum_{1\le k\le l_j} |\eta_{jk}|^2, Z_0\right\}=0.
		\end{equation*}
	\end{lem} 
	\begin{proof}
		Write 
		\begin{equation*}
			Z_0=\sum_{\omega\cdot(\nu-\mu)=0}c_{\mu\nu}{\eta}^{\mu} {\bar{\eta}}^{\nu},
		\end{equation*}
		since $Z_0$ is real, we have $c_{\mu\nu}=\bar{c}_{\nu\mu}$. In addition, by Assumption (V5), $\omega\cdot(\nu-\mu)=0$ implies that for any $j$, $\sum_{k}\nu_{jk}=\sum_{k}\mu_{jk}$. Hence 
		\begin{align*}
			\left\{\sum_k |\eta_{jk}|^2, Z_0\right\}
			&=\mathrm{i}\sum_k \left( \bar{\eta}_{jk}\partial_{\bar{\eta}_{jk}}Z_0-\eta_{jk}\partial_{\eta_{jk}}Z_0 \right)\\
			&=\mathrm{i}\sum_{\omega\cdot(\nu-\mu)=0}\sum_k \left(c_{\mu\nu}{\eta}^{\mu} {\bar{\eta}}^{\nu}\nu_{jk}-\bar{c}_{\mu\nu}{\bar{\eta}}^{\mu} {\eta}^{\nu}\nu_{jk}\right)\\
			&=\mathrm{i}\sum_{\omega\cdot(\nu-\mu)=0}\sum_k \left(c_{\mu\nu}{\eta}^{\mu} {\bar{\eta}}^{\nu}\nu_{jk}-\bar{c}_{\nu\mu}{\bar{\eta}}^{\nu} {\eta}^{\mu}\mu_{jk}\right)\\
			&=\mathrm{i}\sum_{\omega\cdot(\nu-\mu)=0}c_{\mu\nu}{\eta}^{\mu}{\bar{\eta}}^{\nu}\left(\sum_{k}\nu_{jk}-\sum_{k}\mu_{jk}\right)\\
			&=0
		\end{align*}
	\end{proof}
	This observation enable us to treat the ODE as if each $\omega_j$ is simple. Multiplying the equation \eqref{eq:eta} by $\bar{\eta}_{jk}$, taking the real part and sum over $k$, we get
	\begin{align}\label{eq:eta-square}
		\begin{aligned}
			\frac{1}{2}\frac{d}{dt}\sum_{1\le k\le l_j}|\eta_{jk}|^2=
			&- Im\left(\sum_k \sum_{\substack{(\mu, \nu) \in M \\(\mu', \nu') \in M\\ \omega\cdot(\nu-\mu+\mu'-\nu')=0}}\eta^{\mu+\nu'} \bar{\eta}^{\nu+\mu'}(\nu_{jk}c_{\mu\nu\mu'\nu'}+\mu_{jk}'\bar{c}_{\mu'\nu'\mu\nu})\right)\\
			&+ Re\left(\sum_k\sum_{\substack{(\mu, \nu)\in \mathcal{C}_{jk} \\ \omega\cdot(\nu-\mu+e_{jk}) = 0}}c_{\mu\nu}\eta^{\mu}\bar{\eta}^{\nu+e_{jk}}\right)+ Re\left(\sum_k\bar{\eta}_{jk}\mathcal{R}_{2jk}\right).	    
		\end{aligned}
	\end{align}
	To further simplify this equation, we define
	\begin{equation*}
		\Lambda:=\left\{(\lambda,\rho)~|~ \lambda_j=\sum_k \nu_{jk}, \rho_j=\sum_k \mu_{jk}, (\mu, \nu) \in M\right\},
	\end{equation*}
	and its minimal set
	\begin{equation*}
		\Lambda^* :=\left\{(\lambda,\rho)\in \Lambda ~|~ \forall (\lambda', \rho') \in \Lambda,  (\lambda', \rho') \le (\lambda,\nu) \Rightarrow (\lambda', \rho') = (\lambda,\rho)\right\}.
	\end{equation*}
	\begin{rem}
		The equivalent definition of $\Lambda$ is 
		\begin{align*}
			\Lambda=\left\{(\lambda,\rho)\in \mathbb{N}^n \times \mathbb{N}^n ~\bigg|~ |\lambda +\rho|=2k+1, 0 \leq k \leq  100N_n, \sum_{1\le j\le n}\omega_j(\lambda_j-\rho_j)>m\right\}.
		\end{align*}
	\end{rem}
	Denote
	\begin{equation*}
		M_{\lambda,\rho}=\left\{(\mu, \nu) \in M~|~ \sum_k\nu_{jk}=\lambda_j , \sum_k \mu_{jk}=\rho_j, \forall 1\le j\le n\right\}.
	\end{equation*}
	Then $\omega\cdot(\nu-\mu+\mu'-\nu')=0$ implies for any $j$, $\sum_k\nu_j-\sum_k \mu_j=\sum_k\nu'_j-\sum_k \mu'_j$.  Hence,
	\begin{align*}
		&\sum_k \sum_{\substack{(\mu, \nu) \in M \\(\mu', \nu') \in M\\ \omega\cdot(\nu-\mu+\mu'-\nu')=0}}\eta^{\mu+\nu'} \bar{\eta}^{\nu+\mu'}(\nu_{jk}c_{\mu\nu\mu'\nu'}+\mu_{jk}'\bar{c}_{\mu'\nu'\mu\nu})\\
		=&\sum_{\substack{(\lambda,\rho)\in\Lambda \\ (\lambda',\rho')\in\Lambda \\ \lambda-\rho=\lambda'-\rho'}}\sum_{\substack{(\mu, \nu) \in M_{\lambda,\rho} \\ (\mu', \nu') \in M_{\lambda',\rho'} }}\eta^{\mu+\nu'} \bar{\eta}^{\nu+\mu'}(\lambda_j c_{\mu\nu\mu'\nu'}+\rho_{j}'\bar{c}_{\mu'\nu'\mu\nu}).
	\end{align*}
	Using Plemelji formula
	$$\frac{1}{x \mp i 0}= \operatorname{P.V} \frac{1}{x} \pm \mathrm{i} \pi \delta(x),$$ 
	we have
	\begin{align*}
		c_{\mu\nu\mu'\nu'}&=\langle  \Phi_{\mu \nu}, R_{\nu'\mu'}^{+}\bar{\Phi}_{\mu' \nu'}\rangle\\
		&=\left\langle  \Phi_{\mu \nu}, (B-\omega\cdot(\nu'-\mu')-\mathrm{i} 0)^{-1}\bar{\Phi}_{\mu' \nu'}\right\rangle\\
		&=\left\langle  \Phi_{\mu \nu}, (B-\omega\cdot(\lambda-\rho)-\mathrm{i} 0)^{-1}\bar{\Phi}_{\mu' \nu'}\right\rangle\\
		&=\bigg\langle  \Phi_{\mu \nu}, \operatorname{P.V} \frac{1}{B-\omega\cdot(\lambda-\rho)}\bar{\Phi}_{\mu' \nu'}\bigg\rangle + \mathrm{i} \pi \bigg\langle  \Phi_{\mu \nu},  \delta(B-\omega\cdot(\lambda-\rho))\bar{\Phi}_{\mu' \nu'}\bigg\rangle\\
		&:=a_{\mu\nu\mu'\nu'}+\mathrm{i}b_{\mu\nu\mu'\nu'}
	\end{align*}
	Define the matrix
	\begin{equation*}
		T_{\lambda,\rho}=\{c_{\mu\nu\mu'\nu'}\}_{(\mu, \nu), (\mu', \nu') \in M_{\lambda,\rho}},
	\end{equation*}
	then
	\begin{equation*}
		T_{\lambda,\rho}=T_{Re,\lambda,\rho}+\mathrm{i} T_{Im,\lambda,\rho},
	\end{equation*}
	with $T_{Re,\lambda,\rho}=\{a_{\mu\nu\mu'\nu'}\}_{(\mu, \nu), (\mu', \nu') \in M_{\lambda,\rho}}, T_{Im,\lambda,\rho}=\{b_{\mu\nu\mu'\nu'}\}_{(\mu, \nu), (\mu', \nu') \in M_{\lambda,\rho}}.$
	By the definition of $a_{\mu\nu\mu'\nu'}$ and $\mathrm{i}b_{\mu\nu\mu'\nu'}$, it is obvious that $T_{Re,\lambda,\rho}$ and $T_{Im,\lambda,\rho}$ are Hermite matrix, moreover, $T_{Im,\lambda,\rho}$ is semi-definite. 
	Our key assumption in this paper is the so called Fermi's Golden Rule, which is:
	\begin{assu}[Fermi's Golden Rule] \label{assu:FGR}
		For all $(\lambda,\rho)\in \Lambda^*$, the resonant matrix $T_{Im,\lambda,\rho}$ is invertible, or equivalently, is definite.
	\end{assu} 
	Since the expression is quadratic in ${\eta}^{\mu}{\bar{\eta}}^{\nu}$. we define the vector $$\Gamma_{\lambda,\rho}=\{{\eta}^{\mu}{\bar{\eta}}^{\nu}\}_{(\mu, \nu) \in M_{\lambda,\rho}},$$ we also define 
	$$X_j=\frac{1}{2}\sum_{1\le k\le l_j}|\eta_{jk}|^2, X=\{X_j\}_{1\le j\le n}.$$
	Now we isolate the key resonant terms
	\begin{align*}
		&-Im\left(\sum_{\substack{(\lambda,\rho)\in\Lambda^* }}\sum_{\substack{(\mu, \nu) \in M_{\lambda,\rho} \\ (\mu', \nu') \in M_{\lambda,\rho} }}\eta^{\mu+\nu'} \bar{\eta}^{\nu+\mu'}(\lambda_j c_{\mu\nu\mu'\nu'}+\rho_{j}'\bar{c}_{\mu'\nu'\mu\nu})\right)\\
		=&-\sum_{(\lambda,\rho)\in\Lambda^*}(\lambda_j-\rho_{j})\Gamma_{\lambda,\rho}T_{Im,\lambda,\rho}\bar{\Gamma}^{T}_{\lambda,\rho}.
	\end{align*}
	By the Fermi's Golden Rule condition, we have 
	\begin{equation*}
		\Gamma_{\lambda,\rho}T_{Im,\lambda,\rho}\bar{\Gamma}^{T}_{\lambda,\rho}\approx |\Gamma_{\lambda,\rho}|^2= \sum_{(\mu, \nu) \in M_{\lambda,\rho}}|{\eta}^{\mu+\nu}|^2\approx X^{\lambda+\rho}.
	\end{equation*}
	Set 
	\begin{equation*}
		c_{\lambda\rho}=\frac{\Gamma_{\lambda,\rho}T_{Im,\lambda,\rho}\bar{\Gamma}^{T}_{\lambda,\rho}}{X^{\lambda+\rho}},
	\end{equation*}
	then
	\begin{equation*}
		c_{\lambda\rho}\approx 1	
	\end{equation*}
	and
	\begin{align*}
		-Im\left(\sum_{\substack{(\lambda,\rho)\in\Lambda^* }}\sum_{\substack{(\mu, \nu) \in M_{\lambda,\rho} \\ (\mu', \nu') \in M_{\lambda,\rho} }}\eta^{\mu+\nu'} \bar{\eta}^{\nu+\mu'}(\lambda_j c_{\mu\nu\mu'\nu'}+\rho_{j}'\bar{c}_{\mu'\nu'\mu\nu})\right)=-\sum_{(\lambda,\rho)\in\Lambda^*}(\lambda_j-\rho_{j})c_{\lambda\rho}X^{\lambda+\rho}.
	\end{align*}
	The next lemma explains the reason we choose such as key resonant terms and treat other terms perturbatively:
	\begin{lem}\label{lem:pert}
		\begin{align*}
			&\bigg|\bigg(\sum_{\substack{(\lambda,\rho),(\lambda',\rho')\in \Lambda \times \Lambda   \\ \lambda-\rho=\lambda'-\rho'}}-\sum_{\substack{(\lambda,\rho),(\lambda',\rho')\in  \Lambda^* \times \Lambda^*  \\ (\lambda,\rho)=(\lambda',\rho')}} \bigg)\sum_{\substack{(\mu, \nu) \in M_{\lambda,\rho} \\ (\mu', \nu') \in M_{\lambda',\rho'} }}\eta^{\mu+\nu'} \bar{\eta}^{\nu+\mu'}(\lambda_j c_{\mu\nu\mu'\nu'}+\rho_{j}'\bar{c}_{\mu'\nu'\mu\nu})\bigg| \\
			&+\bigg|\sum_{1\le k\le l_j}\sum_{\substack{(\mu, \nu)\in \mathcal{C}_{jk} \\ \omega\cdot(\nu-\mu+e_{jk}) = 0}}c_{\mu\nu}\eta^{\mu}\bar{\eta}^{\nu+e_{jk}}\bigg|\lesssim |X|\sum_{(\lambda,\rho)\in\Lambda^*} X^{\lambda+\rho}(\lambda_{j}+\rho_{j}) + X_j\sum_{(\lambda,\rho)\in\Lambda^*}X^{\lambda+\rho}.
		\end{align*}
	\end{lem}
	Before proceeding, we define 
	\begin{equation*}
		\Theta=\{\theta=\lambda-\rho | (\lambda,\rho)\in \Lambda\}.
	\end{equation*}
	For a given $\theta \in \Theta,$ define
	\begin{equation}
		\Lambda_{\theta}= \{  (\lambda,\rho)\in \Lambda|\lambda-\rho=\theta\}.
	\end{equation}
	Our observation is that the minimal element in $\Lambda_{\theta}$ is unique:
	\begin{lem}\label{lem:Lambda-theta}
		For each $\theta \in \Theta$, there exists a unique minimal element $(\lambda^{\theta},\rho^{\theta})$ in $\Lambda_{\theta}$, in the sense that for any $(\lambda',\rho')\in \Lambda_{\theta},$ we have $(\lambda^{\theta},\rho^{\theta})\le (\lambda',\rho').$
	\end{lem}
	\begin{proof}
		The key observation is that if $(\lambda^{\theta},\rho^{\theta})$ is a minimal element, then $\lambda^{\theta} \cdot \rho^{\theta} = 0$, i.e. for any $1\le j\le n$ at least one of $\lambda^{\theta}_{j}$ and $\rho^{\theta}_{j}$ is zero, otherwise $(\lambda^{\theta}-e_j,\rho^{\theta}-e_j)$ is a smaller element. Hence we define $\theta^+_j=\theta_j$ if $\theta_j>0$, $\theta^+_j=0$ if $\theta_j\le 0$, and define $\theta^-_j=-\theta_j$ if $\theta_j<0$, $\theta^+_j=0$ if $\theta_j\ge 0$, then $\theta=\theta^+ - \theta^-.$ By the orthogonal property, we have $\lambda^{\theta}=\theta^+, \rho^{\theta}=\theta^-,$ hence is unique.
	\end{proof}
	Now we prove Lemma \ref{lem:pert}.
	\begin{proof}[Proof of Lemma \ref{lem:pert}]
		By Lemma \ref{lem:Lambda-theta}, for $(\lambda,\rho), (\lambda',\rho')\in  \Lambda^*$, $\lambda-\rho=\lambda'-\rho'$ implies $(\lambda,\rho)=(\lambda',\rho').$ Hence,
		\begin{align*}
			&\bigg|\bigg(\sum_{\substack{\big((\lambda,\rho),(\lambda',\rho')\big)\in \Lambda \times \Lambda   \\ \lambda-\rho=\lambda'-\rho'}}-\sum_{\substack{\big((\lambda,\rho),(\lambda',\rho')\big)\in  \Lambda^* \times \Lambda^*  \\ (\lambda,\rho)=(\lambda',\rho')}} \bigg)\sum_{\substack{(\mu, \nu) \in M_{\lambda,\rho} \\ (\mu', \nu') \in M_{\lambda',\rho'} }}\eta^{\mu+\nu'} \bar{\eta}^{\nu+\mu'}(\lambda_j c_{\mu\nu\mu'\nu'}+\rho_{j}'\bar{c}_{\mu'\nu'\mu\nu})\bigg| \\
			=&\bigg|\sum_{\substack{\big((\lambda,\rho),(\lambda',\rho')\big)\in \Lambda \times \Lambda \backslash  \Lambda^* \times \Lambda^*\\ \lambda-\rho=\lambda'-\rho'}}\sum_{\substack{(\mu, \nu) \in M_{\lambda,\rho} \\ (\mu', \nu') \in M_{\lambda',\rho'} }}\eta^{\mu+\nu'} \bar{\eta}^{\nu+\mu'}(\lambda_j c_{\mu\nu\mu'\nu'}+\rho_{j}'\bar{c}_{\mu'\nu'\mu\nu})\bigg|\\
			\lesssim& \sum_{\substack{\big((\lambda,\rho),(\lambda',\rho')\big)\in \Lambda \times \Lambda \backslash  \Lambda^* \times \Lambda^*\\ \lambda-\rho=\lambda'-\rho'}}X^{\frac{\lambda+\rho+\lambda'+\rho'}{2}}(\lambda_{j}+\rho'_{j}).
		\end{align*}
		We further analyze the set 
		\begin{equation*}
			\{\big((\lambda,\rho),(\lambda',\rho')\big)\in \Lambda \times \Lambda \backslash  \Lambda^* \times \Lambda^* | \lambda-\rho=\lambda'-\rho'\}.
		\end{equation*}
		$\lambda-\rho=\lambda'-\rho'$ means that $(\lambda,\rho)$ and $(\lambda',\rho')$ belong to a same $\Lambda_{\theta}$. Hence 
		\begin{equation}
			\{\big((\lambda,\rho),(\lambda',\rho')\big)\in \Lambda \times \Lambda \backslash  \Lambda^* \times \Lambda^* | \lambda-\rho=\lambda'-\rho'\}=\bigcup_{\theta \in \Theta} \Lambda_{\theta} \times \Lambda_{\theta} \backslash  \Lambda^* \times \Lambda^*
		\end{equation}
		and
		\begin{align*}
			\sum_{\substack{\big((\lambda,\rho),(\lambda',\rho')\big)\in \Lambda \times \Lambda \backslash  \Lambda^* \times \Lambda^*\\ \lambda-\rho=\lambda'-\rho'}}X^{\frac{\lambda+\rho+\lambda'+\rho'}{2}}(\lambda_{j}+\rho'_{j})=\sum_{\theta \in \Theta}	\sum_{\big((\lambda,\rho),(\lambda',\rho')\big)\in \Lambda_{\theta} \times \Lambda_{\theta} \backslash  \Lambda^* \times \Lambda^*}X^{\frac{\lambda+\rho+\lambda'+\rho'}{2}}(\lambda_{j}+\rho'_{j}).
		\end{align*}
		
		Now we divide our discussion into three cases.\\
		Case(i): $\lambda^{\theta}_{j}+\rho^{\theta}_{j}\ne 0$ and $(\lambda^{\theta},\rho^{\theta})\notin \Lambda^*$. In this case, $\exists (\lambda^*, \rho^*)\in \Lambda^*$, such that $(\lambda^{\theta},\rho^{\theta})=(\lambda^*, \rho^*)+(a,b)$, with $(a,b)\ne  0.$ Then 
		\begin{align*}
			X^{\frac{\lambda+\rho+\lambda'+\rho'}{2}}(\lambda_{j}+\rho'_{j})&\lesssim X^{\lambda^{\theta}+\rho^{\theta}}(\lambda^{\theta}_{j}+\rho^{\theta}_{j})\\
			&=X^{\lambda^*+ \rho^*}X^{a+b}(\lambda^*_{j}+ \rho^*_{j}+a_{j}+b_{j}),
		\end{align*} 
		with $\lambda^*_{j}+ \rho^*_{j}+a_{j}+b_{j}=\lambda^{\theta}_{j}+\rho^{\theta}_{j}\ne 0.$ 
		If $\lambda^*_{j}+ \rho^*_{j}\ne 0$, then 
		\begin{equation*}
			X^{\lambda^*+ \rho^*}X^{a+b}(\lambda^*_{j}+ \rho^*_{j}+a_{j}+b_{j})\lesssim X^{\lambda^*+ \rho^*}X^{a+b}(\lambda^*_{j}+ \rho^*_{j})\lesssim |X|X^{\lambda^*+ \rho^*}(\lambda^*_{j}+ \rho^*_{j})\lesssim |X|\sum_{(\lambda,\rho)\in\Lambda^*} X^{\lambda+\rho}(\lambda_{j}+\rho_{j}).
		\end{equation*}
		If $a_{j}+b_{j}\ne 0,$ then
		\begin{equation*}
			X^{\lambda^*+ \rho^*}X^{a+b}(\lambda^*_{j}+ \rho^*_{j}+a_{j}+b_{j})\lesssim  X_jX^{\lambda^*+ \rho^*} \lesssim X_j\sum_{(\lambda,\rho)\in\Lambda^*}X^{\lambda+\rho}.
		\end{equation*}
		Hence case(i) is proved.\\
		Case(ii): $\lambda^{\theta}_{j}+\rho^{\theta}_{j}\ne 0$ and $(\lambda^{\theta},\rho^{\theta})\in \Lambda^*$. In this case, since $\big((\lambda,\rho),(\lambda',\rho')\big)\in \Lambda \times \Lambda \backslash  \Lambda^* \times \Lambda^*$, we have $\lambda+\rho > \lambda^{\theta}+\rho^{\theta} $ or $\lambda'+\rho' > \lambda^{\theta}+\rho^{\theta}$. Note that by definition the absolute value of any element in $\Lambda$ mush be odd, in both cases we have
		\begin{equation*}
			X^{\frac{\lambda+\rho+\lambda'+\rho'}{2}}\lesssim |X|X^{\lambda^{\theta}+\rho^{\theta}}\lesssim |X|X^{\lambda^{\theta}+\rho^{\theta}}(\lambda^{\theta}_{j}+\rho^{\theta}_{j})\lesssim |X|\sum_{(\lambda,\rho)\in\Lambda^*} X^{\lambda+\rho}(\lambda_{j}+\rho_{j}).
		\end{equation*}
		Hence case(ii) is proved.\\
		case(iii): $\lambda^{\theta}_{j}=0,  \rho^{\theta}_{j}=0.$ In this case, we write 
		\begin{align*}
			&(\lambda, \rho)=(\lambda^{\theta},\rho^{\theta})+(a,b)	\\
			&(\lambda', \rho')=(\lambda^{\theta},\rho^{\theta})+(c,d).
		\end{align*}
		Then $\lambda-\rho=\lambda'-\rho'$ implies $a_j-b_j=c_j-d_j$ or equivalently $a_j+d_j=b_j+c_j$. Since any non vanishing term satisfies $\lambda_{j}+\rho'_{j}\ne 0$, we have $a_j+d_j\ne 0$. Hence
		\begin{equation*}
			X^{\frac{\lambda+\rho+\lambda'+\rho'}{2}}(\lambda_{j}+\rho'_{j})\lesssim X^{\lambda^{\theta}+\rho^{\theta}}X^{\frac{a+b+c+d}{2}}\lesssim X^{\lambda^{\theta}+\rho^{\theta}}X_j^{a_j+d_j}\lesssim X^{\lambda^{\theta}+\rho^{\theta}}X_j\lesssim X_j\sum_{(\lambda,\rho)\in\Lambda^*}X^{\lambda+\rho}.
		\end{equation*} 
		Hence case(iii) is proved. For higher order terms, the estimate is simple. We have for $(\mu, \nu)\in \mathcal{C}_{jk}$ there exists $(\mu_1,\nu_1)$ and $(\mu_2,\nu_2)$ such that $\mu+\nu+e_{jk}>\mu_1+\nu_1+\mu_1+\nu_1$ and $|\mu+\nu+e_{jk}|\ge |\mu_1+\nu_1+\mu_1+\nu_1|+2$, hence we have
		\begin{align*}	&\bigg|\sum_k\sum_{\substack{(\mu, \nu)\in \mathcal{C}_{jk} \\ \omega\cdot(\nu-\mu+e_{jk}) = 0}}c_{\mu\nu}\eta^{\mu}\bar{\eta}^{\nu+e_{jk}}\bigg|\\
			\lesssim& |X|\sum_{\substack{\big((\lambda,\rho),(\lambda',\rho')\big)\in \Lambda \times \Lambda \\ \lambda-\rho=\lambda'-\rho'}}X^{\frac{\lambda+\rho+\lambda'+\rho'}{2}}(\lambda_{j}+\rho'_{j})\\
			\lesssim& |X|\bigg(\sum_{\substack{\big((\lambda,\rho),(\lambda',\rho')\big)\in \Lambda \times \Lambda \backslash  \Lambda^* \times \Lambda^* \\ \lambda-\rho=\lambda'-\rho'}}+\sum_{\substack{\big((\lambda,\rho),(\lambda',\rho')\big)\in  \Lambda^* \times \Lambda^*  \\ (\lambda,\rho)=(\lambda',\rho')}} \bigg)X^{\frac{\lambda+\rho+\lambda'+\rho'}{2}}(\lambda_{j}+\rho'_{j})\\
			\lesssim& |X|\bigg(|X|\sum_{(\lambda,\rho)\in\Lambda^*} X^{\lambda+\rho}(\lambda_{j}+\rho_{j}) + X_j\sum_{(\lambda,\rho)\in\Lambda^*}X^{\lambda+\rho}+\sum_{(\lambda,\rho)\in\Lambda^*} X^{\lambda+\rho}(\lambda_{j}+\rho_{j})\bigg)\\
			\lesssim& |X|\sum_{(\lambda,\rho)\in\Lambda^*} X^{\lambda+\rho}(\lambda_{j}+\rho_{j}) + X_j\sum_{(\lambda,\rho)\in\Lambda^*}X^{\lambda+\rho}.
		\end{align*}
		The proof is complete.
	\end{proof}
	Combining \eqref{eq:eta-square} and Lemma \ref{lem:pert}, we have
	\begin{prop}
		The discrete variable $X$ satisfies following equations:
		\begin{align}\label{eq:X}
			\frac{d}{dt}X_j=-\sum_{(\lambda,\rho)\in\Lambda^*}(\lambda_j-\rho_{j})c_{\lambda\rho}X^{\lambda+\rho}+P_j+R_j,
		\end{align}
		where
		\begin{equation*}
			P_j=\mathcal{O}\bigg(|X|\sum_{(\lambda,\rho)\in\Lambda^*} X^{\lambda+\rho}(\lambda_{j}+\rho_{j}) + X_j\sum_{(\lambda,\rho)\in\Lambda^*}X^{\lambda+\rho}\bigg),
		\end{equation*}
		and
		\begin{equation*}
			R_j= \mathcal{O}\left(\sum_k\bar{\eta}_{jk}\mathcal{R}_{2jk}\right).	
		\end{equation*}
	\end{prop}
	
	\section{Cancellation of the Bad Resonance}\label{sec:bad-res}
	To analyze the dynamical behavior of $X$ by \eqref{eq:X}, the main obstacle is the existence of possible bad resonance, i.e. the term $(\lambda_j-\rho_{j})c_{\lambda\rho}X^{\lambda+\rho}$ with $\lambda_j-\rho_{j}< 0$, which may cause the increase of $X_j$ in some time period. We remark here that the potential bad resonance $\rho_{j}c_{\lambda\rho}X^{\lambda+\rho}$ with a positive sign is inevitable due to the cubic nonlinearity of the equation and the presence of multiple eigenvalues. To illustrate, we recall that the order of normal form is actually increased by two in each step, which constrains that  the integer $|\lambda+\rho|$ must be odd for any multiple indexes $(\lambda, \rho) \in \Lambda$. This leads to the fact that there may exist $(\lambda, \rho)$ with $\rho \ne 0$ in the minimal set $\Lambda^*$, which is a key difference compared with the single eigenvalue case in \cite{SW1999}.
	
	However, the good thing is that this bad resonance is relatively weak, in sense that if we multiply $\omega_{j}$ and add all $j$ up, then
	\begin{equation*}
		\sum_{1\le j\le n}\omega_j(\lambda_j-\rho_{j})X^{\lambda+\rho}> m X^{\lambda+\rho},
	\end{equation*}
	which is positive. Inspired by this, we will introduce a new variable to overcome this difficulty. We begin by the following lemma on the structure of $\Lambda^*$:
	\begin{lem}[Structure of $\Lambda^*$]\label{lem:Lambda}
		For any $(\lambda,\rho)\in\Lambda^*$, we have \\
		(i) $|\rho|=0 ~or~ 1,$\\
		(ii) if $|\rho|=1$, then there exists $j\ge 2$ such that $\rho_{j}=1$ and $\lambda_{k}=0$ for any $k\ge j.$
	\end{lem}
	\begin{proof}
		The proof is simple. If $|\rho|\ge 2,$ we can choose $\tilde{\rho}< \rho$ such that $|\tilde{\rho}|=|\rho|-2,$ then $(\lambda,\tilde{\rho})$ is a smaller element, which is a contradiction. If $\rho_{j}=1$ and there exists $k\ge j$ such that $\lambda_{k}\ne 0,$ then $(\lambda-e_k,\rho-e_j)$ is a smaller element, which also leads to a contradiction.
	\end{proof}
	Now we introduce a new set of good variables $\tilde{X}$:
	\begin{equation}
		\tilde{X}_j=\sum_{k\le j}\omega_k X_k, ~\forall 1\le j\le n,
	\end{equation}
	then
	\begin{equation*}
		\frac{d}{dt}\tilde{X}_j=-\sum_{(\lambda,\rho)\in\Lambda^*}\sum_{k\le j}\omega_k(\lambda_k-\rho_{k})c_{\lambda\rho}X^{\lambda+\rho}+\sum_{k\le j}\omega_k P_k+\sum_{k\le j}\omega_k R_k.
	\end{equation*}
	The advantage of this transformation of variables follows from following two novel observations:
	\begin{lem}\label{lem:cancellation}
		We have
		\begin{equation*}
			\sum_{k\le j}\omega_k(\lambda_k-\rho_{k})\approx \sum_{k\le j}	\lambda_k+\rho_{k}.
		\end{equation*}
	\end{lem}
	\begin{proof}
		If $\rho_{k}=0$ for all $k\le j$, then this is obviously true. If $\rho_{k}=1$ for some $k\le j$, then by Lemma \ref{lem:Lambda} we have $\lambda_{l}=0$ for all $l\ge k.$ Thus, 
		\begin{equation*}
			\sum_{k\le j}\omega_k(\lambda_k-\rho_{k})=\sum_{k=1}^{n}\omega_k(\lambda_k-\rho_{k})>m,	
		\end{equation*}
		hence, we have
		\begin{equation*}
			\sum_{k\le j}\omega_k(\lambda_k-\rho_{k})\approx \sum_{k\le j}	\lambda_k+\rho_{k}.
		\end{equation*}
	\end{proof}
	\begin{lem}\label{lem:X-tildeX}
		We have 
		\begin{align*}
			\sum_{(\lambda,\rho)\in\Lambda}\sum_{k\le j}(\lambda_k+\rho_{k})X^{\lambda+\rho}\approx \sum_{(\lambda,\rho)\in\Lambda}\sum_{k\le j}(\lambda_k+\rho_{k})\tilde{X}^{\lambda+\rho}.
		\end{align*}
	\end{lem}
	\begin{proof} First, we have
		\begin{equation}\label{eq:X-tildeX} 
			\sum_{(\lambda,\rho)\in\Lambda}\sum_{k\le j}(\lambda_k+\rho_{k})X^{\lambda+\rho}\lesssim \sum_{(\lambda,\rho)\in\Lambda}\sum_{k\le j}(\lambda_k+\rho_{k})\tilde{X}^{\lambda+\rho},
		\end{equation}
		which follows directly by the fact $X_j\lesssim \tilde{X}_j.$ It remains to prove the reversed inequality. 
		For every fixed time $t$, we define a map $F^t: \{1, \cdots , n\} \to \{1, \cdots , n\}$, such that 
		\begin{equation*}
			X_{F^t(j)}(t)= \max_{k\le j}\{X_k(t)\}, \forall 1\le j \le n.
		\end{equation*}
		Then, 
		\begin{equation*}
			F^t(j)\le j
		\end{equation*}
		and
		\begin{equation*}
			\tilde{X}_j(t)\approx X_{F^t(j)}(t).
		\end{equation*}
		Hence, for $(\lambda,\rho)\in \Lambda$ we have
		\begin{equation*}
			\tilde{X}^{\lambda+\rho}(t)=\prod_{j=1}^n \tilde{X}_j^{\lambda_{j}+\rho_{j}}(t)\approx \prod_{j=1}^n	X_{F^t(j)}^{\lambda_{j}+\rho_{j}}(t)= X^{\theta^t}
		\end{equation*}
		for some multiple index $\theta^t$, where the last equality holds by a rearrangement of $X_{F^t(j)}$. Moreover, we have $$|\theta^t|=|\lambda+\rho|$$ and \begin{align*}
			\sum_{j=1}^{n} \omega_{j} \theta^t_j= \sum_{j=1}^{n}\omega_{F^t(j)} (\lambda_{j}+\rho_{j})\ge \sum_{j=1}^{n}\omega_{j} (\lambda_{j}+\rho_{j}) > m,
		\end{align*}
		where the first inequality follows by $F^t(j)\le j$ and $\omega_{F^t(j)}\ge \omega_{j}$. This implies that $(\theta^t, 0) \in \Lambda.$ In addition, by the definition of $F^t$, if $\sum_{k\le j}(\lambda_k+\rho_{k})\ne 0$, then we also have $\sum_{k\le j} \theta^t_k \ne 0.$ Thus, 
		\begin{equation*}
			\sum_{k\le j}(\lambda_k+\rho_{k})\tilde{X}^{\lambda+\rho} \lesssim \sum_{k\le j} \theta^t_k X^{\theta^t} \lesssim \sum_{(\lambda,\rho)\in\Lambda}\sum_{k\le j}(\lambda_k+\rho_{k})X^{\lambda+\rho}.
		\end{equation*}
		This implies the reversed version of \eqref{eq:X-tildeX}.
	\end{proof}
	
	\begin{lem} The following estimate holds:
		\begin{equation*}
			\sum_{(\lambda,\rho)\in\Lambda}\sum_{k\le j}(\lambda_k+\rho_{k})X^{\lambda+\rho}\lesssim \sum_{(\lambda,\rho)\in\Lambda^*}\sum_{k\le j}(\lambda_k+\rho_{k})X^{\lambda+\rho}+\tilde{X}_j \sum_{(\lambda,\rho)\in\Lambda^*}X^{\lambda+\rho}.
		\end{equation*}		
	\end{lem}	
	Combining the above lemmas we get
	\begin{equation}
		\frac{d}{dt}\tilde{X}_j=-c_j\sum_{(\lambda,\rho)\in\Lambda}\sum_{k\le j}(\lambda_k+\rho_{k})\tilde{X}^{\lambda+\rho}+ \tilde{P}_j + \tilde{R}_j,
	\end{equation}
	with
	\begin{equation*}
		c_j\approx 1,~~
		\tilde{P}_j=\mathcal{O}\bigg( \tilde{X}_j\sum_{(\lambda,\rho)\in\Lambda^*}\tilde{X}^{\lambda+\rho}\bigg),~~
		\tilde{R}_j=\mathcal{O}\bigg(\sum_{k\le j}R_k\bigg).	
	\end{equation*}
	To eliminate the effect of $\tilde{P}_j$, we further introduce new variables
	\begin{equation}\label{eq:hat-X-def}
		\hat{X}_j=exp \left(-C_0\int_{0}^{t}\sum_{(\lambda,\rho)\in\Lambda^*}\tilde{X}^{\lambda+\rho}ds\right) \tilde{X}_j, ~\forall 1\le j\le n,
	\end{equation}
	where $C_0$ is a fixed large number.
	This validity of this transformation is based on the following observation:  
	\begin{equation*}
		\int_{0}^{\infty}\sum_{(\lambda,\rho)\in\Lambda^*}\tilde{X}^{\lambda+\rho}ds \lesssim \epsilon,
	\end{equation*}
	which we will prove in the next section. As a consequence, 
	\begin{equation*}
		\hat{X}_j\approx \tilde{X}_j.
	\end{equation*}
	We finally derive the ODE that we will work with:
	\begin{prop}
		\begin{equation}\label{eq:X-hat}
			\frac{d}{dt}\hat{X}_j=-\hat{c}_j\bigg(\sum_{(\lambda,\rho)\in\Lambda}\sum_{k\le j}(\lambda_k+\rho_{k})\hat{X}^{\lambda+\rho}+\hat{X}_j\sum_{(\lambda,\rho)\in\Lambda^*}{\hat{X}}^{\lambda+\rho}\bigg)+ \hat{R}_j,
		\end{equation}
		with
		\begin{equation*}
			\hat{c}_j\approx 1,~~
			\hat{R}_j=\mathcal{O}\bigg(\sum_{k\le j}R_k\bigg).	
		\end{equation*}
	\end{prop}
	
	\section{Dynamics of the New Variable $\hat{X}$}\label{sec-ODE}
	In this section we will analyze the dynamics of $\hat{X}$, using Proposition \ref{eq:X-hat}.
	Define
	\begin{equation*}
		\frac{d}{dt}Y=-Y^{2N_n+1}, Y_0=\epsilon^2
	\end{equation*}
	\begin{equation*}
		\frac{d}{dt}W=-W^{2}, W_0=\epsilon^{\kappa}, \kappa=\min\{8, 2(\lambda+\rho)\cdot \alpha-2, (\lambda,\rho)\in\Lambda\},
	\end{equation*}
	The equations of $Y$ and $W$ can be solved explicitly:
	\begin{align*}
		&Y=\frac{\epsilon^2}{\left(1+2N_n\epsilon^{4N_n}t\right)^{\frac{1}{2N_n}}},\\
		&W=\frac{\epsilon^{\kappa}}{1+\epsilon^{\kappa}t}.
	\end{align*}
	We also choose $j_0 \in \{1, \cdots, n\}$, such that for any $j<j_0$, $N_j<N_n$ and for any $j\ge j_0$, $N_j=N_n.$ 
	
	In the following, we shall derive upper bounds of the decay rates of discrete variables $\hat{X}$ using a bootstrap argument. More precisely, we prove that
	\begin{thm}\label{thm:X} The discrete variables $\hat{X}$ satisfy the following estimates:
		\begin{align}
			&|\xi_{jk}|\lesssim Y^{\frac{\alpha_j}{2}} \label{eq:xi-esti-1}\\
			&|\xi|\approx Y^{\frac{1}{2}}\\
			&|\xi^{\mu+\nu}|\lesssim Y^{\frac{1}{2}}W^{\frac{1}{2}},  ~\forall (\mu, \nu)\in M \label{eq:xi-esti-3}\\
			&\tilde{X}^{\lambda+\rho}\lesssim Y^{1+\delta}\langle t\rangle^{-1} \text{~ for some ~} \delta>0, \text{~ if~} (\lambda,\rho)\in\Lambda \text{~and~} \exists j<j_0 ~s.t.~ \lambda_{j}+\rho_{j}\ne 0. \label{eq:xi-esti-4}
		\end{align}
		Here $\delta$ is a small absolute constant, for our choice $\delta=\frac{1}{100N_n}$ is sufficient.
	\end{thm}
	To proceed, we need the following estimates of error terms, which is proved in the next section:
	\begin{prop}\label{prop:error} If \eqref{eq:xi-esti-1}-\eqref{eq:xi-esti-4} holds for $0\le t\le T$, then for  $0\le t\le T$ we have
		\begin{align*}
			&\|f_R\|_{L^{2,-s}}\lesssim \epsilon^3\langle t\rangle^{-\frac{9}{8}}+ \epsilon^{2-\delta} W
		\end{align*}
		for some large $s$. 
	\end{prop}
	As a corollary, we have 
	\begin{align*}
		|R_j|&= \mathcal{O}\left(\sum_k\bar{\eta}_{jk}\mathcal{R}_{2jk}\right)\\
		&\lesssim \sum_k\Big( |\xi_{jk}|\mathcal{R}_{1jk}+|\xi_{jk}-\eta_{jk}|\mathcal{R}_{1jk}\Big)+Y^{50N_n}\\
		&\lesssim \sum_k \sum_{(\mu, \nu) \in \tilde{M}}\|f_R\|_{L^{2,-s}} |\bar{\xi}^{\mu}\xi^{\nu}|+|\xi_{jk}|\|f_R\|_{L^{2,-s}}^2+Y^{50N_n}\\
		&\lesssim \epsilon^{2-\delta}Y^{\frac{1}{2}}W^{\frac{3}{2}}+\epsilon^3\langle t\rangle^{-\frac{9}{8}}Y^{\frac{1}{2}}W^{\frac{1}{2}}+\epsilon^7\langle t\rangle^{-\frac{9}{4}}.\\ 			
		|\hat{R}_j|& \lesssim \sum_{k\le j}|R_j|\lesssim \epsilon^{2-\delta}Y^{\frac{1}{2}}W^{\frac{3}{2}}+\epsilon^3\langle t\rangle^{-\frac{9}{8}}Y^{\frac{1}{2}}W^{\frac{1}{2}}+\epsilon^7\langle t\rangle^{-\frac{9}{4}}. 
	\end{align*}
	\begin{proof}[Proof of Theorem \ref{thm:X}]
		First, the theorem holds trivially for small $t$ due to initial conditions. Now we assume the theorem holds for $0\le t\le T$, then we have 
		\begin{align*}
			&\hat{X}_{j}\lesssim Y^{\alpha_j}\\
			&|\hat{X}|\approx Y\\
			&\hat{X}^{\lambda+\rho}\lesssim YW, ~\forall (\lambda,\rho)\in\Lambda\\
			& (1-\epsilon^{\delta})\tilde{X}_j\le \hat{X}_j\le \tilde{X}_j.
		\end{align*}
		We start by estimating $\hat{X}_j$. Choose $(\lambda,\rho)=\left((2N_j+1)e_j,0\right)\in \Lambda$, by \eqref{eq:X-hat} we have 
		\begin{align*}
			\frac{d}{dt}\hat{X}_j\le -\hat{c}_j\hat{X}_j^{2N_j+1}+|\hat{R}_j|
		\end{align*}
		Choose $t_1=\epsilon^{-4N_n}$, for $t\le t_1$, we have 
		\begin{align*}
			\int_0^{t_1}|\hat{R}_j|ds\lesssim 	\int_0^{t_1}\epsilon^{2-\delta}Y^{\frac{1}{2}}W^{\frac{3}{2}}+\epsilon^3\langle s\rangle^{-\frac{9}{8}}Y^{\frac{1}{2}}W^{\frac{1}{2}}+\epsilon^7\langle s\rangle^{-\frac{9}{4}}ds\lesssim \epsilon^{3+\frac{\kappa}{2}-2\delta},
		\end{align*}
		where we use the fact that
		$$Y\approx (\epsilon^{-4N_n} + t)^{-\frac{1}{2N_n}}, \quad W\approx (\epsilon^{-\kappa} + t)^{-1}.$$
		Besides, we have $3+\frac{\kappa}{2}-2\delta>2\alpha_1\ge 2\alpha_j$, which implies $\hat{X}_j(t_1)\lesssim \epsilon^{2\alpha_j}$.  Actually, for any $(\lambda,\rho)\in \Lambda$, if $|\lambda+\rho|=3$, then $\lambda_1+\rho_1$ must be non-zero to ensure that $\omega\cdot (\lambda-\rho)>m$; thus $(\lambda+\rho)\cdot\alpha\ge \min\{\alpha_1 +2,5\}$ due to $\alpha_j\ge 1$. Hence, by the definition of $\kappa$, $3+\frac{\kappa}{2}-2\delta\ge \min\{7-2\delta, 4+\alpha_1-2\delta\}> 6\ge 2\alpha_1$. 
		For $t> t_1$, we have $|\hat{R}_j|\lesssim \epsilon Y^{3N_n+\frac{1}{2}}\ll Y^{(2N_j+1)\alpha_j},$ by comparison theorem we get
		\begin{equation*}
			\hat{X}_j\lesssim Y^{\alpha_j}.
		\end{equation*}
		
		For any $(\lambda, \rho)\in  \Lambda$, we have 
		\begin{align*}
			\frac{d}{dt}\hat{X}^{\lambda+ \rho}&\le \sum_{j}\hat{X}^{\lambda+ \rho}\frac{\lambda_j+ \rho_j}{\hat{X}_{j}}\left(-\hat{c}_j\sum_{(\tilde{\lambda},\tilde{\rho})\in\Lambda}\sum_{k\le j}(\tilde{\lambda}_k+\tilde{\rho}_{k})\hat{X}^{\tilde{\lambda}+\tilde{\rho}}+\hat{R}_j\right)\\
			&\le -\sum_{j}\hat{X}^{\lambda+ \rho}\frac{\lambda_j+ \rho_j}{\hat{X}_{j}}\left( \hat{c}_j\hat{X}^{\lambda+ \rho}-|\hat{R}_j|\right).
		\end{align*}
		Choosing $t_0=\epsilon^{-\kappa}$, we have for $t\le t_0$
		\begin{align*}
			\int_0^{t} |\hat{R}_j| \frac{\hat{X}^{\lambda+ \rho}}{\hat{X}_{j}}(\lambda_j+ \rho_j) ds
			&\lesssim  \int_0^{t} |\hat{R}_j| |\hat{X}|^2 ds \\
			& \lesssim  \int_0^{t} \epsilon^{2-\delta}Y^{\frac{5}{2}}W^{\frac{3}{2}}+\epsilon^3\langle s\rangle^{-\frac{9}{8}}Y^{\frac{5}{2}}W^{\frac{1}{2}}+\epsilon^7\langle s\rangle^{-\frac{9}{4}}Y^2ds\\
			& \lesssim \epsilon^{7-\delta+\frac{\kappa}{2}} + \epsilon^{8+\frac{\kappa}{2}}+\epsilon^{11}\\
			& \lesssim \epsilon^{2+\kappa},
		\end{align*}
		thus $\hat{X}^{\lambda+ \rho}(t_0)\lesssim \epsilon^{2+\kappa} \approx YW(t_0)$. For $t\ge t_0$, note that 
		$$|\hat{R}_j|\lesssim Y^{1+\delta}W$$
		and 
		$$	-\sum_j \frac{\lambda_j+\rho_j}{\hat{X}_j} \hat{c}_j \lesssim -\frac{1}{Y},$$
		by comparison theorem we get
		\begin{align*}
			\hat{X}^{\lambda+ \rho}\lesssim YW.
		\end{align*}
		In addition, if $\exists j<j_0 $ s.t. $ \lambda_{j}+\rho_{j}\ne 0$, since $|\hat{R}_j|\lesssim Y^{1+2\delta}\langle t\rangle^{-1}$, by comparison theorem we get
		\begin{align*}
			\hat{X}^{\lambda+ \rho}\lesssim Y^{1+\delta}\langle t\rangle^{-1}.
		\end{align*}
		This estimate also implies that
		\begin{align*}
			\int_{0}^{\infty}\sum_{(\lambda,\rho)\in\Lambda^*}\hat{X}^{\lambda+\rho}ds \lesssim \int_{0}^{\infty} YW ds\lesssim \epsilon.
		\end{align*}
		
		For the lower bound of $|\hat{X}|$, we have
		\begin{equation}
			\frac{d}{dt}\hat{X}_n=-\hat{c}_n\bigg(\sum_{(\lambda,\rho)\in\Lambda}|\lambda+\rho|\hat{X}^{\lambda+\rho}+\hat{X}_n\sum_{(\lambda,\rho)\in\Lambda^*}{\hat{X}}^{\lambda+\rho}\bigg)+ \hat{R}_n.
		\end{equation}
		Choosing $t_2=\epsilon^{-4N_n+\frac{\delta}{100}}$, we have
		\begin{equation*}
			\int_0^{t_2}|R_n|ds\lesssim \epsilon^{3+\frac{\kappa}{2}-2\delta}.
		\end{equation*}
		For $(\lambda, \rho)$ such that there exists $j<j_0$ s.t. $\lambda_j+\rho_{j}\ne 0$, we have
		\begin{equation*}
			\int_0^{t_2}\hat{X}^{\lambda+\rho}ds\lesssim \int_0^{t_2}Y^{1+\delta}\langle t\rangle^{-1}ds\lesssim \epsilon^{2+\delta}.
		\end{equation*}
		For $(\lambda, \rho)$ such that for all $j<j_0$, $\lambda_j+\rho_{j}=0$, we have $|\lambda+\rho|=2N_n+1.$ Hence
		\begin{equation*}
			\int_0^{t_2}\hat{X}^{\lambda+\rho}ds\lesssim \int_0^{t_2}Y^{2N_n+1}ds\lesssim \epsilon^{2+\frac{\delta}{100}}.
		\end{equation*}
		Similarly,
		\begin{equation*}
			\int_0^{t_2}\sum_{(\lambda,\rho)\in\Lambda^*}\hat{X}^{\lambda+\rho}\hat{X}_n ds\lesssim \int_0^{t_2} Y^2Wds \lesssim \epsilon^{2+\delta}.
		\end{equation*}
		We have $\hat{X}_n(t_2)\approx \epsilon^{2}$.
		For $t\ge t_2$, we have 
		\begin{equation*}
			\frac{d}{dt}\hat{X}_n \gtrsim -\hat{X}_n^{2N_n+1}-Y^{1+\delta}\langle t\rangle^{-1},
		\end{equation*}
		and
		\begin{align*}
			Y^{2N_n}\ge \epsilon^{\frac{\delta}{100}}\langle t\rangle^{-1},
		\end{align*}
		hence we have
		\begin{equation*}
			\frac{d}{dt}\hat{X}_n\gtrsim -\hat{X}_n^{2N_n+1}-\epsilon^{\frac{\delta}{100}}Y^{2N_n+1+\delta}\gtrsim -\hat{X}_n^{2N_n+1}-Y^{2N_n+1+\frac{\delta}{2}},
		\end{equation*}
		by comparison theorem we get
		\begin{equation*}
			\hat{X}_n\gtrsim Y.	
		\end{equation*}
		
	\end{proof}
	
	\section{Asymptotic Behavior of the Continuum Mode $f$ and Error Estimates}\label{sec-f}
	In this section, we will prove Proposition \ref{prop:error}. As a corollary, we obtain the asymptotic behavior of the continuum mode $f$ and estimates of the error term $f_R$. In the following, we always assume that \eqref{eq:xi-esti-1}-\eqref{eq:xi-esti-4} holds for $0\le t\le T$. 
	\subsection{Strichartz Estimates of $f$}
	By \eqref{eq:xi-esti-3}, for any $(\mu,\nu)\in M$ we have 
	\begin{equation*}
		\|\xi^{\mu+\nu}\|_{L^2_t([0, T])}\lesssim \epsilon.
	\end{equation*}
	The $L^2$-integrability of $\xi^{\mu+\nu}$  would imply the boundedness of high-order Strichartz norms of $f$:
	\begin{prop} 
		Assume that \eqref{eq:xi-esti-1}-\eqref{eq:xi-esti-4} holds for $0\le t\le T$, then for any $0\le k\le 100N_n$, we have 
		\begin{equation}
			\| B^{-1/2}f\|_{L^\infty_t W^{k+1,2}_x} + \| B^{-1/2}f\|_{L^2_t W^{k,6}_x}\lesssim \epsilon.
		\end{equation}
		\begin{proof}
			We have 
			\begin{equation*}
				B^{-1/2}f=e^{-\mathrm{i}B t}B^{-1/2}f(0)+\int_{0}^{t}e^{-\mathrm{i}B (t-s)}B^{-1/2}(-\mathrm{i}\bar{G}-\mathrm{i}\partial_{\bar{f}}\mathcal{R})ds.
			\end{equation*}
			By Proposition \ref{prop:R-f}, we shall only estimate the typical leading order terms 
			$\xi^2 B^{-1/2}\left(\Psi B^{-1/2}f \right)$, 
			$\xi B^{-1/2}\left(\Psi  \left(B^{-1/2}f\right)^2 \right)$ and
			$B^{-1/2}\left(  \left(B^{-1/2}f\right)^3 \right)$ in $\partial_{\bar{f}}\mathcal{R}$, where $\xi^2$ denotes some quadratic monomials of $\xi$ and $\bar{\xi}$. For any $k\ge 0$, by Lemma \ref{lem:Str-1} and Lemma \ref{lem:Str-2}, we have
			\begin{align*}
				&\| B^{-1/2}f\|_{L^\infty_t W^{k+1,2}_x} + \| B^{-1/2}f\|_{L^2_t W^{k,6}_x}\\
				\lesssim& \|B^{-1/2}f_0\|_{W^{k+1,2}_x} + \|G\|_{L^2_t W^{k+\frac{4}{3},\frac{6}{5}}_x}+ \left\|\xi^2 B^{-1/2}\left(\Psi B^{-1/2}f \right)\right\|_{L^{\frac{100N_n}{50N_n+1}}_t H^{k+\frac{1}{2},s}_x}\\
				&+\left\|\xi B^{-1/2}\left(\Psi  \left(B^{-1/2}f\right)^2 \right)\right\|_{L^{\frac{100N_n}{50N_n+1}}_t H^{k+\frac{1}{2},s}_x}
				+\left\|B^{-1/2}\left(  \left(B^{-1/2}f\right)^3 \right)\right\|_{L^1_t W^{k+\frac{1}{2},2}_x}.
			\end{align*}
			Note that 
			$$\|B^{-1/2}f_0\|_{W^{k+1,2}_x}\lesssim \epsilon,$$
			and
			$$\|G\|_{L^2_t W^{k+\frac{4}{3},\frac{6}{5}}_x}\lesssim \sum_{(\mu,\nu)\in M}\|\xi^{\mu+\nu}\|_{L^2_t}\lesssim \epsilon.
			$$
			Moreover, by H\"older's inequality, we have
			\begin{align*}
				\left\|\xi^2 B^{-1/2}\left(\Psi B^{-1/2}f \right)\right\|_{L^{\frac{100N_n}{50N_n+1}}_t H^{k+\frac{1}{2},s}_x}\lesssim \|\xi\|_{L^{200N_n}_t}^2 \| B^{-1/2}f\|_{L^2_t W^{k,6}_x},
			\end{align*}
			\begin{align*}
				\left\|\xi B^{-1/2}\left(\Psi  \left(B^{-1/2}f\right)^2 \right)\right\|_{L^{\frac{100N_n}{50N_n+1}}_t H^{k+\frac{1}{2},s}_x}\lesssim \|\xi\|_{L^{100N_n}_t} \| B^{-1/2}f\|_{L^\infty_t W^{k+1,2}_x}\| B^{-1/2}f\|_{L^2_t W^{k,6}_x},
			\end{align*}
			\begin{align*}
				\left\|B^{-1/2}\left(  \left(B^{-1/2}f\right)^3 \right)\right\|_{L^1_t W^{k+\frac{1}{2},2}_x}\lesssim \| B^{-1/2}f\|_{L^\infty_t W^{k+1,2}_x} \| B^{-1/2}f\|_{L^2_t W^{k,6}_x}^2.
			\end{align*}
			Hence, using a bootstrap argument we have 
			\begin{align*}
				\| B^{-1/2}f\|_{L^\infty_t W^{k+1,2}_x} + \| B^{-1/2}f\|_{L^2_t W^{k,6}_x}\lesssim \epsilon.	
			\end{align*}
		\end{proof}
	\end{prop} 
	The advantage of the decomposition of $f$ is that it preserves the boundedness of high-order Strichartz norms of its components.
	\begin{prop}
		Assume that \eqref{eq:xi-esti-1}-\eqref{eq:xi-esti-4} holds for $0\le t\le T$, then for any $k\ge 0$ we have 
		\begin{equation}
			\| B^{-1/2}f^{(l)}_M\|_{L^\infty_t W^{k+1,2}_x} + \| B^{-1/2}f^{(l)}_M\|_{L^2_t W^{k,6}_x}\lesssim \epsilon.
		\end{equation}
		As a corollary,  we also have for any $0\le k\le 100N_n$
		\begin{align*}
			\| B^{-1/2}f^{(l)}\|_{L^\infty_t W^{k+1,2}_x} + \| B^{-1/2}f^{(l)}\|_{L^2_t W^{k,6}_x}\lesssim \epsilon.
		\end{align*}
	\end{prop}
	\begin{proof}
		
		Recall that
		\begin{equation*}
			f^{(l)}_M=-\mathrm{i}\int_{0}^{t}e^{-\mathrm{i}B (t-s)}Q^{(l)}_0 ds,
		\end{equation*}
		where the leading order terms of $Q^{(l)}_0$ are  
		\begin{align*}
			(i)&\xi^2 B^{-1/2}\left(\Psi B^{-1/2}f^{(l-1)}_M \right),\\
			(ii)&\xi B^{-1/2}\left(\Psi  B^{-1/2}f^{(j)}_M B^{-1/2}f^{(l-1)}_M \right), 0\le j\le l-1,\\
			(iii)& B^{-1/2}\left(B^{-1/2}f^{(i)}_M B^{-1/2}f^{(j)}_M B^{-1/2}f^{(l-1)}_M \right), 0\le i, j\le l-1,
		\end{align*}
		see Section \ref{subsec-Qd}. Hence, we have
		\begin{align*}
			&\| B^{-1/2}f^{(l)}_M\|_{L^\infty_t W^{k+1,2}_x} + \| B^{-1/2}f^{(l)}_M\|_{L^2_t W^{k,6}_x}\\
			\lesssim &\left\|\xi^2 B^{-1/2}\left(\Psi B^{-1/2}f^{(l-1)}_M \right)\right\|_{L^2_t W^{k+\frac{4}{3},\frac{6}{5}}_x}\\
			&+\sum_{0\le j\le l-1}\left\|\xi B^{-1/2}\left(\Psi  B^{-1/2}f^{(j)}_M B^{-1/2}f^{(l-1)}_M \right)\right\|_{L^2_t W^{k+\frac{4}{3},\frac{6}{5}}_x}\\
			&+\sum_{0\le i, j\le l-1}\left\|B^{-1/2}\left(B^{-1/2}f^{(i)}_M B^{-1/2}f^{(j)}_M B^{-1/2}f^{(l-1)}_M \right)\right\|_{L^1_t W^{k+\frac{1}{2},2}_x}\\
			\lesssim & \epsilon\sum_{0\le j\le l-1}\left(	\| B^{-1/2}f^{(j)}_M\|_{L^\infty_t W^{k+3,2}_x}+\| B^{-1/2}f^{(j)}_M\|_{L^2_t W^{k+2,6}_x}\right).
		\end{align*}
		Since
		\begin{align*}
			\| B^{-1/2}f^{(0)}_M\|_{L^\infty_t W^{k+1,2}_x} + \| B^{-1/2}f^{(0)}_M\|_{L^2_t W^{k,6}_x}\lesssim \|G\|_{L^2_t W^{k+\frac{4}{3},\frac{6}{5}}_x} \lesssim \epsilon,
		\end{align*}
		by an induction argument we have
		\begin{align*}
			\| B^{-1/2}f^{(l)}_M\|_{L^\infty_t W^{k+1,2}_x} + \| B^{-1/2}f^{(l)}_M\|_{L^2_t W^{k,6}_x}\lesssim \epsilon.
		\end{align*}
		Since 
		\begin{align*}
			\| B^{-1/2}f\|_{L^\infty_t W^{k+1,2}_x} + \| B^{-1/2}f\|_{L^2_t W^{k,6}_x}\lesssim \epsilon,
		\end{align*}
		this implies that
		\begin{align*}
			\| B^{-1/2}f^{(l)}\|_{L^\infty_t W^{k+1,2}_x} + \| B^{-1/2}f^{(l)}\|_{L^2_t W^{k,6}_x}\lesssim \epsilon.
		\end{align*}
	\end{proof}
	\subsection{Proof of Proposition \ref{prop:error}}
	Assuming \eqref{eq:xi-esti-1}-\eqref{eq:xi-esti-4} holds for $0\le t\le T$, we have
	$$|\xi|\lesssim Y^\frac{1}{2}$$
	and
	$$|\xi^{\mu+\nu}|\lesssim Y^{\frac{1}{2}}W^{\frac{1}{2}},  ~\forall (\mu, \nu)\in M.$$
	Recall that 
	\begin{equation*}
		f=\sum_{l=0}^{l_0-1} f^{(l)}_M + f^{(l_0)}.
	\end{equation*}
	We first estimate $f^{(l)}_M$. By Lemma \ref{Lp-Lp'}, choosing $p=8$ we have
	\begin{align*}
		&\left\|B^{-1/2}f^{(l)}_M(t)\right\|_{W^{k,8}_x}\\
		\lesssim& \int_0^t \langle t-s \rangle^{-\frac{9}{8}}\left\|Q^{(l)}_0(s)\right\|_{W^{k+\frac{3}{2},\frac{8}{7}}}ds\\
		\lesssim& \int_0^t \langle t-s \rangle^{-\frac{9}{8}} \bigg(\left\|\xi^2 B^{-1/2}\left(\Psi B^{-1/2}f^{(l-1)}_M \right)\right\|_{ W^{k+\frac{3}{2},\frac{8}{7}}} \\
		&+\sum_{0\le j\le l-1}\left\|\xi B^{-1/2}\left(\Psi  B^{-1/2}f^{(j)}_M B^{-1/2}f^{(l-1)}_M \right)\right\|_{ W^{k+\frac{3}{2},\frac{8}{7}}}\\
		& +\sum_{0\le i, j\le l-1}\left\|B^{-1/2}\left(B^{-1/2}f^{(i)}_M B^{-1/2}f^{(j)}_M B^{-1/2}f^{(l-1)}_M \right)\right\|_{ W^{k+\frac{3}{2},\frac{8}{7}}}\bigg)\\
		\lesssim& \int_0^t \langle t-s \rangle^{-\frac{9}{8}}\bigg(|\xi(s)|^2 \left\|B^{-1/2}f^{(l-1)}_M(s)\right\|_{W^{k+1,8}_x}\\
		&+|\xi(s)|\sum_{0\le j\le l-1}\left\|B^{-1/2}f^{(j)}_M(s)\right\|_{W^{k+1,8}_x} \left\|B^{-1/2}f^{(l-1)}_M(s)\right\|_{W^{k+1,8}_x}\\
		&+\sum_{0\le i\le l-1}\|B^{-1/2}f^{(i)}_M(s)\|_{W^{k+1,2}_x}\sum_{0\le j\le l-1}\|B^{-1/2}f^{(j)}_M(s)\|_{W^{k+1,2}_x}^{\frac{1}{3}}\|B^{-1/2}f^{(l-1)}_M(s)\|_{W^{k+1,8}_x}^{\frac{5}{3}}\bigg)ds.
	\end{align*}
	Denote 
	$$A_{l,k}(t):=\left\|B^{-1/2}f^{(l)}_M(t)\right\|_{W^{k,8}_x},$$
	we have
	\begin{align*}
		A_{l,k}\lesssim \int_0^t \langle t-s \rangle^{-\frac{9}{8}}\bigg(YA_{l-1,k+1}+Y^{\frac{1}{2}}\sum_{0\le j\le l-1}A_{j,k+1}A_{l-1,k+1}+ A_{l-1,k+1}^{\frac{5}{3}} \bigg)ds.
	\end{align*}
	Note that
	$$A_{0,k}=\bigg\|B^{-1/2}\int_{0}^{t}e^{-\mathrm{i}B (t-s)}\bar{G}ds\bigg\|_{W^{k,8}_x}\lesssim Y^{\frac{1}{2}}W^{\frac{1}{2}}, \forall~ k\ge 0,$$
	by induction we can obtain that
	\begin{equation*}
		A_{l,k}\lesssim Y^{l+\frac{1}{2}}W^{\frac{1}{2}}.
	\end{equation*}
	Hence, for $l \ge \frac{5N_n}{4}$  we have 
	\begin{align*}
		\|B^{-1/2}f^{(l)}_M(t)\|_{W^{k,8}_x}\lesssim Y^{\frac{9N_n}{4}}.
	\end{align*}
	In a similar way, we can also show that
	\begin{align*}
		\left\|B^{-1/2}f^{(l)}_M(t)\right\|_{W^{k,6+}_x}\lesssim Y^{\frac{1}{2}}W^{\frac{1}{2}}.
	\end{align*}  
	We then estimate $f^{(l_0)}$. since
	\begin{equation*}
		f^{(l_0)}=e^{-\mathrm{i}B t}f(0)-\mathrm{i}\int_{0}^{t}e^{-\mathrm{i}B (t-s)}\sum_{d=0}^{4}Q^{(l_0)}_d(f^{(l_0)}) ds,
	\end{equation*}
	we have
	\begin{align*}
		\left\|B^{-1/2}f^{(l_0)}(t)\right\|_{W^{k,8}_x}\lesssim \langle t \rangle^{-\frac{9}{8}}\|B^{-1/2}f(0)\|_{W^{k+2,\frac{8}{7}}_x}+\int_0^t \langle t-s \rangle^{-\frac{9}{8}}\sum_{d=0}^{4}\left\|Q^{(l_0)}_d(s)\right\|_{W^{k+\frac{3}{2},\frac{8}{7}}}ds.
	\end{align*}
	Since 
	\begin{align*}		 
		\left\|B^{-1/2}f^{(l_0)}_M(t)\right\|_{W^{k,8}_x}\lesssim Y^{\frac{9N_n}{4}},
	\end{align*}
	we can prove in a similar way that 
	\begin{equation*}
		\int_0^t \langle t-s \rangle^{-\frac{9}{8}}\left\|Q^{(l_0)}_0(s)\right\|_{W^{k+\frac{3}{2},\frac{8}{7}}}ds\lesssim Y^{\frac{9N_n}{4}}.
	\end{equation*}
	For $1\le d\le 4$, using the structure of $Q^{(l_0)}_d$ (see Section \ref{subsec-Qd}) and H\"older's inequality, we have
	\begin{align*}
		\sum_{d=0}^{4}\left\|Q^{(l_0)}_d(s)\right\|_{W^{k+\frac{3}{2},\frac{8}{7}}}\lesssim Y(s)^{\frac{1}{2}}\left\|B^{-1/2}f^{(l_0)}(s)\right\|_{W^{k+1,8}_x}+\left\|B^{-1/2}f^{(l_0)}(s)\right\|_{W^{k,8}_x}^{\frac{5}{3}}.	
	\end{align*}
	Then, 
	\begin{align*}
		\left\|B^{-1/2}f^{(l_0)}(t)\right\|_{W^{k,8}_x}
		\lesssim& \langle t \rangle^{-\frac{9}{8}}\left\|B^{-1/2}f(0)\right\|_{W^{k+2,\frac{8}{7}}_x}+Y^{\frac{9N_n}{4}}(t)\\
		&+ \int_0^t \langle t-s \rangle^{-\frac{9}{8}}\left(Y(s)^{\frac{1}{2}}\left\|B^{-1/2}f^{(l_0)}(s)\right\|_{W^{k+1,8}_x}+\left\|B^{-1/2}f^{(l_0)}(s)\right\|_{W^{k,8}_x}^{\frac{5}{3}}\right)ds\\
		\lesssim & \langle t \rangle^{-\frac{9}{8}}\epsilon^{3}+Y^{\frac{9N_n}{4}}(t) + \int_0^t \langle t-s \rangle^{-\frac{9}{8}}\left(Y(s)^{\frac{1}{2}}+\left\|B^{-1/2}f^{(l_0)}(s)\right\|_{W^{k,8}_x}^{\frac{5}{3}}\right)ds\\
		\lesssim& \langle t \rangle^{-\frac{9}{8}}\epsilon^{3}+Y(t)^{\frac{1}{2}}+\int_0^t \langle t-s \rangle^{-\frac{9}{8}}\|B^{-1/2}f^{(l_0)}(s)\|_{W^{k,8}_x}^{\frac{5}{3}}ds.
	\end{align*}
	Using a bootstrap argument we have
	\begin{align*}
		\|B^{-1/2}f^{(l_0)}(t)\|_{W^{k,8}_x}\lesssim\langle t \rangle^{-\frac{9}{8}}\epsilon^{3}+Y^{\frac{1}{2}}(t), \forall~ 0\le t\le T.
	\end{align*}
	Furthermore, replacing $k$ by $k-1$ we have
	\begin{align*}
		\|B^{-1/2}f^{(l_0)}(t)\|_{W^{k-1,8}_x}\lesssim&  \langle t \rangle^{-\frac{9}{8}}\epsilon^{3}+Y^{\frac{9N_n}{4}}(t)\\
		&+\int_0^t \langle t-s \rangle^{-\frac{9}{8}}\left(Y(s)^{\frac{1}{2}}\|B^{-1/2}f^{(l_0)}(s)\|_{W^{k,8}_x}+\|B^{-1/2}f^{(l_0)}(s)\|_{W^{k-1,8}_x}^{\frac{5}{3}}\right)ds\\
		\lesssim&  \langle t \rangle^{-\frac{9}{8}}\epsilon^{3}+Y^{\frac{9N_n}{4}}(t) +\int_0^t \langle t-s \rangle^{-\frac{9}{8}}\left(Y(s)+\|B^{-1/2}f^{(l_0)}(s)\|_{W^{k-1,8}_x}^{\frac{5}{3}}\right)ds\\
		\lesssim& \langle t \rangle^{-\frac{9}{8}}\epsilon^{3}+Y(t)+\int_0^t \langle t-s \rangle^{-\frac{9}{8}}\|B^{-1/2}f^{(l_0)}(s)\|_{W^{k,8}_x}^{\frac{5}{3}}ds.
	\end{align*}
	Using a bootstrap argument again we have
	\begin{align*}
		\|B^{-1/2}f^{(l_0)}(t)\|_{W^{k-1,8}_x}\lesssim\langle t \rangle^{-\frac{9}{8}}\epsilon^{3}+Y(t).
	\end{align*}
	Repeating this procedure, we have for $k'\le k$ 
	\begin{align*}
		\|B^{-1/2}f^{(l_0)}(t)\|_{W^{k-k',8}_x}\lesssim\langle t \rangle^{-\frac{9}{8}}\epsilon^{3}+Y^{\frac{k'+1}{2}}(t).
	\end{align*}
	Thus, if we choose $k\ge \frac{9N_n}{2}$ and $k'=k-1$, we have
	\begin{equation*}
		\|B^{-1/2}f^{(l_0)}(t)\|_{W^{1,8}_x}\lesssim\langle t \rangle^{-\frac{9}{8}}\epsilon^{3}+Y^{\frac{9N_n}{4}}(t).
	\end{equation*}
	Combining the estimates of $f^{(l)}_M$ and $f^{(l_0)}$, we have
	\begin{cor}\label{cor:f}
		Assuming \eqref{eq:xi-esti-1}-\eqref{eq:xi-esti-4} holds for $0\le t\le T$, we have 
		\begin{equation*}
			\|B^{-1/2}f(t)\|_{W^{1,8}_x}\lesssim\langle t \rangle^{-\frac{9}{8}}\epsilon^{3}+Y^{\frac{1}{2}}W^{\frac{1}{2}}.
		\end{equation*}
	\end{cor}
	We now turn to the estimates of $f_R.$ Recall that
	\begin{equation*}
		f_{R}=\sum_{l=0}^{l_0-1}f^{(l)}_{M, R}+	f^{(l_0)},
	\end{equation*}
	we shall only estimate $f^{(l)}_{M, R}$. By Proposition \ref{prop:f-lM}, $f^{(l)}_{M, R}$ mainly consists of the following three parts:\\
	(i) $\int_{0}^{t}e^{-\mathrm{i}B (t-s)} \xi B^{-1/2}(\Psi  B^{-1/2}f^{(j)}_M B^{-1/2}f^{(l-1)}_M )ds, 0\le j\le k-1,$\\
	(ii) $\int_{0}^{t}e^{-\mathrm{i}B (t-s)}B^{-1/2}(B^{-1/2}f^{(i)}_M B^{-1/2}f^{(j)}_M B^{-1/2}f^{(l-1)}_M )ds, 0\le i, j\le k-1$\\
	(iii)terms coming from integration by parts, which takes the form 
	$$\bar{\xi}^{\mu} \xi^{\nu}\left[-\mathrm{i}\frac{\nu_{jk}}{\xi_{jk}}\partial_{\bar{\xi}_{jk}}\mathcal{R}+\mathrm{i}\frac{\mu_{jk}}{\bar{\xi}_{jk}}\partial_{\xi_{jk}}\mathcal{R}\right]R_{\nu\mu}^{+}\bar{\Phi}_{\mu \nu}$$
	or
	$$\bar{\xi}^{\mu}_0 \xi^{\nu}_0 e^{-\mathrm{i}B t}R_{\nu\mu}^{+}\bar{\Phi}_{\mu \nu}.$$
	Estimate of (i): for $0\le j\le l-1$, we have
	\begin{align*}
		&\left\|\int_{0}^{t}e^{-\mathrm{i}B (t-s)}\xi B^{-1}\left(\Psi  B^{-1/2}f^{(j)}_M B^{-1/2}f^{(l-1)}_M \right)ds\right\|_{W^{k,\frac{6}{1-6\delta_0}}_x}  \\
		\lesssim &\int_0^t \langle t-s \rangle^{-(1+3\delta_0)}\left\|\xi \Psi  B^{-1/2}f^{(j)}_M B^{-1/2}f^{(l-1)}_M \right\|_{W^{k+1,\frac{6}{5+6\delta_0}}_x}ds \\
		\lesssim &\int_0^t \langle t-s \rangle^{-(1+3\delta_0)}|\xi|\|B^{-1/2}f^{(j)}_M(s)\|_{W^{k+1,\frac{6}{6-6\delta_0}}_x}\|B^{-1/2}f^{(l-1)}_M(s)\|_{W^{k+1,\frac{6}{6-6\delta_0}}_x}ds\\
		\lesssim &\int_0^t \langle t-s \rangle^{-(1+3\delta_0)}YW ds\\
		\lesssim & \epsilon^2 W .
	\end{align*}
	Estimate of (ii): for $0\le i, j\le l-1,$ we have
	\begin{align*}
		&\left\|\int_{0}^{t}e^{-\mathrm{i}B (t-s)}B^{-1}\left(B^{-1/2}f^{(i)}_M B^{-1/2}f^{(j)}_M B^{-1/2}f^{(l-1)}_M \right)ds\right\|_{W^{k,\frac{6}{1-6\delta_0}}_x}\\
		\lesssim &\int_0^t \langle t-s \rangle^{-(1+3\delta_0)}\left\|B^{-1/2}f^{(i)}_M B^{-1/2}f^{(j)}_M B^{-1/2}f^{(l-1)}_M\right\|_{W^{k+1,\frac{6}{5+6\delta_0}}_x}\\
		\lesssim &\int_0^t \langle t-s \rangle^{-(1+3\delta_0)}\sum_{0\le j\le l-1}\|B^{-1/2}f^{(j)}_M(s)\|_{W^{k+1,2}_x}^{\frac{1+12\delta_0}{1+3\delta_0}}\sum_{0\le j\le l-1}\|B^{-1/2}f^{(j)}_M(s)\|_{W^{k+1,\frac{6}{6-6\delta_0}}_x}^{\frac{2-3\delta_0}{1+3\delta_0}}ds\\
		\lesssim &\int_0^t \langle t-s \rangle^{-(1+3\delta_0)} \big(YW\big)^{\frac{2-3\delta_0}{2+6\delta_0}}ds\\
		\lesssim & \epsilon^{2-\delta}W, 
	\end{align*}
	where in the last inequality we choose $\delta_0\le \frac{1}{10000N_n^2}$ and use the fact $Y^{2N_n}\lesssim W.$ \\
	Estimates of (iii): by Lemma \ref{lem-partial-xi-R}, we have
	\begin{align*}
		&\left\|\langle x \rangle^{-\sigma} \bar{\xi}^{\mu} \xi^{\nu}\left[-\mathrm{i}\frac{\nu_{jk}}{\xi_{jk}}\partial_{\bar{\xi}_{jk}}\mathcal{R}+\mathrm{i}\frac{\mu_{jk}}{\bar{\xi}_{jk}}\partial_{\xi_{jk}}\mathcal{R}\right]R_{\nu\mu}^{+}\bar{\Phi}_{\mu \nu}\right\|_{L^2}\\
		\lesssim&  \left\|\langle x \rangle^{-\sigma} R_{\nu\mu}^{+}\bar{\Phi}_{\mu \nu}\right\|_{L^2}\left|\partial_{\bar{\xi}_{jk}}\mathcal{R}\right|\\	
		\lesssim& \|B^{-1/2}f(t)\|_{W^{1,8}_x}^2\\
		\lesssim& \langle t \rangle^{-\frac{9}{4}}\epsilon^{6}+YW.
	\end{align*}
	By Lemma \ref{lem-singular-resolvents}, we have
	\begin{align*}
		\left\|\langle x \rangle^{-\sigma} \bar{\xi}^{\mu}_0 \xi^{\nu}_0 e^{-\mathrm{i}B t}R_{\nu\mu}^{+}\bar{\Phi}_{\mu \nu}\right\|_{L^2}\lesssim \langle t \rangle^{-\frac{6}{5}}\epsilon^3.
	\end{align*}
	Combining the estimates of $f^{(l)}_{M, R}$ and $f^{(l_0)}$, we get
	\begin{equation*}
		\left\|\langle x \rangle^{-\sigma}f_{R} \right\|_{L^2}	\lesssim \langle t \rangle^{-\frac{9}{8}}\epsilon^{3}+\epsilon^{2-\delta}W.
	\end{equation*}	
	
	\section{Proof of the Main Theorem}\label{sec-final}
	By Theorem \ref{thm:X} and Corollary \ref{cor:f}, we have
	$$|\xi_{jk}|\lesssim Y^{\frac{\alpha_j}{2}},$$ 
	$$|\xi|\approx Y^{\frac{1}{2}},$$
	and
	$$	\|B^{-1/2}f(t)\|_{W^{1,8}_x}\lesssim\langle t \rangle^{-\frac{9}{8}}\epsilon^{3}+Y^{\frac{1}{2}}W^{\frac{1}{2}}.$$ 
	We note that $(\xi, f)$ are the variables after the normal form transformation. For the original variables
	$(\xi',f') = \mathcal{T}_{100N_n}(\xi,f)$,
	by $
	\left\|z-\mathcal{T}_{r}(z)\right\|_{\mathcal{P}^{\kappa, s}} \lesssim \|z\|_{\mathcal{P}^{-\kappa,-s}}^{3}
	$ we have 
	$$|\xi'_{jk}|\lesssim Y^{\frac{\alpha_j}{2}},$$
	$$|\xi'|\approx Y^{\frac{1}{2}},$$ 
	and
	$$\|B^{-1/2}f'(t)\|_{W^{1,8}_x}\lesssim Y^{\frac{3}{2}}.$$
	Thus we complete the proof of Theorem \ref{thm:main}. 
	
	\section*{Acknowledgment}
	We would like to thank Professor H. Jia for bringing the problem to us and for useful discussions. The authors were in part supported by NSFC (Grant No. 11725102), Sino-German Center Mobility Programme (Project ID/GZ M-0548) and Shanghai Science and Technology Program (Project No. 21JC1400600 and No. 19JC1420101).

	\frenchspacing
	\bibliographystyle{plain}
	
\end{document}